\documentclass{article}

\usepackage[margin=1.1in]{geometry}
\usepackage[utf8]{inputenc} % allow utf-8 input
\usepackage[T1]{fontenc}    % use 8-bit T1 fonts
\usepackage[backref=page]{hyperref}% hyperlinks
\usepackage{url}            % simple URL typesetting
\usepackage{booktabs}       % professional-quality tables
\usepackage{amsfonts}       % blackboard math symbols
\usepackage{nicefrac}       % compact symbols for 1/2, etc.
\usepackage{microtype}      % microtypography
\usepackage{tikz}

\usepackage[inline]{enumitem}   

\usepackage[nottoc,numbib]{tocbibind}

\usetikzlibrary{shapes,arrows,positioning,arrows.meta}

\tikzstyle{block} = [rectangle, draw,
    text width=16em, text centered, rounded corners, minimum height=4em]
\tikzstyle{block2} = [rectangle, draw,
    text width=8em, text centered, rounded corners, minimum height=2.5em]
\tikzstyle{block3} = [rectangle,
    text width=8em, text centered, minimum height=4em]
\tikzstyle{line} = [draw, -implies, double distance=3pt]
\tikzstyle{dashedline} = [draw, dashed]
\tikzstyle{doubleimplies} = [draw, implies-implies, double distance=3pt]
\tikzstyle{arrow} = [draw, -{Latex[length=3mm]}]
\tikzstyle{doublearrow} = [draw, {Latex[length=3mm]}-{Latex[length=3mm]}]
\tikzstyle{noblock} = [rectangle, text width=5em, text centered,
    minimum height=4em]
\tikzstyle{widenoblock} = [rectangle, text width=9em, text centered,
    minimum height=2em]
\tikzstyle{wideblock} = [rectangle, draw,
    text width=20em, text centered, rounded corners, minimum height=6em]
\tikzstyle{widesmallblock} = [rectangle, draw,
    text width=15em, text centered, rounded corners, minimum height=1.5em]

\usepackage{amsmath,amsfonts,amsthm,amssymb}
\usepackage[]{xcolor}
\usepackage[numbers]{natbib}
\usepackage[normalem]{ulem}

\newcommand{\Pval}{\mathfrak p}
\newcommand{\Eval}{\mathfrak e}

\newcommand{\RR}{\mathbb{R}}
\newcommand{\NN}{\mathbb{N}}
\newcommand{\EE}{\mathtt{E}}

\newcommand{\Pbb}{\mathcal{P}}

\newcommand{\Qcalu}{\bar{\mathcal{Q}}}

\newcommand{\Qcall}{\underline{\mathcal{Q}}}

\newcommand{\Ptt}{\mathtt{P}}
\newcommand{\Qtt}{\mathtt{Q}}
\newcommand{\Rtt}{\mathtt{R}}

\newcommand{\Utt}{\mathtt{U}}
\newcommand{\Htt}{\mathtt{H}}

\newcommand{\CI}{\mathfrak C}

\newcommand{\1}{\mathbf{1}}
\newcommand{\dd}{\mathrm{d}}

\newcommand{\Ncal}{\mathcal{N}}
\newcommand{\Fcal}{\mathcal{F}}
\newcommand{\Acal}{\mathcal{A}}

\newcommand{\Pcal}{\mathcal{P}}
\newcommand{\Hcal}{\mathcal{H}}
\newcommand{\Scal}{\mathcal{S}}
\newcommand{\Qcal}{\mathcal{Q}}
\newcommand{\Zcal}{\mathcal{Z}}
\newcommand{\Gcal}{\mathcal{G}}

\DeclareMathOperator*{\esssup}{ess\,sup}
\DeclareMathOperator*{\essinf}{ess\,inf}

\newtheorem{theorem}{Theorem}
\newtheorem{lemma}[theorem]{Lemma}
\newtheorem{proposition}[theorem]{Proposition}
\newtheorem{example}[theorem]{Example}
\newtheorem{remark}[theorem]{Remark}

\newtheorem{corollary}[theorem]{Corollary}

\usepackage{graphicx} 
\usepackage{caption}
\usepackage{subcaption}
\usepackage{enumitem}
\usepackage{comment}
\usepackage{setspace}
\newcommand{\com}[1]{\marginpar{{\begin{minipage}{0.15\textwidth}{\setstretch{1.1} \begin{flushleft} \footnotesize \color{red}{#1} \end{flushleft} }\end{minipage}}}}

\hypersetup{
     colorlinks   = true,
     citecolor    = gray
}
\hypersetup{linkcolor=black}

\title{
Admissible anytime-valid sequential inference\\
must rely on nonnegative martingales\\
\bigskip
}

\author{
Aaditya Ramdas$^{1}$, Johannes Ruf$^{2}$, Martin Larsson$^{3}$, Wouter M. Koolen$^{4}$\\
\and 
$^{1}$ Departments of Statistics and Machine Learning, Carnegie Mellon University \\
$^{2}$ Department of Mathematics, London School of Economics\\
$^{3}$ Department of Mathematics, Carnegie Mellon University \\
$^{4}$ Machine Learning Group, CWI Amsterdam \\
\and
\texttt{aramdas@cmu.edu, j.ruf@lse.ac.uk}\\ \texttt{martinl@andrew.cmu.edu, wmkoolen@cwi.nl} 
}

\begin{document}

\maketitle

\begin{abstract}
Confidence sequences, anytime p-values (called p-processes in this paper), and e-processes all enable sequential inference for composite and nonparametric classes of distributions at arbitrary stopping times. Examining the literature, one finds that at the heart of all these (quite different) approaches has been the identification of nonnegative (super)martingales. Thus, informally, nonnegative (super)martingales are known to be sufficient for \emph{anytime-valid} sequential inference, even in composite and nonparametric settings. Our central contribution is to show that nonnegative martingales are also universal---after appropriately defining \emph{admissibility}, we show that all admissible constructions of confidence sequences, p-processes, or e-processes must necessarily utilize nonnegative martingales. 
% Sufficient conditions for composite admissibility are also provided. 
Our proofs utilize several modern mathematical tools for composite testing and estimation problems: max-martingales, Snell envelopes, transfinite induction, and new Doob-L\'evy martingales make appearances in previously unencountered ways. Informally, if one wishes to perform anytime-valid sequential inference, then any existing approach can be recovered or dominated using nonnegative martingales. We provide several nontrivial examples, with special focus on testing symmetry, where our new constructions render past methods inadmissible. We also prove the subGaussian supermartingale to be admissible.
\end{abstract}


 \tableofcontents

\newpage

\hypersetup{linkcolor=blue}

\section{Introduction}
\label{sec:introduction}

% Our abstract mathematical treatment will be  helped by a hypothetical example. 
Consider a laboratory that wishes to understand if a particular intervention (`treatment') has any positive effect on a prespecified outcome of interest. Without getting bogged down by details, suppose the `average treatment effect' of the intervention over the relevant population is denoted by $\theta$. Suppose the lab wants to test $H_0: \theta \leq 0$ against $H_1: \theta > 0$, or to estimate $\theta$ using a confidence interval.  The lab believes that there is an effect, but has no idea how many subjects to collect data from: a larger sample size means more power, but also more time and money. So they conduct their experiment sequentially: subjects enter the study one at a time and are assigned to treatment or control completely at random; denote the data from subject $t$ as $X_t$. 

After observing $X_t$, they analyze the data $X_1,\dots,X_t$ they have so far, and then decide if they wish to collect more data, or whether what they already have suffices to demonstrate an effect (to themselves, or to a journal, or to the world). The lab stops their experiment at some data-dependent stopping time $\tau$---maybe time ran out, or they used the money up faster than expected; maybe the effect was sufficiently large, or perhaps they lowered their aims by being satisfied with a smaller effect, or became more optimistic and kept the experiment running longer in the hope for a narrower confidence interval around a (hopefully) large effect. In other words, the stopping criterion used may itself have changed over time with funding coming in or drying up, with initial results being more/less promising than anticipated. 
In any case, the experiment was stopped at time $\tau$ and not earlier or later, and there could be multiple data-dependent reasons for stopping at $\tau$ that were hard to anticipate in advance.

As a measure of evidence, they may hope to construct a ``p-process'', which is a sequence of p-values $(\Pval_t)_{t\in\NN}$ that satisfies
\begin{equation}\label{eq:stopped-pvalue}
\text{for any arbitrary stopping time $\tau$, } \Pr_{H_0} (\Pval_\tau \leq a) \leq a, \text{ for all } a \in [0,1].
\end{equation}
\citet{johari_always_2015} are the major modern proponents of this idea, and call this an ``anytime-valid p-value'', but for consistency with e-processes introduced later, we call it a p-process.
% (Replacing $\Pval_\tau \leq a$ with $\Pval_\tau < a$ above does not change anything, since it must be true at all $a$.)
Unfortunately, naively using a $t$-test, a chi-squared test, or permutations, does not yield the above property. Indeed, `standard' non-sequential p-values only satisfy the weaker property
\[
\text{for any data-independent time $t$, } \Pr_{H_0} (\Pval_t \leq a) \leq a, \text{ for all } a \in [0,1].
\]
% The most straightforward way to construct `p-processes' that satisfy property~\eqref{eq:stopped-pvalue} is to use mixture likelihood ratios.
% at least for point nulls, is to employ bonafide sequential tests like Wald's sequential likelihood ratio test~\cite{wald_sequential_1945,wald_sequential_1947}, 
% but extensions to nonparametric settings have also been explored recently~\cite{howard_exponential_2018,wasserman2020universal}. 
% These are also called sequentially-adjusted p-values~\cite{dupont1983sequential}, or always-valid p-values~\cite{johari_always_2015}.

Instead of testing, if the lab wished to estimate $\theta$ with an error tolerance $\alpha \in [0,1]$, they may hope to construct a ``confidence sequence'', which is sequence of confidence sets $(\CI_t(\alpha))_{t\in\NN}$ that satisfies
\begin{equation}\label{eq:stopped-CI}
\text{\quad for any arbitrary stopping time $\tau$, } \Pr (\theta \in \CI_\tau(\alpha)) > 1-\alpha.
\end{equation}
One could have used $\geq$ instead of $>$ above; the two definitions are often equivalent/interchangeable modulo minor adjustments like swapping open for closed sets.
Confidence sequences were proposed by Robbins and collaborators like Darling, Siegmund, and Lai~\cite{darling_confidence_1967,robbins_expected_1974,lai_confidence_1976}, and have been regaining interest in recent years~\cite{pace2019likelihood,howard_uniform_2019}. They were originally defined slightly differently, but as we point out later, the above is an equivalent definition that gels better with definitions of other objects in this paper.

Once more, unfortunately, a naive confidence interval based on the central limit theorem or the bootstrap does not satisfy the desired property. `Standard' confidence sets instead satisfy the above property only at fixed \emph{data-independent} times $t$. 
We provide an example of~\eqref{eq:stopped-CI} for the reader to have a concrete instance in mind. Suppose $(X_s)_{s\geq 1}$ are i.i.d.~$N(\theta,1)$, then it can be shown~\cite{howard_uniform_2019} that
\begin{equation}\label{eq:normal-mixture-CS}
\text{for any stopping time $\tau$, } \Pr\left( \theta \in \left[ \frac{\sum_{s \leq \tau} X_s}{\tau} \pm \sqrt{\frac{(1+1/\tau) \log ((\tau+1)/\alpha^2)}{\tau}}\right]\right) > 1-\alpha.
\end{equation}
At a fixed time $t$, with  $z_{q}$ denoting the $q$-quantile of a standard Gaussian, we could have used a confidence interval width of ${z_{1-\alpha/2}}/{\sqrt{t}} \leq \sqrt{\log(2/(\pi\alpha^2))/t}$ by the Gaussian tail inequality, which is reasonably tight for small $\alpha$. 
% To approximate $z_{1-\alpha/2}$, note that the Gaussian tail inequality yields that for $t\geq 1$, we have $\Pr(Z > t) \leq (\sqrt{2\pi})^{-1/2}\exp(-t^2/2)$. Setting the right hand side to $\alpha/2$, we get that $z_{1-\alpha/2} \leq \sqrt{\log(2/(\pi\alpha^2))}$, which is known to be reasonably tight for small $\alpha$.
Thus the above confidence sequence is looser by a factor of $\approx \sqrt{\log \tau}$. The above inequality is proved by applying Ville's inequality (see \eqref{eq:Ville}) to an exponential Gaussian-mixture martingale. These are fundamental tools we encounter later in this paper so we do not elaborate on them further here. 

% Recently, another set of highly interrelated ideas has been put forward under a variety of names by authors such as Shafer, Vovk, Gr\"unwald, and their collaborators: test martingales~\cite{shafer_test_2011}, e-processes or sequential e-processes~\cite{vovk2020nonalgorithmic,vov_wan_19} ($e$ for expectation), betting scores~\cite{shafer2019language} (since they have roots in gambling), or e-processes~\cite{grunwald_safe_2019} (safe under optional stopping). Even though these concepts often have origins in parametric settings, the ideas have been extended to complicated nonparametric settings involving composite irregular models~\cite{wasserman2020universal,howard_exponential_2018}.

Recently, the notion of an ``e-process'' has gained particular prominence as a central object in a so-called game-theoretic approach to statistical testing. 
% we denote it by $\Eval$ to reinforce the fact that it is an e-process. 
% Importantly, our use of the term `safe' does not in any way imply that the other two concepts (anytime p-values and confidence sequences) are unsafe; 
% indeed they are also safe against optional stopping and continuation of experiments. 
An e-process for $H_0$ is a nonnegative sequence $(\Eval_t)_{t\in\NN}$ that satisfies
\begin{equation}\label{eq:stopped-safe}
\text{\quad for any arbitrary stopping time $\tau$, }  \EE_{H_0}[\Eval_\tau] \leq 1.
\end{equation}
Special cases, that we will encounter in this paper, include test martingales~\cite{shafer_test_2011,vovk_testing_randomness_2019} and test supermartingales~\cite{shafer2019language,grunwald_safe_2019,howard_exponential_2018}, but it has recently been established that e-processes are more general than these two objects~\cite{ramdas2021testing}. For example, universal inference~\cite[Section 8]{wasserman2020universal} yields an e-process that is not a test (super)martingale. As before, a standard e-value~\cite{vovk2019values} is not anytime-valid and exhibits the above property only at fixed times $t$, which does not suffice for sequential settings. 

Of course, p-processes, e-processes and confidence sequences are more general than their nonsequential counterparts (p-values, e-values, confidence intervals): whenever you have the former, you in particular also have the latter, since fixed times are also stopping times but not vice versa.
All three tools can be used to make decisions of accepting or rejecting a null hypothesis $H_0$. A level-$\alpha$ sequential test is a decision rule that maps the data (and $\alpha$) onto $\{0,1\}$, stopping when it first outputs one (rejection of the null). Concretely, a level-$\alpha$ sequential test is a binary sequence $(\psi_t)_{t\in\NN}$ that satisfies
\begin{equation}\label{eq:stopped-test}
\text{\quad for any arbitrary stopping time $\tau$, }  \Pr_{H_0}(\psi_\tau = 1) \leq \alpha.
\end{equation}
 In contrast to Wald's sequential tests (which specified a single stopping rule for both type-I and type-II error control), the above notion of a sequential test is more closely related to Robbins' ``level-$\alpha$ tests of power one''. Once more, `standard' hypothesis tests only satisfy such a type-I error guarantee at fixed times $t$. As we shall see later in more detail, p-processes, e-processes, or confidence sequences can each be used to derive a level-$\alpha$ sequential test.

\bigskip
Our main interest in this work is to develop a general admissibility theory for estimating nonparametric functionals and testing composite nonparametric null hypotheses. 
We formally define all four of the above concepts in full generality in Section~\ref{sec:the-4-tools}, but the above semi-formal description suffices for the moment to describe our main findings below. 

One common theme amongst all the aforementioned works over the decades is the repeated appearance of various, often sophisticated, \emph{nonnegative supermartingales} as the central object that enables the construction of all four aforementioned tools of anytime-valid inference. In the rest of this paper, we further examine this central role of nonnegative martingales in constructing p-processes, confidence sequences, e-processes and sequential tests with the desired robustness to continuous monitoring, and arbitrary optional stopping (and continuation) of experiments. Specifically, we show that \emph{all admissible constructions of these objects must employ nonnegative martingales (explicitly or implicitly).}

\bigskip

\textbf{Admissibility.}
Naturally, the desire for methods satisfying properties like \eqref{eq:stopped-pvalue}, \eqref{eq:stopped-CI},  \eqref{eq:stopped-safe}, or \eqref{eq:stopped-test} comes with an implicit wish for efficiency. In other words, setting $\Pval_\tau=1$,  $\Eval_\tau=1$, $\psi_\tau = 0$, or $\CI_\tau = \RR$ (or $\CI_\tau=\Theta$ for a more general parameter space) trivially satisfies the above requirements, but is clearly uninformative. We want $\Pval_\tau$ to be as small as possible, $\Eval_\tau$ and $\psi_\tau$ to be as large as possible, and $\CI_\tau$ to be as narrow as possible, while still being statistically valid measures of uncertainty. We use the term  `dominates' to compare pairs of these objects (in order to avoid using case-by-case adjectives like small/large/narrow)---so, if $\Pval' \leq \Pval$ then $\Pval'$ dominates $\Pval$. Similarly,   if $\Eval' \geq \Eval$ then $\Eval'$ dominates $\Eval$, if $\psi' \geq \psi$ then $\psi'$ dominates $\psi$, and if $\CI' \subseteq \CI$ then $\CI'$ dominates $\CI$. 
In this paper, we use the notion of admissibility to capture this idea: informally, a p-process (or e-process, test, confidence set) is inadmissible if it is strictly dominated by another valid p-process (or  e-process, test, confidence set). We define admissibility formally in Section~\ref{sec:the-4-tools}.

\bigskip

\textbf{Paper outline.} Sections~\ref{sec:WaldRobbins} and~\ref{sec:Ville} lay out the formal definitions of several of the basic tools---nonnegative (super)martingales, max-martingales, Doob's optional stopping theorem, and Ville's inequality.  Section~\ref{sec:the-4-tools} introduces the four central tools of anytime-valid sequential inference: p-processes, e-processes, sequential tests, and Robbins' confidence sequences. Section~\ref{sec:two-examples} provides two simple examples: Gaussian and symmetric (super)martingales. Then Section~\ref{sec:admissible-pointwise} and Section~\ref{sec:composite-via-pointwise} summarize this paper's central message about the centrality of nonnegative martingales in constructing the aforementioned tools. Section~\ref{sec:admissible-pointwise} formalizes the necessary and sufficient conditions for admissibility in the point null setting; it uses a Doob-L\'evy max-martingale construction to show the necessary conditions for p-processes, a Doob-Le\'vy martingale for sequential tests, and uses the Doob decomposition of an appropriate Snell envelope to  prove that admissible e-processes must also (explicitly or implicitly) employ nonnegative martingales. Section~\ref{sec:composite-via-pointwise}
develops several novel reductions of admissibility in the composite null setting to the point null case, and presents extensions to estimation (confidence sequences).
Section~\ref{sec:sufficient} presents deeper investigations on admissibility, including anti-concentration bounds and the role of randomization.
Section~\ref{sec:subG} utilizes the learnt lessons to show that Robbins' subGaussian supermartingale is admissible. 
Section~\ref{sec:examples} shows that past constructions for testing symmetry were inadmissible, and to produce improved admissible tests. 
Appendix~\ref{app:A} recaps certain technical concepts like local domination and essential suprema. Appendix~\ref{sec:proofs} details all proofs that are not in the main paper. Finally, Appendix~\ref{sec:more-examples} contains examples and counterexamples to support several claims made in the paper.

\section{Preliminaries}\label{sec:WaldRobbins}

\subsection{A time-uniform equivalence lemma}

The following lemma is quite central to the construction and interpretation of p-values, confidence sets, and tests that are valid at arbitrary stopping times.

\begin{lemma}[Equivalence lemma]\label{lem:equiv_uniform_defns}
  Let $(A_t)_{t \in \NN}$ be an adapted sequence of events in some filtered
  probability space and let $A_\infty := \limsup_{t \to \infty} A_t := \bigcap_{t \in \NN} \bigcup_{s \geq t} A_s$. The
  following statements are equivalent:
  \begin{enumerate}[label={\rm(\roman{*})}, ref={\rm(\roman{*})}] 
  \item\label{L1:1} $\Pr(\bigcup_{t \in \NN} A_t) \leq \alpha.$
  \item\label{L1:2} $\smash{\sup_T \Pr(A_T) \leq \alpha}$, where $T$ ranges over random times that are possibly infinite, and are not necessarily stopping times.
  \item\label{L1:3} $\sup_\tau \Pr(A_\tau) \leq \alpha$, where $\tau$ ranges over stopping times, possibly infinite.
  \end{enumerate}
\end{lemma}

The proof can be found in Appendix~\ref{sec:proofs}. If the event $A_t$ is associated with making an erroneous claim at time $t$, we interpret the aforementioned three statements as follows:
\begin{enumerate}[label={\rm(\roman{*})}, ref={\rm(\roman{*})}] 
    \item\label{L1:1'} The probability of ever making an erroneous claim, over infinite time, is at most $\alpha$. 
    \item\label{L1:2'} The probability of making an erroneous claim at an arbitrary data-dependent time $T$, perhaps chosen post-hoc as a past time after an experiment is stopped, is at most $\alpha$.
    \item\label{L1:3'} When we stop an experiment at an arbitrary stopping time $\tau$, the probability of making an erroneous claim at that time is at most $\alpha$.
\end{enumerate}
Intuitively, it is clear that \ref{L1:1'} implies \ref{L1:2'}, which in turn implies \ref{L1:3'}, but the aforementioned lemma establishes that all three properties are actually equivalent: if you want one of them, you get all of them for free. This lemma gives the first hint of the centrality of martingales---the third statement is very directly about optional stopping, even though this fact is somewhat masked in the first two ways of framing the desired error control. While \ref{L1:3'} enables inferences at stopping times as initially motivated, property \ref{L1:2'} allows further introspection at previous times, enabling statistically valid answers to questions like `what was the estimate of the treatment effect at time $\tau/2$?' (where $\tau$ was the stopping time).

The above lemma first appeared in \citet{howard_uniform_2019}. While Lemma~\ref{lem:equiv_uniform_defns} did not motivate its original definition in 1967, \citet{darling_confidence_1967}  first defined a `confidence sequence' for a parameter $\theta$ as an infinite sequence of confidence sets $(\CI_t)_{t \in\NN}$  such that
\begin{equation*}
  \Pr(\exists t \in \NN: \theta \notin \CI_t) \leq \alpha.
\end{equation*}
In other words the aforementioned confidence sets satisfy property \ref{L1:1'} for $A_t := \1_{\theta \notin \CI_t}$. 

There has been much recent effort to constructing p-processes for testing 
and confidence sequences for estimation. 
Underlying the construction of these objects in a variety of works, one often finds the repeated use of Ville's inequality for nonnegative supermartingales; see \citet{howard_exponential_2018} for a thorough survey. We will show that this is not a coincidence: we prove that nonnegative martingales underlie all admissible constructions for performing anytime-valid sequential inference.   

To make these claims more formal, and especially to handle composite hypothesis testing, we need to clarify what the probability $\Pr$ means in the above definitions, and we do so next.

\subsection{Notation and conventions}
\label{sec:conventions}
Let $\NN$ represent the natural numbers and $\NN_0 = \NN \cup \{0\}$. 
We use $(B_t)$ to denote a sequence $(B_t)_{t\in \NN}$ or $(B_t)_{t\in \NN_0}$ where the indexing of $t$ is either implicitly understood from the context or unimportant, but we use $B_t$ without the brackets to denote a particular element from the sequence. Thus, for example, $\Fcal_t$ will denote a sigma-field at time $t$ but $(\Fcal_t)$ denotes a filtration, which is an increasing sequence of sigma-fields. Unless otherwise mentioned, $\Fcal_0=\sigma(U)$, where $U$ is a $[0,1]$-uniform random variable that is independent of everything else, signifying an external source of randomness, and $\Fcal_t := \sigma(U, X_1,\dots,X_t)$ will denote the canonical filtration, where $X_t$ is the data observed at time $t$. We allow $X_t$ to take values in some general space, which we do not need to specify here, e.g., in $\mathbb R^d$, equipped with the Borel sigma algebra.

Earlier, we used $\Pr$ to represent the probability taken over all sources of randomness, but in what follows we will use a more explicit notation:
% First, $P$ will denote the distribution of a single outcome if the data are i.i.d.; otherwise $P_t$ denotes the conditional distribution of $X_t | \Fcal_{t-1}$. 
% In the i.i.d. setting, $P^t = P \times P \times \dots \times P$ denotes the probability distribution over the first $t$ datapoints, while $P^\infty$ denotes the distribution of the entire sequence. 
% In the general (potentially non-i.i.d.) setting, w
we denote the {distribution of an infinite sequence of observations} by $\Ptt$; this means that $\Ptt$ is a probability measure on $\Fcal_\infty:=\sigma(U,(X_t)_{t \in \NN}) = \sigma(\bigcup_{t \in \NN_0} \Fcal_t)$. Expectations with respect to $\Ptt$ are denoted $\EE_{\Ptt}$.
% A set of distributions over single observations will be denoted $\Pcal$, while 
A set consisting of distributions over sequences will be denoted $\Pbb$; so $\Pbb=\{\Ptt\}$ is the singleton case, but more generally there may be uncountably many $\Ptt \in \Pbb$. In the case of testing, we denote the null set of distributions by $\Qcal \subset \Pcal$.

Next, $\tau$ will always denote a stopping time, while $t$  denotes a fixed time. 
A subscript $t$ for $\Pval_t$, $\Eval_t$,  $\psi_t$, and $\CI_t$ means that these objects were constructed using only the data available up to time $t$. In other words, $\Pval_t$, $\Eval_t$, $\psi_t$, and $\CI_t$ are $\Fcal_t$-measurable, or the sequences
$(\Pval_t)$, $(\Eval_t)$, $(\psi_t)$, and $(\CI_t)$ are adapted to $(\Fcal_t)$. It is also understood that  $\Pval_t$ has range $[0,1]$, $\Eval_t$ has range $[0,\infty]$ and a sequential test $\psi_t$ has range $\{0,1\}$; the range of the confidence set $\CI_t$ will be formally specified later.

If $\Ptt$ is a probability measure on $\Fcal_\infty$ and $\tau$ is a stopping time, we write $\Ptt_\tau$ for the restriction of $\Ptt$ to $\Fcal_\tau$. This is simply the probability measure on $\Fcal_\tau$ defined by $\Ptt_\tau(A)=\Ptt(A)$, for $A\in\Fcal_\tau$. (This is the `coarsening' of $\Ptt$ that only operates on events observable up to time $\tau$.)

We sometimes, but not always, assume that $\Pcal$ is `locally dominated' by (i.e., locally absolutely continuous with respect to) a fixed reference measure $\Rtt$; we review the meaning of this in Appendix~\ref{sec:ref_measure}. For example, if each observation $X_s$ has a Lebesgue density under all $\Ptt \in \Pcal$, one can choose the reference measure to be the distribution of an i.i.d.\ sequence of standard Gaussians. Existence of a reference measure is needed to unambiguously interpret conditional expectations like $\EE_{\Ptt}[Y\mid\Fcal_t]$ under measures different from $\Ptt$, since \emph{a priori}, such conditional expectations are only defined up to $\Ptt$-nullsets. For completeness, we elaborate on this in Appendix~\ref{sec:ref_measure}, but this issue will not actually be visible in the proofs of our results.

{\color{black}
We sometimes adopt the following shorthand:
\[\Qcalu(A) = \sup_{\Qtt\in\Qcal} \Qtt(A), \quad  \Qcall(A) = \inf_{\Qtt\in\Qcal} \Qtt(A),
\]
which are supposed to remind the reader of `upper' and `lower' probabilities.
A set $A$ is called $\Qtt$-polar if $\Qtt(A)=0$, meaning it is a measure zero set under $\Qtt$, and it is called $\Qcal$-polar if $\Qcalu(A)=0$, meaning it is measure zero under every $\Qtt\in\Qcal$.
}

% If $\Ptt$ is a probability measure on $\Fcal_\infty$ and $\tau$ is a stopping time, we write $\Ptt_\tau$ for the restriction of $\Ptt$ to $\Fcal_\tau$. (This is simply the probability measure on $\Fcal_\tau$ defined by $\Ptt_\tau(A)=\Ptt(A)$, $A\in\Fcal_\tau$. Think of this as the `coarsening' of $\Ptt$ that only operates on events observable up to time $\tau$.)
% For example, for an arbitrary index set $V$, we may use $\Pcal_V:=\{P_v\}_{v \in V}$ and $\Pbb_V:=\{\Ptt_v\}_{v \in V}$. 
% If another variable is needed, the same notation holds for $Q\in \Qcal$ and $\mathtt Q \in \mathbb Q$. Similarly, alternative confidence sets are denoted $\mathfrak D_t$, alternative p-values may be $\mathfrak q_t$, and e-processes may be $\mathfrak f_t$. 

\subsection{A composite time-uniform equivalence lemma}

One of this paper's central contributions is to characterize admissibility in composite settings.
With that motivation, we present the following extension of Lemma~\ref{lem:equiv_uniform_defns} using the notation introduced above.
\begin{lemma}[Composite equivalence lemma]\label{lem:composite_equiv_uniform_defns}
Let $\Qcal$ be a family of probability measures. 
  Let $(A_t)_{t \in \NN}$ be an adapted sequence of events in some filtered
  probability space and let $A_\infty := \limsup_{t \to \infty} A_t := \bigcap_{t \in \NN} \bigcup_{s \geq t} A_s$. The
  following statements are equivalent:
    \begin{enumerate}[label={\rm(\roman{*})}, ref={\rm(\roman{*})}] 
  \item\label{L2:1} $\sup_{\Qtt\in\Qcal} \Qtt(\bigcup_{t \in \NN} A_t) \leq \alpha.$
  \item\label{L2:2} $\smash{\sup_T \sup_{\Qtt\in\Qcal}\Qtt(A_T) \leq \alpha}$, where $T$ ranges over all random times, possibly infinite.
  \item\label{L2:3} $\sup_\tau \sup_{\Qtt\in\Qcal} \ \Qtt(A_\tau) \leq \alpha$, where $\tau$ ranges over all stopping times, possibly infinite.
  \end{enumerate}
  Further, if equality holds for any one, then it holds for the other two. 
\end{lemma}

The proof is exactly the same as in Lemma~\ref{lem:equiv_uniform_defns} and is thus omitted.
For contrast, we now state a version with expectations instead of probabilities in which the corresponding statements are not equivalent. Such a non-equivalence points to forthcoming differences between p-processes and e-processes. The following result resembles Lemma~\ref{lem:equiv_uniform_defns} but a composite version resembling Lemma~\ref{lem:composite_equiv_uniform_defns} can also be  easily stated.

\begin{lemma}[Non-equivalence lemma]\label{lem:non-equiv}
  Let $(N_t)_{t\in \NN_0}$ be an adapted sequence of nonnegative integrable random variables in a filtered
  probability space; let \smash{$N_\infty := \limsup_{t \to \infty}N_t$}. 
  Consider the following statements:
  \begin{enumerate}[label={\rm(\roman{*})}, ref={\rm(\roman{*})}] 
  \item\label{L12:1} $\EE[\sup_{t \in \NN_0} N_t] \leq 1$.
  \item\label{L12:2} $\EE[N_T] \leq 1$ for all random times $T$, possibly infinite and
    not necessarily stopping times.
  \item\label{L12:3} $\EE[N_\tau] \leq 1$ for all stopping times $\tau$, possibly infinite.
  \item\label{L12:4} $\EE[g(1) \vee \sup_{t \in \NN_0}  g(N_t)] \leq 1$ for any nondecreasing function $g$ such that $\smash{\int_1^\infty {g(y)}/{y^2}  \textnormal{d} y = 1}$; in particular, $\smash{\EE[1 \vee \sup_{t \in \NN_0} \sqrt{N_t}] \leq 2}$. 
  \end{enumerate}
 Then \ref{L12:1} and \ref{L12:2} are equivalent.
  Both \ref{L12:1} and \ref{L12:2} imply \ref{L12:3}, which in turn implies \ref{L12:4}. 
\end{lemma}

The proof can be found in Appendix~\ref{sec:proofs}. Contrasting the above lemma with Lemma~\ref{lem:equiv_uniform_defns} brings out some of the differences between p-processes and e-processes. To dig deeper at the difference, note that one could have equivalently written Lemma~\ref{lem:equiv_uniform_defns} in terms of Bernoulli random variables $B_t := \1_{A_t}$, in which case the formulae above involving $\Qtt(\dots)$ would be replaced by \ref{L1:1} $\EE_\Qtt[\sup_{t \in \NN_0} B_t]$, \ref{L1:2} $\EE_\Qtt [B_T]$, and \ref{L1:3} $\EE_{\Qtt} [B_\tau]$, respectively. Thus, for these specific nonnegative binary random variables $(B_t)$, the relevant statements are all equivalent, but more generally they are not. This difference later manifests itself in the inability to take running suprema for e-processes, and overall a rather different underlying structure.

\section{Martingales, max-martingales,
  and Ville's inequality}\label{sec:Ville}

A martingale is a stochastic process adapted to an underlying filtration, whose value at any time is the conditional expectation of its value at any later time. This is however not the only possible notion of martingale; another interesting notion is obtained by replacing conditional expectations by so-called conditional suprema, leading to \emph{max-martingales}. Both notions play an important role in this paper. In particular, max-martingales turn out to be particularly suitable for dealing with p-processes. We briefly review their definitions and basic properties next.

\subsection{Martingales (based on conditional expectation)}

Given a filtration $(\Fcal_t)$ and a measure $\Ptt$ on  $\Fcal_\infty$, a process $(M_t)_{t\in \NN_0}$ is called a martingale (with respect to $(\Fcal_t)$) if $M_t$ is $\Fcal_t$-measurable, $\Ptt$-integrable, and
\[
\EE_{\Ptt}[M_t | \Fcal_s] = M_s \text{ for any $t$ and $s \leq t$}.
\]
(Sub- and supermartingales are defined by relaxing the martingale property and allowing for inequality, $\ge$ respectively, $\le$.) Since we had earlier mentioned that $\Fcal_0$ includes an initial source of independent randomness, $M_0$ is itself allowed to be random. Naturally, we have $\EE_\Ptt [M_t] = \EE_\Ptt [M_0]$. Often, in this paper, $(M_t)$ will be nonnegative and the latter quantity equals one and so when we say `a nonnegative martingale with initial value one', we implicitly mean with initial \emph{expected} value one.

Given an $\Fcal_\infty$-measurable integrable random variable $Y$, the process $M_t := \EE_\Ptt[Y | \Fcal_t]$ is known as the Doob (or Doob--L\'evy) martingale associated with $Y$. The fact that this is a martingale follows from the tower rule of the conditional expectation: $\EE_{\Ptt}[M_t | \Fcal_s] = \EE_{\Ptt}[\EE_{\Ptt}[Y | \Fcal_t] ~| \Fcal_s] = M_s$ if $s \le t$.

We now generalize these definitions to hold for an entire set of measures.
Given a set $\Pbb$ of measures on $\Fcal_\infty$, a process $(M_t)_{t\in \NN_0}$ is called a nonnegative $\Pbb$-supermartingale ($\Pbb$-NSM) (with respect to $(\Fcal_t)$) if $M_t$ is nonnegative, $\Fcal_t$-measurable, and 
\begin{equation}\label{eq:P-NSM}
\EE_\Ptt [M_t | \Fcal_s] \leq M_{s} \text{ for all $t\in\NN$, $s\leq t$ and every $\Ptt \in \Pbb$}.
\end{equation}
If \eqref{eq:P-NSM} holds with equality, $(M_t)$ is called a nonnegative $\Pbb$-martingale ($\Pbb$-NM). If $\Pbb=\{\Ptt\}$ is a singleton, we write `$\Ptt$-NSM' instead of `$\{\Ptt\}$-NSM'. This notational choice is also applied to other objects.  
We refer to a $\Pcal$-NM (or NSM) as a `composite' NM (or NSM), while a $\Ptt$-NM is called a `pointwise' NM (or NSM).

Doob's optional stopping theorem \cite{durrett_probability:_2017} extends the (sub-/super-) martingale property from deterministic times to stopping times. In general, only bounded stopping times are allowed in the optional stopping theorem; however, the nonnegativity of an NSM relieves us of this restriction. In particular, if $(N_t)$ is (upper bounded by) a $\Pcal$-NSM starting in $N_0$, the optional stopping theorem implies that
\begin{equation*}
\EE_{\Ptt}[N_\tau] \leq \EE_\Ptt[N_0] \text{ for all stopping times $\tau$, potentially infinite, and every $\Ptt \in \Pbb$.} 
\end{equation*}
In fact, if $(M_t)$ is a $\Pcal$-NSM, then we additionally have $\EE_{\Ptt}[M_\tau | \Fcal_\rho] \leq M_\rho$ for all stopping times $\rho$ and $\tau$ such that $\rho \leq \tau$, $\Ptt$-almost surely, for each $\Ptt \in \Pcal$. Further, Doob's supermartingale convergence theorem implies that if $(M_t)$ is a $\Pcal$-NSM with initial expected value one, then its limit $M_\infty := \lim_{t\to\infty}M_t$ exists $\Ptt$-almost surely and $\EE_\Ptt[M_\infty] \in [0,1]$, for each $\Ptt \in \Pcal$.

Stemming from his 1939 PhD thesis~\cite{ville_etude_1939}, Ville's inequality is a time-uniform generalization of Markov's inequality; for our purposes, the relevant version states that if $(M_t)$ is (upper-bounded by) a $\Pbb$-NSM with initial expected value one, then the following three equivalent statements hold:
\begin{subequations}
\begin{align}
\label{eq:Ville}
\Ptt\left(\exists t \in \NN: M_t \geq \frac{1}{\alpha}\right) &\leq \alpha \text{ for every } \Ptt \in \Pbb \text{ and } \alpha \in [0,1];\\
\label{eq:Ville2}
\Longleftrightarrow \qquad \, \sup_{\Ptt \in \Pbb} \Ptt\left(\sup_{t\in\NN} M_t \geq \frac{1}{\alpha}\right) &\leq \alpha \text{ for every }  \alpha \in [0,1];\\
\label{eq:Ville3}
\Longleftrightarrow \qquad  \sup_{\Ptt \in \Pbb, \tau \geq 0} \Ptt\left(M_\tau \geq \frac{1}{\alpha}\right) &\leq \alpha \text{ for every } \alpha \in [0,1].
\end{align}
\end{subequations}
Note that \eqref{eq:Ville2} and \eqref{eq:Ville3}  usually  only hold with  inequality (for example, for the singleton $\Pcal = \{\Ptt\}$) but it can hold with equality for larger nontrivial nonparametric classes $\Pcal$, as we shall encounter later in Section~\ref{sec:anti-concentration}.
We also remark that a conditional version of Ville's inequality is also true, though we do not utilize it much in this paper. Specifically, if $(M_t)$ is a $\Pcal$-NSM, then \begin{equation}\label{eq:conditional-Ville}
    \sup_{\Ptt \in \Pcal} \Ptt\left(\left. \exists t \geq s: M_t \geq \frac{M_s}{\alpha} \right| \Fcal_s \right) \leq \alpha \text{ for every } \alpha \in [0,1].
\end{equation}

\subsection{The relationship between likelihood ratios and nonnegative martingales} 
A simple NM that arises naturally is the likelihood ratio; this is at the heart of Wald's sequential likelihood ratio test~\cite{wald_sequential_1945,wald_sequential_1947}. To be specific, when testing $H_0: X_s \sim Q$ versus $H_0: X_s \sim P$, define the likelihood ratio $M_t := \prod_{s \leq t} \tfrac{\dd P}{\dd Q}(X_s)$, assuming that the Radon-Nikodym derivative $\dd P/\dd Q$ exists. If $P,Q$ have densities $p,q$ with respect to a common measure then each term in the product is just $p(X_s)/q(X_s)$. Let $Q^\infty$ now denote the distribution under which the sequence is i.i.d., each element distributed according to $Q$. Wald effectively proved that $(M_t)$ is a $Q^\infty$-NM and a test that rejects if $M_t \geq 1/\alpha$  controls the Type-I error at level $\alpha$ due to Ville's inequality (Wald proved the result from scratch, but he was not unaware of the evolving language of martingales and Ville's thesis~\citep{bienvenu2009history}). It is also apparent that every nonnegative martingale is a product of nonnegative random variables with conditional mean one, meaning that if $(M_t)$ is a $\Qcal$-NM, then $M_t = \prod_{s \leq t} Y_s$, where $(Y_t)$ is adapted to $(\Fcal_t)$ and $\EE_\Qtt[Y_t|\Fcal_{t-1}]=1$ for every $\Qtt \in \Qcal$; to see this, simply define the multiplicative increment as $Y_t := M_t/M_{t-1}$ with $0/0:=1$. At first sight, despite having a product form, it may appear like nonnegative martingales are strict generalizations of likelihood ratio processes. However, in fact, a converse statement is also true: not only is every likelihood ratio a martingale (under the null), but every martingale is also implicitly a likelihood ratio; this was discussed by \citet{shafer_test_2011} for point nulls, and we generalize it below to the composite case, borrowing the terminology of `implied alternative' from \citet{shafer2019language}.

To make the following result precise we assume that the sequence of observations $(U, (X_t)_{t\in\NN})$ is a process on the space $\Omega=\RR^\NN$ of real-valued sequences, and we let $(\Fcal_t)$ be the canonical filtration.

\begin{proposition}\label{prop:NM-likelihood}
Consider any composite null set $\Qcal$ of probability measures on $\Fcal_\infty$, and let $\Pcal$ consist of all probability measures $\Ptt$ that are locally absolutely continuous with respect to some $\Qtt \in \Qcal$. (Thus in particular, $\Qcal \subset \Pcal$.) If $(M_t)$ is a $\Qcal$-NM starting at one, then for every $\Qtt \in \Qcal$ there exists some `implied alternative' distribution $\Ptt \in \Pcal$ (depending on $\Qtt$) that is locally dominated by $\Qtt$, such that $M_t = \dd\Ptt_t/\dd\Qtt_t$. In other words, a composite nonnegative martingale is a `composite' likelihood ratio (meaning, it takes the form of a likelihood ratio under every element of the null).
\end{proposition}

We recognize that the above statement may be known to different researchers in some form, but it does provide useful intuition and we have not seen it stated in the generality above in the statistics literature. The proof is in Section~\ref{sec:proofs}. The informal takeaway message is that nonnegative martingales are implicitly likelihood ratios, but the former are typically easier to identify (or construct) in composite null settings.

\subsection{Max-martingales (based on conditional supremum)} \label{SS:MM}
%\subsection{Conditional supremum martingales, or max-martingales}

Max-martingales are defined by replacing the conditional expectation by the \emph{conditional supremum}, so we start by reviewing this notion; more information can be found in \citet{bar_car_jen_03} and \citet{larsson2018conditional}; see also Appendix~\ref{sec:esssup}. For a given probability measure $\Ptt$, random variable $Y$, and sub-$\sigma$-algebra $\Gcal$, the \emph{$\Gcal$-conditional supremum} is defined as the smallest $\Gcal$-measurable almost sure upper bound on $Y$:
\[
{\textstyle\bigvee_\Ptt} [ Y \mid \Gcal ] := \essinf\{ Z \colon \text{$Z$ is $\Gcal$-measurable and $Z \ge Y$,  $\Ptt$-almost surely} \}.
\]

Note that ${\textstyle\bigvee_\Ptt} [ Y \mid \Gcal ] \ge Y$ by construction, and is the smallest $\Gcal$-measurable random variable with that property. (Here and in the rest of this subsection, equalities are understood in the $\Ptt$-almost sure sense.) The conditional supremum can be viewed as a nonlinear analog of the conditional expectation, and has similar properties. In particular, one has a `tower property' which states that for nested sub-$\sigma$-algebras $\Gcal \subset \Hcal$, one has
\[
{\textstyle\bigvee_\Ptt} \Big[ {\textstyle\bigvee_\Ptt} [ Y \mid \Hcal ] \Big| \Gcal \Big] = {\textstyle\bigvee_\Ptt} [ Y \mid \Gcal ].
\]
Given a filtration $(\Fcal_t)$, a process $(Y_t)$ is called a \emph{$\Ptt$-max-martingale}, or $\Ptt$-MM for short,
if
\begin{equation*}
Y_s = {\textstyle\bigvee_\Ptt} [ Y_t \mid \Fcal_s ], \quad s\le t.
\end{equation*}
Any max-martingale $(Y_t)$ is almost surely decreasing, which ensures that the limit $ \lim_{t\to\infty}Y_t$ exists in $[-\infty, \infty]$. We call a max-martingale \emph{closed} if \smash{$Y_t = \bigvee_\Ptt [\lim_{s\to\infty}Y_s \mid \Fcal_t ]$ for all $t\in\NN_0$}.

As earlier, given a set $\Pbb$ of measures on $\Fcal_\infty$, a process $(Y_t)$ is called a (closed) $\Pbb$-max-martingale ($\Pbb$-MM)  if $(Y_t)$ is a (closed) $\Ptt$-MM for each $\Ptt \in \Pbb$.

One can also introduce notions of sub- and supermartingales using the conditional supremum, although we will not use these here. Moreover, one can analogously define \emph{conditional infimum} martingales using $\bigwedge$ instead of $\bigvee$, and all properties stated above and below also hold analogously.

In further analogy with (standard) martingales, max-martingales satisfy an optional stopping theorem \cite[Lemma~2.10]{larsson2018conditional}. Specifically, consider the process $Y_t := \bigvee_\Ptt[ Y \mid \Fcal_t ]$ for an $\Fcal_\infty$-measurable $Y$. Thanks to the tower property, $(Y_t)$ is then a max-martingale; we call such a construction a  Doob-L\'evy MM. We then have
\[
Y_\tau = {\textstyle\bigvee_\Ptt} [ Y \mid \Fcal_\tau ] \qquad \text{and} \qquad Y_\rho = {\textstyle\bigvee_\Ptt} [ Y_\tau \mid \Fcal_\rho ]
\]
for all finite stopping times $\rho$ and $\tau$ such that $\rho \le \tau$.

The connection between $\Ptt$-MMs and $\Ptt$-NMs goes beyond mere analogies. If $(M_t)$ is a $\Ptt$-NM with $M_0>0$, 
the following statement easily follows from \cite[Proposition~4.1]{larsson2018conditional}:
\begin{align}\label{eq_inf_M_recip_mg}
\inf_{s \le t} \frac{1}{M_s} = {\textstyle\bigvee_\Ptt} \Big[ \inf_{s \in \NN_0} \frac{1}{M_s} \Big| \Fcal_t \Big], \qquad t \in \NN_0.
\end{align}

A word of warning: although the conditional supremum and conditional expectation are in some ways similar, they sometimes behave very differently. In particular, the conditional supremum only depends on the underlying measure $\Ptt$ through its zero measure sets. Computing the conditional supremum under a different measure $\Ptt'$ thus gives the same result whenever the two measures are mutually absolutely continuous. This is in stark contrast to the behavior of the conditional expectation. (On the other hand, if $\Ptt$ and $\Ptt'$ are mutually singular, as is often the case in our infinite-horizon situations, then one can make no general statements about the relation between the corresponding conditional suprema.)

Although any nonnegative max-martingale $(Y_t)$ is almost surely decreasing, which ensures that the limit $ \lim_{t\to\infty}Y_t$ exists, it is possible that $Y_t$ and $\bigvee_\Ptt [\lim_{s\to\infty}Y_s \mid \Fcal_t ]$ are not necessarily the same, i.e., that $(Y_t)$ is not closed.  
Indeed, we have the following non-uniqueness property. If $(Y_t)$ is a max-martingale then we usually can find another max-martingale $(Y_t')$ with $Y_t' > Y_t$ for each $t \in \NN$,  but $(Y'_t)$ and $(Y_t)$ have the same limit, almost surely. For example, assume
$(Z_t)$ is i.i.d.\ Bernoulli($1/2$), independent of $(Y_t)$, and adapted to $(\Fcal_t)$. Define $Y_t' = Y_t + \prod_{s\leq t} (Z_s \vee 1/2)$. Then $(Y_t')$ is also a max-martingale, with $Y_t' \geq Y_t + 2^{-t} > Y_t$ for all $t \in \NN$, but converging almost surely to the same limit as $Y'_t$ as $t\to\infty$.
% but $(Y'_t)$ and $(Y_t)$ have the same limit almost surely.

A similar phenomenon also holds true for martingales: if $(M_t)$ is an NM then 
$M_t' = M_t + \prod_{s \leq t} (Z_s + 1/2)$ is another NM with  $M_t' \geq M_t + (1/2)^t > M_t$ for all $t \in \NN$, but $\lim_{t \to \infty} M'_t = \lim_{t \to \infty} M_t$. Hence also for nonnegative martingales, we often do not have that $M_t$ and $\EE[\lim_{s \rightarrow \infty} M_s|\Fcal_t]$ are the same.

We conclude this subsection with an interesting example.
\begin{example} \label{ex:200818}
    Let $V$ denote a uniformly distributed random variable and assume that $(X_t)$ is, conditionally on $V$, i.i.d.~$\mathrm{Bernoulli}(V)$. Then $V = \lim_{t \rightarrow \infty} \sum_{s \leq t} X_s / t$ (the limit frequency of ones), hence $V$ is $\Fcal_\infty$-measurable. Moreover, $Y_t := \bigvee_\Ptt[ V \mid \Fcal_t ]$ yields a max-martingale. It is now relatively easy to check that $Y_t = 1$ for each $t \in \NN_0$.  This yields an instance where $\lim_{t \rightarrow \infty} Y_t \neq V$, but $(Y_t)$ is nonetheless a closed max-martingale since $Y_t = 1 = \bigvee_\Ptt[ 1 \mid \Fcal_t ]$.
\end{example}

\section{Tools for anytime sequential inference and their admissibility}
\label{sec:the-4-tools}

We formally introduce the four instruments for anytime-valid sequential inference that play central roles in our paper. We present a rather natural definition of admissibility for each of the instruments below, but recognize that other alternatives may be suitable depending on the context.

\subsection{$\Qcal$-p-processes, $\Qcal$-e-processes, and $(\Qcal,\alpha)$-sequential tests}
Let the (unknown) distribution of the observed data sequence be denoted $\Ptt$. Suppose we wish to test the null $H_0: \Ptt \in \Qcal$ for some $\Qcal \subset \Pbb$, against $H_1: \Ptt \in \Pcal$ (or against $H_1: \Ptt \in \Pcal \backslash \Qcal$).
A sequence $(\Pval_t)$ is called a p-process for $\Qcal$ if it satisfies
\begin{equation*}
\Qtt(\Pval_\tau \leq \alpha)
\leq \alpha \text{ for arbitrary  stopping times $\tau$, every } \Qtt \in \Qcal, \text{ and } \alpha \in [0,1].
\end{equation*}
Above, we have implicitly defined and used
\[
\Pval_\infty := \liminf_{t \to \infty} \Pval_t.
\]
For succinctness, we say $(\Pval_t)$ is a \underline{$\Qcal$-p-process}, or simply a p-process if $\Qcal$ is clear from context.
 By Lemma~\ref{lem:composite_equiv_uniform_defns}, the above condition is equivalent to requiring that 
\[
\Qtt(\exists t \in \mathbb N: \Pval_t \leq \alpha) \leq \alpha \text{ for every } \Qtt \in \Qcal \text{ and } \alpha \in [0,1].
\]
Note that $(\Pval_t)$ is a p-process if and only if the running infimum $(\inf_{s \leq t} \Pval_s)$ is a p-process, so it helps to think of $(\Pval_t)$ as a nonincreasing sequence; in this case, we have $\Pval_\infty = \lim_{t\to\infty} \Pval_t$ (the limit exists and equals the $\liminf$ used earlier).
When we refer only to the validity of the single random variable $\Pval_t$ and not the sequence $(\Pval_t)$, we say that $\Pval_t$ is a p-value for $\Qcal$, meaning that its distribution is stochastically larger than uniform under any $\Qtt \in \Qcal$. 
The connection to Ville's inequality \eqref{eq:Ville}, and thus to martingales, should be  apparent. 

A sequence $(\Eval_t)$ is called an e-process for $\Qcal$, or a \underline{$\Qcal$-e-process}, if
\begin{equation*}
\EE_{\Qtt}[\Eval_\tau] \leq 1 \text{ for arbitrary  stopping times $\tau$, and every } \Qtt \in \Qcal.
\end{equation*}
Analogous to the p-process case above, we have implicitly defined and used 
\[
\Eval_\infty := \limsup_{t \to \infty} \Eval_t.
\]
% \fbox{revisit limsup vs liminf}
% \fbox{add lemma: finite times}
% In fact, replacing $\limsup$ by $\liminf$ would be reasonable as well, and all our proofs go through with either definition. 
In fact, e-processes themselves can be defined in many equivalent ways, and we record this below.

\begin{lemma}\label{L:e-process_definitions}
For any set of distributions $\Qcal$ on any filtered measurable space $(\Omega,(\Gcal_t))$, and any adapted nonnegative process $(\Eval_t)$, the following statements are equivalent:
\begin{enumerate}[label={\rm(\roman{*})}, ref={\rm(\roman{*})}]
    \item\label{L:e-process_definitions.1} $\EE_{\Qtt}[\Eval_\tau] \leq 1 \text{ for arbitrary  stopping times $\tau$, and every } \Qtt \in \Qcal$, where $\Eval_\infty := \limsup_{t \to \infty} \Eval_t$.
    \item\label{L:e-process_definitions.2} $\EE_{\Qtt}[\Eval_\tau] \leq 1 \text{ for arbitrary  stopping times $\tau$, and every } \Qtt \in \Qcal$, where $\Eval_\infty := \liminf_{t \to \infty} \Eval_t$.
    %  \item $\EE_{\Qtt}[\Eval_\tau] \leq 1 \text{ for arbitrary  stopping times $\tau$, and every } \Qtt \in \Qcal.$
   
     \item\label{L:e-process_definitions.3} $\EE_{\Qtt}[\Eval_\tau] \leq 1 \text{ for arbitrary \emph{finite} stopping times $\tau$, and every } \Qtt \in \Qcal.$
      \item\label{L:e-process_definitions.4} $\EE_{\Qtt}[\Eval_\tau] \leq 1 \text{ for arbitrary \emph{bounded} stopping times $\tau$, and every } \Qtt \in \Qcal.$
     \item\label{L:e-process_definitions.5} $(\Eval_t)$ is almost surely upper bounded by some $\Qtt$-NM for each $\Qtt\in\Qcal$.
     \item\label{L:e-process_definitions.6} $(\Eval_t)$ is almost surely upper bounded by some $\Qtt$-NSM for each $\Qtt\in\Qcal$.
\end{enumerate}
Above, the stopping times and (super)martingales are also defined relative to the filtration $(\Gcal_t)$. 
% which may or may not equal the richest filtration $\Fcal$ defined in~Section~\ref{sec:conventions}.
% Further, the above statements are also equivalent if we had defined $\Eval_\infty := \liminf_{t \to \infty} \Eval_t$.
\end{lemma}
Property \ref{L:e-process_definitions.1} is used as the defining property in~\citet{grunwald_safe_2019}, while \ref{L:e-process_definitions.6} is the defining property used in~\citet{howard_exponential_2018}, and while these seem to be visually different, they are in fact identical.
The proof below establishes that property \ref{L:e-process_definitions.4} implies \ref{L:e-process_definitions.5}, by constructing the Snell envelope (under each $\Qtt$) for $(\Eval_t)$, and taking its Doob decomposition. The optional stopping theorem for nonnegative supermartingales establishes that property \ref{L:e-process_definitions.6} implies \ref{L:e-process_definitions.1}. All other implications are straightforward.

\begin{proof}
% The proof uses the filtration $\Fcal$ used elsewhere in the paper, but it can be easily checked that it goes through in any filtration $\Gcal$.
The implications \ref{L:e-process_definitions.1} $\Rightarrow$ \ref{L:e-process_definitions.2} $\Rightarrow$ \ref{L:e-process_definitions.3} $\Rightarrow$ \ref{L:e-process_definitions.4} and \ref{L:e-process_definitions.5} $\Rightarrow$ \ref{L:e-process_definitions.6} $\Rightarrow$ \ref{L:e-process_definitions.1} are immediate. We thus focus on the non-trivial implication \ref{L:e-process_definitions.4} $\Rightarrow$ \ref{L:e-process_definitions.5}. To do so we fix any $\Qtt \in \Qcal$ and define the \emph{Snell envelope} of  $(\Eval_t)$ as the process $(L_t)$ given by
\[
L_t := 
\esssup_{\tau \geq t} \EE_\Qtt[\Eval_\tau \mid \Gcal_t],
\]
where $\tau$ ranges over all bounded stopping times. 

It is clear that $L_t \geq \Eval_t$ since $\tau=t$ is one of the stopping time occurring in the essential supremum, for all $t \in \NN_0$. We claim that $(L_t)$ is a $\Qtt$-supermartingale with $\EE_\Qtt[L_0] \le 1$. The supermartingale property of the Snell envelope is well-known, but for the convenience of the reader we include the short proof. This uses properties of the essential supremum reviewed in Appendix~\ref{sec:esssup}, in particular Proposition~\ref{P_esssup_closed_max}.

For each fixed $t \in \NN_0$, $L_t$ is the essential supremum of the family consisting of all $\EE_\Qtt[M_\tau|\Gcal_t]$ where $\tau\ge t$ is a bounded stopping time. This family is closed under maxima. To see this, let $\tau$ and $\tau'$ be two bounded stopping times, define $A=\{
\EE_\Qtt[\Eval_\tau|\Gcal_t] >\EE_\Qtt[\Eval_{\tau'}\mid\Gcal_t]\}$, and set $\tau'' = \tau\mathbf1_A + \tau'\mathbf1_{A^c}$. Since $A$ lies in $\Gcal_t$ and $\tau,\tau'\ge t$, we have that $\tau''$ is a (bounded) stopping time and we obtain
\[
   \EE_\Qtt[\Eval_{\tau''}\mid\Gcal_t] = \1_A\EE_\Qtt[\Eval_{\tau}\mid\Gcal_t] + \1_{A^c} \EE_\Qtt[\Eval_{\tau'}\mid\Gcal_t] = \max\left\{\EE_\Qtt\left[\Eval_{\tau}\mid\Gcal_t\right],\EE_\Qtt\left[\Eval_{\tau'}\mid\Gcal_t\right]\right\}.
\]
This demonstrates closure under maxima. Consequently we can apply Proposition~\ref{P_esssup_closed_max} to obtain a sequence of bounded stopping times $\{\tau_n\}_{n \in \NN}$ with $\tau_n\ge t$ such that $\EE_\Qtt[\Eval_{\tau_n}|\Gcal_t]\uparrow L_t$ almost surely. The conditional version of the monotone convergence theorem, the tower rule, and the definition of $L_{t-1}$, then yield
\begin{align*}
   \EE_\Qtt[L_t\mid\Gcal_{t-1}] &=\EE_\Qtt\left[\left.\lim_{n \to \infty} \EE_\Qtt[\Eval_{\tau_n}\mid\Gcal_t]\right|\Gcal_{t-1}\right] = \lim_{n \to \infty}\EE_\Qtt[\EE_\Qtt[\Eval_{\tau_n}\mid\Gcal_t]\mid\Gcal_{t-1}] \\
   &= \lim_{n \to \infty}\EE_\Qtt[\Eval_{\tau_n}\mid\Gcal_{t-1}] \le L_{t-1}.
\end{align*}
Note that this argument applies even with $t=0$, where we use the convention that $\Gcal_{-1}$ is the trivial sigma-field and $L_{-1} = 1$. This is where the hypothesis \ref{L:e-process_definitions.4} is used: in the last inequality we have $\EE_\Qtt[\Eval_{\tau_n}] \leq 1$. We conclude that $\EE_\Qtt[L_0] \le 1$, that each $L_t$ is integrable, and that $(L_t)$ is a supermartingale.

Since we have established that the Snell envelope $(L_t)$ is a supermartingale we can write down its Doob decomposition as
$L_t = M_t - A_t$ for a unique $\Qtt$-martingale $(M_t)$ with $M_0=L_0$, and a unique nondecreasing predictable process $(A_t)$ with $A_0=0$. Thus
$M_t \geq L_t \geq \Eval_t$ $\Qtt$-almost surely for all $t \in \NN_0$, showing that $(M_t)$ is the required $\Qtt$-NM that upper bounds $(\Eval_t)$. 
\end{proof}

It is worth noting that already in the case of a singleton $\Qcal :=\{\Qtt\}$, if one additionally desired a conditional safety property to hold for $(\Eval_t)$, namely that $\EE_{\Qtt}[\Eval_\tau | \Fcal_s] \leq \Eval_{s \wedge \tau}$ for arbitrary times $s$ and stopping times $\tau$, then $(\Eval_t)$ must necessarily be a $\Qtt$-NSM; indeed by taking $\tau = t$ for $t \geq s$, we recover the definition of a $\Qtt$-NSM.
In fact, if one would like such a property to hold for every $\Qtt$ in a composite $\Qcal$, such a requirement can only be satisfied by a $\Qcal$-NSM. 
Despite the fact that we will not require this conditional property, we will see that (composite) $\Qcal$-NMs or (pointwise) $\Qtt$-NMs play a central, universal role in constructing e-processes that may not themselves be martingales.

A p-process or e-process does not directly yield a decision rule for when to reject the null hypothesis; they are real-valued measures of evidence, and need to be coupled with a decision rule in order to yield a test.
A binary sequence $(\psi_t)$ is called a $(\Qcal,\alpha)$-sequential test for $H_0$---or 
\underline{$(\psi_t)$ is a $(\Qcal,\alpha)$-ST}---if
\begin{equation*}
\EE_{\Qtt}[\psi_\tau] = \Qtt(\psi_\tau = 1) \leq \alpha \text{ for arbitrary  stopping times $\tau$, and every } \Qtt \in \Qcal.
\end{equation*}
As with p-processes, we think of $(\psi_t)$ as nondecreasing, and have implicitly defined $\psi_\infty := \lim_{t \to \infty} \psi_t$.

\subsection{Defining inadmissibility and admissibility} 
We follow the standard convention of using the term `admissible' as a shorthand for `not inadmissible', so we only define inadmissibility below.

We say that $(\Pval_t)$ is \underline{inadmissible} if there exists $(\Pval'_t)$ that is a $\Qcal$-p-process, and is always at least as good but sometimes strictly better; more formally,
 $\Qtt(\Pval'_t \leq \Pval_t) = 1$ for all $t \in \NN_0$ and all $\Qtt\in\Qcal$, and that $\Qtt(\Pval'_t < \Pval_t) > 0$ for some $t \in \NN_0$, and some $\Qtt\in\Qcal$. 
We say that $(\Eval_t)$ is inadmissible if there exists a $\Qcal$-e-process $(\Eval'_t)$, such that $\Qtt(\Eval'_t \geq \Eval_t)=1$ for all $t \in \NN_0$ and all $\Qtt \in \Qcal$, and also $\Qtt(\Eval'_t > \Eval_t)>0$ for some $t \in \NN_0$, and some $\Qtt \in \Qcal$. 
Finally, $(\psi_t)$ is inadmissible if there exists a  $(\Qcal,\alpha)$-sequential test $(\psi'_t)$, such that $\Qtt(\psi'_t \geq \psi_t)=1$ for all $t \in \NN_0$ and all $\Qtt \in \Qcal$, and $\Qtt(\psi'_t > \psi_t) > 0$ for some $t \in \NN_0$, and some $\Qtt \in \Qcal$.

\begin{remark}
We could have equivalently formulated admissibility in terms of $\Qtt(\Eval'_\tau > \Eval_\tau) > 0$ for some stopping time $\tau$, etc.; this is in fact identical to the current definition, because for any discrete-time random sequences $(W_t)$ and $(W'_t)$, the statement `$\Qtt(W'_t=W_t)=1$ for all $t$' is equivalent to `$\Qtt(W'_t=W_t \text{ for all } t)=1$', meaning that pointwise equality at fixed times yields simultaneous equality (including stopping times). In case it is useful to the reader, another restatement of inadmissibility is the following:
\begin{equation}\label{eq:inadmissible-rewrite}
\inf_{\Qtt \in \Qcal} \Qtt(\forall t \in \NN_0: \Pval'_t \leq \Pval_t) = 1 \text{ and } \sup_{\Qtt \in \Qcal} \Qtt(\exists t \in \NN_0: \Pval'_t < \Pval_t) > 0.
\end{equation}
\end{remark}

% \begin{remark}
The above definition of admissibility does not specify any alternative. What allows us to do so is that the admissibility conditions are stated `almost surely'. To elaborate, assume there was an alternative, say $\Ptt^*$ such that there existed some time $t$ and an event $A \in \Fcal_t$ with $\Ptt^*(A) > 0$ but $\Qtt(A) = 0$ for all $\Qtt \in \Qcal$. In this case, on the event $A$ we would always set $\Pval_t = 0$, $\Eval_t = \infty$, and $\psi_t = 1$. We could always do so because no $\Qtt \in \Qcal$ `is aware of' the event $A$, hence modifying a p-process on $A$ would not change any of its $\Qtt$-distributional properties (like validity). For this reason, and to avoid any notational complication arising, from now on \underline{we always assume} the following:
\begin{quote}
If there exists some $t \in \NN$, some $A \in \Fcal_t$, and some $\Ptt^* \in \Pcal$ with $\Ptt^*(A) > 0$ then there exists also some $\Qtt \in \Qcal$  with $\Qtt(A) > 0$.
\end{quote}
This is a very mild assumption! It is, for example, satisfied if there exists some $\Qtt^* \in \Qcal$ such that every $\Ptt \in \Pcal$ is locally absolutely continuous with respect to $\Qtt^*$ (see Appendix~\ref{sec:ref_measure} for a review). This previous condition is satisfied, for example, if all measures in $\Pcal$ are locally equivalent. 
In interesting situations a locally dominating measure $\Qtt^*$ %or $\Rtt$ 
may not actually exist---for example when testing for symmetry, as we encounter later---but our standing assumption above is still true. Specifically, the aforementioned assumption also holds if each $\Ptt \in \Pcal$  is locally absolutely continuous with respect to some specific $\Qtt \in \Qcal$ (potentially different for each $\Ptt$)---this is often easier to check, and indeed holds in the symmetric example.

Under the above assumption, the inadmissibility condition~\eqref{eq:inadmissible-rewrite} holds with $\Qtt,\Qcal$ swapped to $\Ptt, \Pcal$.

%We intentionally use $\Pbb$ in some spots in the above definitions of admissiblity, and not $\Qcal$, since we technically prefer to get smaller p-values or larger e-processes under the alternative $\Pbb \backslash \Qcal$. Hence, one could also consider defining admissibility by replacing $\Pbb$ above with $\Pbb \backslash \Qcal$.  \fbox{to discuss}
% \end{remark}

% \begin{remark}\label{rem:admissibility-subset-supset}
Note that validity is subset-proof, meaning that validity for $\Qcal$ implies validity for any subset of $\Qcal$. In the presence of validity, admissibility is --- under a condition involving polar sets --- superset-proof. The following proposition will play a key role in Section~\ref{sec:subG}.

\begin{proposition}\label{prop_adm_superset_proof}
Suppose $\Qcal \subset \Qcal'$ are such that for all $t \in \NN_0$, each $\Qcal$-polar set $A \in \Fcal_t$ is $\Qcal'$-polar; this means that if $\Qtt(A) = 0$ for all $\Qtt \in \Qcal$, then $\Qtt'(A) = 0$ for all $\Qtt' \in \Qcal'$. Let $(\Eval_t)$ be a $\Qcal'$-e-process (and thus also a $\Qcal$-e-process). If $(\Eval_t)$ if $\Qcal$-admissible then it is $\Qcal'$-admissible.
% \begin{enumerate}
%     \item If $(\Eval_t)$ is a $\Qcal'$-e-process (and thus also a $\Qcal$-e-process) that is $\Qcal$-admissible, then it is also $\Qcal'$-admissible.
% \end{enumerate}
\end{proposition}
\begin{proof}
Let $(\Eval_t)$ be as described in the proposition, and assume it is $\Qcal$-admisible. We must show it is $\Qcal'$-admissible. Consider another $\Qcal'$-e-process $(\Eval'_t)$ such that $\Qtt'(\Eval'_t \ge \Eval_t) = 1$ for all $t \in \NN_0$ and all $\Qtt' \in \Qcal'$, hence all $\Qtt \in \Qcal$. By $\Qcal$-admissibility, $\Qtt(\Eval'_t = \Eval_t) = 1$ for all $\Qtt \in \Qcal$ and then also for all $\Qtt \in \Qcal'$ thanks to the assumption regarding the polar sets. Thus $(\Eval_t)$ is $\Qcal'$-admissible.
% Conversely, assume $(\Eval_t)$ is $\Qcal'$-admissible. Consider a $\Qcal$-e-process $(\Eval'_t)$ such that $\Qtt(\Eval'_t \ge \Eval_t) = 1$ for all $t \in \NN_0$ and all $\Qtt \in \Qcal$. Then this holds for all $\Qtt' \in \Qcal'$ as well, since the polar sets coincide (indeed, $\Qcal'$-polar sets are automatically $\Qcal$-polar because $\Qcal \subset \Qcal'$). By $\Qcal'$-admissibility, $\Qtt(\Eval'_t = \Eval_t) = 1$ for all $\Qtt \in \Qcal'$, hence all $\Qtt \in \Qcal$. Thus $(\Eval_t)$ is $\Qcal'$-admissible.
\end{proof}

We sometimes call $\Qcal$-admissibility as just admissibility if $\Qcal$ is clear from context.

% \fbox{Clean up:}
% however, admissibility is neither subset-proof nor superset-proof. It is posssible that $(\Pval_t)$ may be admissible for $H_0: \Qtt \in \Qcal$, but not for $H_0: \Qtt \in \Qcal'$ for some $\Qcal' \subset \Qcal$ or $\Qcal' \supset \Qcal$; indeed, it may not even be valid in the latter case. 
% To complicate matters further, admissibility is subset-proof under the alternative, meaning that if $(\Pval_t)$ is admissible against $\Pcal$, it is also admissible against $\Pcal' \subset \Pcal$ but not necessarily against $\Pcal' \supset \Pcal$. 
% Thus, to avoid confusions, we should say $(\Pval_t)$ is  $\Qcal$-admissible, but we sometimes drop the additional prefix if it can be inferred from the context.  (The same logic applies for $(\Eval_t)$ and $(\psi_t)$.)

A family of sequential tests $\{(\psi_t(\alpha))\}_{\alpha \in [0,1]}$ is said to be `nested in $\alpha$', if $\psi_t(\alpha') \geq \psi_t(\alpha)$ for any $t$ and $\alpha' \geq \alpha$. 
This ensures that it is easier to reject the null with a larger error budget. 

We conclude this subsection with the following observation.
\begin{proposition} \label{P:200905}
Assume that $\Qcal$ is locally dominated. Then any p-process for $\Qcal$, e-process for $\Qcal$, or $(\Qcal, \alpha)$-sequential test can  be dominated by an admissible p-process, e-process, or sequential test, respectively. 
\end{proposition}
The proof (Appendix~\ref{sec:proofs}) uses transfinite induction. \smash{We now move from testing to estimation.}

%Example~\ref{Ex:200905} below illustrates that without assuming that $\Qcal$ is locally dominated an admissible sequential test, etc., might just not exist. This shows the relevance of this assumption when deriving theoretical existence results. All hope is not lost, however, for the non-dominated situation. Indeed, in many applications, such as in the case of symmetric distributions, one is often able to construct admissible sequential tests `by hand'; see, for instance, Section~\ref{sec:examples} later on.

\subsection{$(\phi,\Pbb,\alpha)$-confidence sequences}

 Let $\phi$ be a measurable map from $\Pcal$ to an arbitrary set $\Zcal$. For every $\gamma \in \Zcal$, define
\[
\Pcal^\gamma := \{\Ptt \in \Pcal: \phi(\Ptt) = \gamma\},
\]
and note that $\{\Pcal^\gamma\}_{\gamma \in \Zcal}$ is a partition of $\Pcal$.
Of special interest is the i.i.d.\ case where $\Ptt=\mu^\infty$. As examples, consider $\phi(\Ptt) = \phi^{\rm{med}}(\mu)$ denoting the median of $\mu$, or $\phi(\Ptt) = \phi^{\rm{mean}}(\mu)$ denoting the mean of $\mu$. Another case of interest is where each $\Ptt \in \Pbb$ can be represented as $\Ptt_\theta$ for some unique $\theta \in \Theta$ (the `parametric' setting) and then $\phi(\Ptt_\theta) = \theta$. One may also care about fully nonparametric functionals, for example $\phi^{\rm{cdf}}$, which maps $\mu$ to its cumulative distribution function (or $\phi^{\rm{pdf}}$ if $\mu$ has a Lebesgue density). For a set of distributions $\Qcal$, define $\phi(\Qcal)=\bigcup_{\Qtt \in \Qcal} \phi(\Qtt)$ for short.

We then define a $(1-\alpha)$-confidence sequence for a functional $\phi$ as an adapted sequence of confidence sets $(\CI_t)$ such that
\begin{equation*}
\sup_{\Ptt \in \Pcal} \Ptt\left(\exists t \in \NN: \phi(\Ptt) \notin \CI_t\right) \leq \alpha.
\end{equation*}
This error condition can be phrased equivalently as a coverage criterion:  for every $\Ptt \in \Pbb$, we have \smash{$\Ptt\left(\forall t \in \NN_0: \phi(\Ptt) \in \CI_t\right) \geq 1-\alpha$}. Above, we have suppressed the dependence of $\CI_t$ on $\alpha$ for notational succinctness. Moreover, we emphasize that the probability $\Ptt$ on the outside matches with the inner functional $\phi(\Ptt)$, as it should. Indeed, the coverage probability for $\phi(\Ptt)$ should hold given that the data used to construct $\CI_t$ was drawn according to this $\Ptt$. We then say that \underline{$(\CI_t)$ is a $(\phi,\Pbb,\alpha)$-CS}. If required, we define $\CI_\infty := \liminf_{t\to\infty} \CI_t$. 
% Further, if $\phi(\Ptt)=\Ptt$ is the identity mapping, then we just say that $\CI_t$ is a \underline{$(\Pcal,\alpha)$-CS}---however, one can obviously not explicitly represent such CS in general, but sometimes one can efficiently check if a distribution $\Ptt$ lies in the CS.

Note that even though we defined $(\Qcal,\alpha)$-sequential tests without explicit reference to any functional $\phi$, one could associate $\Qcal$ to the subset of $\Pcal$ in which $\phi$ takes on certain values. In other words, we can absorb the role of any such $\phi$ into the definition of $\Qcal$ to reduce notation.
% This is already implicit in.

We say $(\CI_t)$ is \underline{inadmissible} if there exists $(\CI_t')$ that is a $(\phi,\Pbb,\alpha)$-CS such that $\smash{\Ptt(\CI_t' \subseteq \CI_t) = 1}$ for all $t$ and all $\Ptt \in \Pbb$, and $\Ptt(\CI_t'\subsetneq \CI_t) > 0$ for some $t$ and some $\Ptt \in \Pbb$.

Finally, let us mention a minor technical point. A CS $(\CI_t)$ is by definition an \emph{adapted sequence of random sets}. In this paper, we take `adapted' to simply mean that all events of the form $\{z\in\CI_t\}$ belong to $\Fcal_t$. This is sufficient for our needs, and lets us avoid dealing with $\sigma$-algebras on spaces of sets.

Last, as for sequential tests, a family of CSs $\{(\CI_t(\alpha))\}_{\alpha \in [0,1]}$ is said to be `nested in $\alpha$' if $\CI_t(\alpha') \subseteq \CI_t(\alpha)$ for any $t$ and $\alpha' \geq \alpha$. 
\begin{remark} \label{R:200905}
    An analogue of Proposition~\ref{P:200905} also holds for $(\phi,\Pbb,\alpha)$-confidence sequences, where $\Pcal$ is locally dominated. This will later follow from Theorem~\ref{thm:ST/CS-admissibility}, in conjunction with Proposition~\ref{P:200905}.
\end{remark}

\subsection{Reductions between the four instruments}
\label{sec:reductions}

We now demonstrate how to transform one tool for sequential inference into another; some of these are well-known or `obvious' but some are new, especially in the composite setting. These transformations are usually not lossless or bidirectional, meaning that something may be lost in going from one to another and back again. All proofs are relegated to Appendix~\ref{sec:proofs}. 

% \smallskip
We start with the simplest direction: forming $(\Qcal,\alpha)$-sequential tests from the other three.
\begin{proposition}\label{prop:construct-seq-tests}
One can construct $(\Qcal,\alpha)$-sequential tests in the following ways:
\begin{enumerate}[label={\rm(\arabic{*})}, ref={\rm(\arabic{*})}]
    \item\label{prop:construct-seq-tests:1} If $(\Pval_t)$ is a $\Qcal$-p-process, then $\{(\1_{\Pval_t \leq \alpha})\}_{\alpha \in [0,1]}$ is a nested family of  $(\Qcal,\alpha)$-STs.
    \item\label{prop:construct-seq-tests:2} If $(\Eval_t)$ is a $\Qcal$-e-process, then $\{(\1_{\Eval_t \geq 1/\alpha})\}_{\alpha \in [0,1]}$ is a nested family of   $(\Qcal,\alpha)$-STs.
    \item\label{prop:construct-seq-tests:3} If $(\CI_t)$ is a $(\phi, \Pcal,\alpha)$-CS, then $(\1_{\phi(\Qcal) \cap \CI_t = \emptyset})$ is a $(\Qcal,\alpha)$-ST.
\end{enumerate}
\end{proposition}

Next, we show how to form a composite p-process from the other three.
\begin{proposition}\label{prop:construct-p-vals}
One can construct p-processes for $\Qcal$ in the following ways:
\begin{enumerate}[label={\rm(\arabic{*})}, ref={\rm(\arabic{*})}]
    \item\label{prop:construct-p-vals:1} If $(\Eval_t)$ is a $\Qcal$-e-process, then $(1 \wedge \inf_{s \leq t} 1/\Eval_s)$ is a $\Qcal$-p-process.
    \item\label{prop:construct-p-vals:2} If $\{(\psi_t(\alpha))\}_{\alpha \in [0,1]}$ is a nested family of $\Qcal$-STs, then $(\inf\{\alpha: \psi_t(\alpha)=1\})$ is a $\Qcal$-p-process.
    \item\label{prop:construct-p-vals:3} If $\{(\CI_t(\alpha))\}_{\alpha \in [0,1]}$ is a nested family of  $(\phi, \Pcal)$-CSs, then $(\inf\{\alpha: \phi(\Qcal) \cap  \CI_t(\alpha) = \emptyset\})$ is a $\Qcal$-p-process.
\end{enumerate}
\end{proposition}

A confidence sequence can be formed by inverting families of tests, as follows.

\begin{proposition}\label{prop:construct-CS}
Recall that $\Pcal^\gamma := \{\Ptt \in \Pcal: \phi(\Ptt) = \gamma\}$.
% Consider a decomposition $\smash{\Pcal := \bigcup_{\gamma \in \Zcal} \Pcal^\gamma}$. %; in particular, the finest (singleton) partition of $\Pcal$ suffices, but in general the sets $\Pcal^\gamma$ may overlap.
If $(\psi^\gamma_t)$ is a $(\Pcal^\gamma,\alpha)$-sequential test for each $\gamma \in \Zcal$ then $\{\gamma \in \Zcal: \psi_t^\gamma=0\}$ is a $(\phi, \Pcal,\alpha)$-confidence sequence. 
To convert a family of p-processes or e-processes into confidence sequences, we can first convert them to sequential tests using Proposition~\ref{prop:construct-seq-tests} and then invert those tests as in the previous sentence.  
% \com{proof needs work}
\end{proposition} 

Finally, we can form e-processes by calibrating p-processes, following~\citet{shafer_test_2011}.
\begin{proposition}\label{prop:construct-e-vals}
One can construct e-processes for $\Qcal$ as follows:
If $(\Pval_t)$ is a $\Qcal$-p-process, then $1/(2\sqrt{\Pval_t})$ is a $\Qcal$-e-process. In fact, $(f(\Pval_t))$ is a $\Qcal$-e-process for any nonincreasing `calibration' function $f:[0,1]\to[0,\infty)$ such that $\int_0^1 f(u)\dd u=1$.
To convert a nested family of sequential tests or confidence sequences into e-processes, we can first convert them to p-processes using Proposition~\ref{prop:construct-p-vals} and then apply the aforementioned calibration.
\end{proposition}

% \subsection{Some basic closure properties}
% \label{sec:closure-properties}

We end this section with some basic closure properties.

% At several points in this paper, we will see that convexity of the class of distributions plays a central role, the first hint of which is provided below.

\begin{proposition}\label{P:convex_hull}
Let $\rm{conv}(\Qcal)$ denote the convex hull of $\Qcal$. Then the following statements hold.
\begin{enumerate}[label={\rm(\arabic{*})}, ref={\rm(\arabic{*})}]
\item\label{P:convex_hull:1} An e-process for $\Qcal$ is also automatically an e-process for $\rm{conv}(\Qcal)$. Such closure under the convex hull also holds for p-processes and sequential tests.
\item\label{P:convex_hull:2}  An admissible e-process for $\Qcal$ is also automatically an admissible e-process for $\rm{conv}(\Qcal)$. Such closure under the convex hull also holds for p-processes and sequential tests.
\end{enumerate}
\end{proposition}

The proof can be found in Appendix~\ref{sec:proofs}. Example~\ref{eg:conv-closure-CS} in Appendix~\ref{sec:more-examples} shows that the statement of Proposition~\ref{P:convex_hull} does not extend to confidence sequences.

\begin{remark}\label{rem:running-extremum}
Note that stability with respect to taking a running extremum does not hold for e-processes, but does hold for the other three. To elaborate, if $(\Pval_t)$ is a $\Qcal$-p-process, then so is $(\min_{s \leq t} \Pval_s)$; if $(\CI_t)$ is a $(\phi,\Pcal,\alpha)$-CS, then so is $(\bigcap_{s \leq t} \CI_s)$; if $(\psi_t)$ is $(\Qcal,\alpha)$-ST, then so is $(\max_{s \leq t} \psi_s)$. However, a running maximum does not usually preserve safety for e-processes!
% \paragraph{Min, max, intersection, union.} 
Moreover, if $(\Eval_t)$ is a $\Pcal$-e-process and $(\mathfrak f_t)$ is a $\Qcal$-e-process, then $\min(\Eval_t,\mathfrak f_t)$ is a $\rm{conv}(\Pcal \cup \Qcal)$-e-process, while $(\Eval_t+\mathfrak f_t)/2$ is a $\rm{conv}(\Pcal \cap \Qcal)$-e-process. 
Similarly, if $(\Pval_t)$ is a $\Pcal$-p-process and $(\mathfrak q_t)$ is a $\Qcal$-p-process, then $\max(\Pval_t, \mathfrak q_t)$ is a $\rm{conv}(\Pcal \cup \Qcal)$-p-process, while $2\min(\Pval_t,\mathfrak q_t)$ is a $\rm{conv}(\Pcal \cap \Qcal)$-p-process.
Last, if $(\CI_t)$ is a $(\phi,\Pcal,\alpha)$-CS and $(\mathfrak D_t)$ is a $(\phi,\Qcal,\alpha)$-CS, then $\CI_t \cup \mathfrak D_t$ is a $(\phi,\Pcal \cup \Qcal,\alpha)$-CS, while $\CI_t \cap \mathfrak D_t$ is a $(\phi,\Pcal \cap \Qcal,2\alpha)$-CS. 
\end{remark}
% \com{running intersection}

%Of course, these properties are far from exhaustive, but they suffice for the moment.

\section{Three instructive examples of exponential (super)martingales}
\label{sec:two-examples}

We now illustrate the above concepts with three examples: the Gaussian NM, a subGaussian NSM and the symmetric NSM, the former being a simple parametric example, and the latter two being sophisticated nonparametric examples (we use the term sophisticated because it includes atomic and nonatomic distributions, and there is no underlying reference measure).%, making later arguments about admissibility quite nontrivial).

\subsection{A Gaussian nonnegative martingale}
\label{sec:gaussian}

If $(X_t)$ is a sequence of i.i.d.\ standard Gaussians, then it is well known that $\exp(\sum_{s\leq t}X_s - t/2)$ is a nonnegative martingale. This is one of the simplest nontrivial pointwise NMs that one can construct. Below, we elaborate on how to construct a simple composite NM in the Gaussian case. We generalize the Gaussian case to an arbitrary fixed mean $m$ and a variance process, where each $\sigma_t^2$ is revealed to us on demand. Technically, we
%
% \begin{example}[Gaussian NM]\label{eg:gaussian}
% \com{rephrase --- point null?}
suppose the process $(X_t)$ consists of two components, $X_t = (Y_t,\sigma_{t+1}^2)$, and let $\sigma_1^2$ be a constant. Note that $(\sigma_t^2)$ is a predictable sequence. We now let $\Gcal^m$ denote the set of distributions such that the outcome $Y_t$ at time $t$ is conditionally Gaussian with deterministic mean $m$ and  variance $\sigma_t^2$, that is,
%
%=========
%Suppose the process $(X_t)$ consists of two components, $X_t = (Y_t,\sigma_t^2)$, and let $\Gcal^m$ denote the set of distributions such that the outcome $Y_t$ at time $t$ is conditionally Gaussian with deterministic mean $m$ and variance $\sigma_{t-1}^2$.
%=========
%Let $\Gcal^m$ denote the set of distributions over sequences such that the element at time $t$ is conditionally Gaussian with deterministic mean $m$ and predictable variance $\sigma_t^2$, that is,
\[
\Gcal^m := \{\Ptt: Y_t ~|~ \Fcal_{t-1} \sim \Ncal(m,\sigma_t^2) \text{ for all $t \in \NN$} \}.
%\Gcal^m := \{\Ptt: X_t ~|~ \Fcal_{t-1} \sim \Ncal(m,\sigma_t^2) \text{ for some positive and predictable sequence } (\sigma_t^2) \}.
\]
Since the distribution generating the predictable variances is left unspecified, $\Gcal^m$ is a highly composite set of measures.
Next, define the process $(G_t^m)$ by
\begin{equation}\label{eq:Gaussian-NM}
G_t^m := \exp\left(\sum_{s \leq t} (Y_s-m) - \frac{1}{2} \sum_{s \leq t} \sigma_s^2
\right).
\end{equation}
It is easy to check that $(G_t^m)$ is a $\Gcal^m$-NM for each $m \in \RR$, specifically by evaluating the moment generating function of $X_t$.
As a direct consequence of Ville's inequality, a $(1-\alpha)$-CS $(\CI^{\rm{mean}}_t)$ for an unknown mean is given by
\[
\CI^{\rm{mean}}_t := \left\{ m \in \RR: G_t^m < \frac{1}{\alpha} \right\}. 
\]
Formally, $(\CI^{\text{mean}}_t)$ is a $(\phi^{\text{mean}}, \bigcup_{m \in\RR} \Gcal^m,\alpha)$-CS for the mean, where we define the mean functional $\phi^{\rm{mean}}(\Ptt)=m$ for every $\Ptt \in \Gcal^m$.

\begin{remark}
The confidence sequence $(\CI^{\rm{mean}}_t)$ derived above does not yield interval \eqref{eq:normal-mixture-CS} in the introduction. That CS is obtained by noting that $(G_t^m(\lambda))$, given by
\[
    G_t^m(\lambda) := \exp\left(\lambda \sum_{s\leq t} (Y_s-m) - \frac{\lambda^2}{2} \sum_{s\leq t} \sigma_s^2\right),
\]
is a $\Gcal^m$-NM for every $\lambda \in \RR$. Fubini's theorem then implies that $(\int G_t^m(\lambda) \dd \Phi(\lambda))$ is also a $\Gcal^m$-NM for any distribution function $\Phi$. Choosing $\Phi$ as a standard Gaussian for example, yields the normal mixture martingale (dating back at least to Darling and Robbins~\cite{darling_confidence_1967}). Applying Ville's inequality then yields~\eqref{eq:normal-mixture-CS}.
\end{remark}

{\color{black}
\subsection{A subGaussian supermartingale and e-process}
\label{sec:subgaussian-defn}

For a positive predictable process $(\sigma_t)$, we say that $Y_t$ is conditionally $\sigma_t$-subGaussian, if after centering by its conditional mean $\mu_t := \EE_{\Ptt}[Y_t | \Fcal_{t-1}]$ (assumed to exist), its conditional moment generating function is upper bounded by that of a Gaussian with variance $\sigma_t^2$; let $\Gcal^\downarrow$ consist of all such distributions:
% For simplicity, below we can think of $\sigma_t^2 \equiv \sigma^2$ being constant:
\[
\Gcal^\downarrow := \{\Ptt: \EE_\Ptt [\exp(\lambda (Y_t-\mu_t)) ~|~ \Fcal_{t-1}] \leq \exp(\lambda^2\sigma_t^2/2) \text{ for all $t \in \NN$, $\lambda \in \RR$} \}.
\]
Then, one can check that the process
\begin{equation}\label{eq:subG-NSM}
M^\Ptt_t := \exp\Big(\lambda \sum_{i \leq t} (Y_i-\mu_i) - \lambda^2\sum_{i \leq t}\sigma_i^2/2\Big)
\end{equation}
is a $\Ptt$-NSM for each $\Ptt \in \Gcal^\downarrow$. Note that the process $M^\Ptt_t$ changes with the conditional means of $\Ptt$, and thus it depends on $\Ptt$, and so $M^\Ptt_t$ is not a $\Gcal^\downarrow$-NSM.

Now, suppose we want to test the null hypothesis that $\mu_t \leq 0$ for all $t$, assuming that $\sigma_t \equiv \sigma$, a constant, for simplicity. Consider three subsets of $\Gcal^\downarrow$:
\begin{align}
\Gcal_{0}^\downarrow &:= \{\Ptt: \mu_t = 0 \text{ and } \EE_\Ptt [\exp(\lambda Y_t) ~|~ \Fcal_{t-1}] \leq \exp(\lambda^2\sigma^2/2) \text{ for all $t \in \NN$, $\lambda \in \RR$} \}.
\\
\Gcal_{-}^\downarrow &:= \{\Ptt: \mu_t \leq 0 \text{ and } \EE_\Ptt [\exp(\lambda (Y_t-\mu_t)) ~|~ \Fcal_{t-1}] \leq \exp(\lambda^2\sigma^2/2) \text{ for all $t \in \NN$, $\lambda \in \RR$} \}.
\\
\Gcal_{\sim}^\downarrow &:= \{\Ptt: \sum_{i\leq t}\mu_i \leq 0 \text{ and } \EE_\Ptt [\exp(\lambda (Y_t-\mu_t)) ~|~ \Fcal_{t-1}] \leq \exp(\lambda^2\sigma^2/2) \text{ for all $t \in \NN$, $\lambda \in \RR$} \}.
\end{align}
Now consider the process
\begin{equation}\label{eq:subG-NSM2}
N_t := \exp\left(\lambda \sum_{i \leq t} Y_i - \lambda^2\sigma^2 t/2\right).
\end{equation}
Then, we have the following three properties:
\begin{enumerate}
    \item $N$ is a $\Gcal_{0}^\downarrow$-NSM for any $\lambda \in \RR$.
    \item $N$ is a $\Gcal_{-}^\downarrow$-NSM for any $\lambda \geq 0$.
    \item  $N$ is an $\Gcal_{\sim}^\downarrow$-e-process for any $\lambda \geq 0$.
\end{enumerate}
Importantly, note that in the third case, $N$ is not a $\Gcal_{\sim}^\downarrow$-NSM. Indeed, the weak null allows $\mu_t$ sequences like $(-1,-1,1,0,-1,\dots)$, and the NSM property would be violated at the third step. In fact, one can prove (but we omit the proof) that there are $\Gcal_{\sim}^\downarrow$-NSMs (except for constant or decreasing processes).
However, the e-process property follows thanks to Lemma~\ref{L:e-process_definitions}\ref{L:e-process_definitions.6}; this is because for positive $\lambda$ and for each $\Ptt \in \Gcal^\downarrow_{\sim}$, we have $N \leq M^\Ptt_t$. 
% Note importantly that unlike~\eqref{eq:subG-NSM}, the process in~\eqref{eq:subG-NSM2} only depends on the data (and thus can be calculated for the purposes of testing the weak null). 

In Section~\ref{sec:subG}, we will prove that $N$ is actually an \emph{admissible} e-process for all three classes above, which turns out to be a rather nontrivial fact for reasons we elaborate on later. 
% To further appreciate the nontriviality of this fact, note that $\Gcal_{-}^\downarrow$ contains continuous distributions as well as discrete distributions with arbitrary supports, so it is has no single common reference measure, making it a very rich nonparametric class of distributions. When proving admissibility, one has to make sure that there is not \emph{some} (say discrete) distribution for which one could strictly improve $N$ at some time $t$ (by altering it on a Lebesgue-measure-zero set) while still managing to remain an NSM for all $\Gcal_{-}^\downarrow$ distributions.
}

\subsection{A nonnegative supermartingale under symmetry}
\label{sec:symmetric-NSM}

We now move to a nonparametric example that first appears in %de le Peña (1999)~
\citet{de_la_pena_general_1999} but the core idea can be traced back to %Efron (1969)~
\citet{efron_students_1969}, who was interested in the robustness of the $t$-test to heavy tailed (but symmetric) distributions; further, it has recently been extended to the matrix setting by \citet{howard_exponential_2018}.
This example is particularly interesting because of three reasons: (a) it deals with a nonparametric class of distributions that does not have a common dominating measure (it includes atomic and non-atomic measures), (b) there exists a rather elegant well-known `nonparametric' NSM, meaning that one single process is a composite NSM, (c) the visual form of the NSM below is reminiscent of the aforementioned Gaussian example, making it intuitively appealing. Unfortunately, we will later demonstrate that this construction is suboptimal, and leads to inadmissible sequential inference.

Consider the convex set of distributions indexed by $m \in \RR$, where each increment is conditionally symmetric around $m$, i.e.,
\begin{equation}\label{eq:symmetric-distributions}
\Scal^m := \{\Ptt: (X_t-m) \sim -(X_t-m) ~|~ \Fcal_{t-1}  \}.
\end{equation}
Consider then the following family of processes $(S_t^m)$, indexed by $m\in\RR$:
\begin{equation}\label{eq:symmetric-NSM}
S_t^m := \exp\left(\sum_{s\leq t} (X_s-m) - \frac{1}{2} \sum_{s\leq t} (X_s-m)^2 \right).
\end{equation}
It is known~\cite{de_la_pena_general_1999} that $(S_t^m)$ is an $\Scal^m$-NSM for each fixed $m \in \RR$; the proof stems from the observation that $\cosh(z) \leq \exp(z^2/2)$, and thus for a single symmetric random variable $Z$, we have
\[
\EE[e^{Z - Z^2/2}] = \EE[e^{-Z -Z^2/2}] = \EE\left[\frac{e^{Z -Z^2/2} + e^{-Z - Z^2/2}}{2}\right] = \EE[e^{-Z^2/2} \cosh(Z)] \leq 1.
\]
Note that the symmetric distribution could be different at each time point (e.g., Gaussian with mean $m$ at time one, $(\delta_{m-|X_1|} + \delta_{m+|X_1|})/2$ at time two, etc., where $\delta_z$ denotes the Dirac measure at some $z \in \mathbb{R}$).

The process $(S^m_t)$ is visually quite similar to the Gaussian process $(G^m_t)$ from the previous subsection.
The relaxation from using the true variance in $(G_t^m)$ to using an empirical variance in $(S_t^m)$, allows the NM property to transform to an NSM property for a much larger class of heavy-tailed distributions (such as $t$ and Cauchy distributions). Even when applied to Gaussians, one no longer needs to know the variance.
% if the data are being drawn from a distribution that is symmetric around some point, then 

One can check using Propositions~\ref{prop:construct-seq-tests} and~\ref{prop:construct-CS} that the sets
\[
\CI^{\rm{center}}_t := \left\{ m \in \RR: S_t^m < \frac{1}{\alpha}\right \} 
\]
together form a $(\phi^{\rm{center}},  \bigcup_{m \in\RR} \Scal^m,\alpha)$-CS for the center of symmetry. 
Said differently, if $\Ptt \in \Scal^m$ then $\Ptt(\exists t\in \NN: m \notin \CI^{\text{center}}_t) \leq \alpha$.

Above, we have described only confidence sequences, but one could have defined e-processes and p-processes. For example, to test the null $H_0: \Ptt \in \Scal^0$, one can use the fact that $(S_t^0)$ is an NSM under the null, and thus $(S_t^0)$ is an e-process for $\Scal^0$, and $(\inf_{s\leq t}1/S_s^0)$ is a p-process for $\Scal^0$. Of course, all of these in turn define $(\Scal^0,\alpha)$-sequential tests.

We return to these and other examples later in the paper.

\section{Admissible inference for point nulls via martingales}\label{sec:admissible-pointwise}

We begin our examination of admissibility via the lens of testing (p-processes, e-processes, and sequential tests) and leave the results on estimation (confidence sequences) for Subsection~\ref{subsec:CS}.

\begin{theorem}[Necessary and sufficient conditions for pointwise admissibility]\label{thm:pointwise}
Consider a point null $\Qcal=\{\Qtt\}$. The following statements describe necessary and sufficient conditions for pointwise admissibility.
\begin{enumerate}[label={\rm(\arabic{*})}, ref={\rm(\arabic{*})}] 
    \item\label{thm:pointwise:1} 
    If $(\Pval_t)$ is admissible, then it is a closed $\Qtt$-MM with $F(\inf_{t \in \NN_0} \Pval_t) = \inf_{t \in \NN_0} \Pval_t$, where $F$ is the distribution function of $\inf_{t \in \NN_0} \Pval_t$. In the other direction, if $(\Pval_t)$ is a closed $\Qtt$-MM and $\inf_{t \in \NN_0} \Pval_t$ is $\Qtt$-uniformly distributed, then it is admissible. 
    \item\label{thm:pointwise:2} 
    $(\Eval_t)$ is admissible if and only if it is a $\Qtt$-NM with $\EE_{\Qtt}[\Eval_0] = 1$.
    \item\label{thm:pointwise:3} $(\psi_t)$ is an admissible $(\Qtt,\alpha)$-sequential test if and only if $\psi_t := \1_{\sup_{s\leq t}M_s \geq 1/\alpha}$, where $(M_t)$ is a $\Qtt$-NM with no overshoot at $1/\alpha$ and $M_\infty \in \{0,1/\alpha\}$.
\end{enumerate}
\end{theorem}
The theorem is proven later in this section, with one subsection per statement.
\begin{remark}  
    Similar to max-martingales (Subsection~\ref{SS:MM}), we could have introduced min-martingales by replacing conditional suprema by infima. With such a notion in place it would be easy to see that 
    $(\psi_t)$ is an admissible $(\Qtt,\alpha)$-sequential test if and only if $(\psi_t)$ is a closed min-martingale such that $\sup_{t \in \NN_0} \psi_t$ is $\{0,1\}$-valued and $\Qtt(\sup_{t \in \NN_0} \psi_t = 1) = \alpha$. 
\end{remark}

Unfortunately, a version of the first statement in Theorem~\ref{thm:pointwise} that is phrased as follows---`if $(\Pval_t)$ is admissible then
$\Pval_t =  \inf_{s \leq t} 1/ {M_s}$
for all $t \in \NN_0$, where $(M_t)$ is a $\Qtt$-NM'---is incorrect. Below is a simple counterexample.
\begin{example}
Let $\Pval_0=U$,  and $\Pval_0 = \Pval_1 = \Pval_2 = \Pval_3 = \dots$. Then $(\Pval_t)$ is admissible since it cannot be improved without violating uniformity. However, since the inverse of a uniform random variable is not integrable, we cannot find a NM $(M_t)$ that yields $(\Pval_t)$; indeed, no nonnegative integrable random variable $N_0$ can yield $\Pval_0 = 1/N_0$. 
% Indeed, this is consistent with Proposition~\ref{prop:construct-e-vals} as the function $f: u \mapsto 1/u$ for $u \in [0,1]$ is not integrable.
\end{example}
The above example demonstrates that max-martingales, not the `usual' martingales, are the right mathematical construct to deal with p-processes.
We remark that if we are allowed to construct continuous-time processes, then one can work with usual martingales, see~\citet[Theorem 2]{shafer_test_2011}. 
However, we do obtain the following corollary of Theorem~\ref{thm:pointwise}\ref{thm:pointwise:1}.

\begin{corollary} \label{C:200906}
Let $(M_t)$ be a $\Qtt$--martingale with $M_0 > 0$ and let $F$ denote the distribution of $\inf_{t \in \NN_0} 1/ M_t$. If $F$ is atomless, then 
$\Pval_t := F(\inf_{s \leq t} 1/M_s)$ is an admissible p-process.
\end{corollary}
\begin{proof}
We will make use of the fact that the conditional supremum commutes with continuous nondecreasing functions: $\bigvee[ f(Y) \mid \Gcal ] = f( \bigvee[ Y \mid \Gcal ] )$ for every continuous nondecreasing function $f$; we will use this with $f=F$. Combining this with the max-martingale property of the reciprocal of the running supremum of an NM (see \eqref{eq_inf_M_recip_mg}) we get, 
\[
\Pval_t = F\left(\inf_{s\le t} \frac{1}{M_s} \right) = F\left( {\textstyle\bigvee_\Qtt} \Big[ \inf_{s\in\NN_0} \frac{1}{M_s} \Big| \Fcal_t \Big] \right) =  {\textstyle\bigvee_\Qtt} \Big[ F\Big(\inf_{s\in\NN_0} \frac{1}{M_s} \Big) \Big| \Fcal_t \Big] = {\textstyle\bigvee_\Qtt} \Big[ \inf_{s\in\NN_0} \Pval_s \Big| \Fcal_t \Big].
\]
The last equality follows because $\inf_{s\in\NN_0} \Pval_s = F(\inf_{s \in \NN_0} 1 / M_s)$, which is uniformly distributed since $F$ is atomless. We are now in a position to apply Theorem~\ref{thm:pointwise}\ref{thm:pointwise:1} to conclude that $(\Pval_t)$ is admissible.
\end{proof}

Let us now provide an example of a closed $\Qtt$-MM that satisfies $F(\inf_{t \in \NN_0} \Pval_t) = \inf_{t \in \NN_0} \Pval_t$ in the notation of Theorem~\ref{thm:pointwise}\ref{thm:pointwise:1}, but is not admissible.
\begin{example}[The gap between sufficient and admissible conditions in Theorem~\ref{thm:pointwise}\ref{thm:pointwise:1}]\label{eg:inadmissible-NM-based-pvalue}
This example shows that simply using $F(\inf_{s\leq t} 1/ M_s)$ for a $\Qcal$-NM $(M_t)$ and $F$ as in Corollary~\ref{C:200906} does not typically yield an admissible p-process.
It also provides an example for the gap between the sufficient and admissible conditions in Theorem~\ref{thm:pointwise}\ref{thm:pointwise:1}.
Define the martingale $(M_t)$ by $M_t :=  1+ \1_{U \leq 1/2}$ for all $t \in \NN_0$.  Then $F (\inf_{s \leq t} 1/{M_s}) = \inf_{s \leq t} 1/{M_s}$
is an inadmissible p-process. (It is, however, a p-process despite $\EE_\Qtt[M_0] > 1$.)
To see this, define a p-process $(\Pval_t)$ by $\Pval_t = U$ for all $t  \in \NN_0$. Then $(\Pval_t)$ strictly dominates $\inf_{s \leq t} 1/ {M_s}$ because
\[
    U \leq \frac{1}{2} \1_{U \leq 1/2} +  \1_{U > 1/2} = \inf_{s \leq t}  \frac1{M_s}.
\]
\end{example}

\subsection{Necessary and sufficient conditions for p-processes (proof)}

\begin{proof}[Proof of Theorem~\ref{thm:pointwise}\ref{thm:pointwise:1}]
We start with the necessary conditions for pointwise admissibility of a p-process.  To this end, let $(\Pval_t)$ be a $\Qtt$-admissible p-process, which must necessarily be nonincreasing by Remark~\ref{rem:running-extremum}. 
First define
\[
\bar\Pval := \inf_{t\in\NN_0}\Pval_t = \lim_{t \to \infty} \Pval_t,
\]
and let $F$ be the distribution function of $\bar\Pval$. Since $(\Pval_t)$ is valid, $\bar \Pval$ is stochastically larger than uniform and so $F(x) \leq x$ for all $x \in [0,1]$.  For later use, let us observe that $F$ is right-continuous, hence
\begin{align} \label{eq:200906}
\lim_{t \to \infty} F(\Pval_t) =  F(\bar \Pval).
\end{align}

%Recall now the general fact that if a random variable $Y$ with distribution function $G$ is stochastically larger than uniform, then so is $G(Y)$. We deduce that $F(\bar\Pval)$ is also stochastically larger than uniform. 

We now define
\[
\Pval_t' := {\textstyle\bigvee_\Qtt} \left[\left. F(\bar \Pval) \right| \Fcal_t \right] .
\]
%Since $\inf_{t \in \NN_0} \Pval_t' \geq F(\bar \Pval)$, which we argued to be stochastically larger than uniform, we obtain that $(\Pval_t')$ is valid.
Since $F$ is a nondecreasing function, we have $F(\bar\Pval) \leq F(\Pval_t)$ for all $t \in \NN_0$. By definition of conditional supremum, $\Pval'_t$ is the \emph{smallest} $\Fcal_t$-measurable random variable with this property; therefore, 
$\Pval'_t \le F(\Pval_t)$ for all $t \in \NN_0$. Since we also have $F(x) \leq x$ for all $x \in [0,1]$ we get $\Pval'_t \le\Pval_t$ for all $t \in \NN_0$. But $(\Pval_t)$ is admissible by assumption, so we must in fact have the equality $\Pval'_t = \Pval_t$ for all $t \in \NN_0$.

We have now argued $F(\bar \Pval) \leq \Pval_t' = \Pval_t \leq F(\Pval_t)$. Taking now limits in $t$ and recalling \eqref{eq:200906}, we get 
$F(\bar \Pval) \leq \bar \Pval \leq \lim_{t \to \infty} F(\Pval_t) = F(\bar \Pval)$, thus allowing us to conclude that $F(\bar \Pval) = \bar \Pval$ and that $(\Pval_t)$ is closed.
% $\lim_{t \to \infty} F(\Pval_t) =  F(\bar \Pval)$ and $\bar \Pval = \lim_{t \to \infty} \Pval_t = F(\bar \Pval)$ and 
% $(\Pval_t)$ is closed.  
This yields the necessary conditions of the theorem. 

Let us now discuss the sufficient conditions for pointwise admissibility of a p-process. To this end, let $(\Pval_t)$ be 
a closed MM such that $\inf_{t \in \NN_0} \Pval_t$ is uniformly distributed.  It is then clear that $(\Pval_t)$ is valid.
Consider now an arbitrary p-process $(\Pval_t')$ with
$\Pval_t' \le \Pval_t$ for all $t \in \NN_0$. We must argue that we have equality. 
We clearly have $\inf_{t\in\NN_0} \Pval_t' \le \inf_{t\in\NN_0} \Pval_t$. The validity of $(\Pval_t')$ implies moreover that $\inf_{t\in\NN_0} \Pval_t'$ stochastically dominates a uniform. This now directly yields that we have indeed $\inf_{t\in\NN_0} \Pval_t' = \inf_{t\in\NN_0} \Pval_t =: \bar \Pval$, which is uniform.
 By definition of max-martingales, $\Pval_t$ is the smallest $\Fcal_t$-measurable upper bound on $\bar \Pval$; since $\Pval_t'$ is another such bound, we must have $\Pval_t' \ge \Pval_t$ for all $t \in \NN_0$.  This proves that $(\Pval_t)$ is admissible.
 \end{proof}
 
 \subsection{Necessary and sufficient conditions for e-processes (proof)} \label{SS:6.2}

\begin{proof}[Proof of Theorem~\ref{thm:pointwise}\ref{thm:pointwise:2}]
%\fbox{fix with randomization by adding initial time -1, defining $e_{-1}=1$, stopping times if they could be -1 then they must be -1 identically, etc.}
Again, let us start with the necessary conditions for pointwise admissibility of an e-process. 
 To this end, let us fix an admissible $(\Eval_t)$. Thanks to the implication \ref{L:e-process_definitions.6} $\Rightarrow$ \ref{L:e-process_definitions.6} of Lemma~\ref{L:e-process_definitions}, there is a $\Qtt$-NM $(M_t)$ that dominates $(\Eval_t)$. Since $(\Eval_t)$ was assumed admissible we get $\Eval_t = M_t$
for all $t \in \NN_0$.
 Moreover, we can assume that $\EE_{\Qtt}[M_0]=1$, else we can replace $(M_t)$ by $(M_t + 1 - \EE_{\Qtt}[M_0])$, which is again an e-process (because it remains a nonnegative martingale, now with initial value one, and the optional stopping theorem applies). 

Let us now discuss the sufficient conditions for pointwise admissibility of a NM $(\Eval_t)$ with $\EE_{\Qtt}[\Eval_0] = 1$. First of all, the optional stopping theorem yields that $(\Eval_t)$ is an e-process.  Consider now some e-process $(\Eval_t')$ for $\Qtt$ with $\Eval_t' \geq \Eval_t$. Since $1 = \EE_{\Qtt}[\Eval_t] \leq \EE_{\Qtt}[\Eval_t'] \leq 1$, we then have $\Eval_t = \Eval_t'$ for each $t \in \NN_0$, yielding the admissibility of $(\Eval_t)$, hence the assertion.
\end{proof}

\subsection{Necessary and sufficient conditions for sequential tests (proof)} \label{SS:6.3}

\begin{proof}[Proof of Theorem~\ref{thm:pointwise}\ref{thm:pointwise:3}]
   Let us start with the necessary conditions for pointwise admissibility of a sequential test. 
   Recall that by assumption, $(\psi_t)$ satisfies
\[
\bar \alpha := \Qtt(\exists t \in \NN_0: \psi_t = 1) \leq \alpha.
\]
Define now
\begin{align*}
   \psi'_t := \1_{U \leq \alpha - \bar \alpha} + \psi_t \1_{U > \alpha - \bar \alpha}, \qquad t \in \NN_0.
   \end{align*}
Note that $\psi'_t \geq \psi_t$ for all $t \in \NN_0$ and $(\psi'_t)$ is again a sequential test. If $\bar \alpha < \alpha$ then indeed $(\psi'_t)$ strictly dominates $(\psi_t)$, in contradiction to the admissibilty of $(\psi_t)$. Hence we may assume that $\bar \alpha = \alpha$.  

Define now the Doob-L\'evy  martingale $(M_t)$ by
\begin{equation*}
M_t:= \frac{ \Qtt(\exists s \in \NN_0: \psi_s = 1 ~|~ \Fcal_t) }{\Qtt(\exists s \in \NN_0: \psi_s = 1) } = \frac{ \Qtt(\exists s \in \NN_0: \psi_s = 1 ~|~ \Fcal_t) }{\alpha}.
\end{equation*}
Note that $\EE_\Qtt[M_0] =1$ and if there exists a time $\tau$ at which $\psi_\tau =1$, then $M_t = 1/\alpha$ for any $t\geq \tau$. So, $M_\infty \in \{0,1/\alpha\}$. 
Define next $\widetilde \psi_t := \1_{M_t \geq 1/\alpha}$. By Ville's inequality, $(\widetilde \psi_t)$ is a sequential test that dominates $(\psi_t)$. Since the latter was assumed admissible, we have established $\psi_t = \widetilde \psi_t$ for all $t \in \NN_0$.

Consider now a NM $(M_t)$ with no overshoot at $1/\alpha$ and $M_\infty \in \{0,1/\alpha\}$ and define
 $\psi_t := \1_{\sup_{s\leq t}M_s \geq 1/\alpha}$. By Ville's inequality, $(\psi_t)$ is a sequential test.
 Consider next some sequential test $(\psi_t')$ with $\psi_t' \geq \psi_t$ and fix some $t^* \in \NN_0$. Since $\EE_\Qtt[\psi_\infty] = \alpha$ we know that $\psi_{t^*}' = 1$ only on the event $\{\psi_\infty = 1\} = \{\tau < \infty\}$, where $\tau :=  \inf\{t \in \NN_0: M_t \geq 1/\alpha\}$. 
Hence, $\psi_{t^*}' = 1$ implies that $M_t \geq 1/\alpha$; otherwise 
 the martingale property of $(M_t)$ would be contradicted. Thus $(\psi_t)$ is indeed $\Qcal$-admissible, concluding the proof of the statement.
\end{proof}

\section{Reducing admissible composite inference to the pointwise case}\label{sec:composite-via-pointwise}

To build intuition towards composite admissibility, we begin with a question on validity: is there a systematic way to construct tools for valid (potentially inadmissible) inference in composite settings?

\subsection{Necessary and sufficient conditions for valid (composite) inference}

% We describe how one can perform inference in composite settings by using simpler `pointwise' building blocks. 
The following observations are straightforward and arguably well-known in some form or another, but are nevertheless useful to spell out formally in order to lay the path for the admissibility results.

\begin{proposition}[Pointwise-to-composite validity]\label{prop:pointwise-to-composite}
The following statements lay out necessary and sufficient conditions that connect validity in pointwise and  composite settings.
\begin{enumerate}[label={\rm(\arabic{*})}, ref={\rm(\arabic{*})}] 
    % \item $(M_t)$ is a $\Pbb$-NSM if and only if $M_t = \mathrm{ess}\inf_{\Ptt \in \Pbb} M_t^\Ptt$ for all $t$, where $(M_t^\Ptt)$ is some $\Ptt$-NSM.
    \item\label{prop:pointwise-to-composite:1} $(\Pval_t)$ is a $\Qcal$-p-process if and only if $\Pval_t \geq \Pval_t^\Qtt$, $\Qtt$-a.s., for all $t$ and $\Qtt \in \Qcal$, where $(\Pval_t^\Qtt)$ is some p-process for $\Qtt$.
    \item\label{prop:pointwise-to-composite:2} $(\Eval_t)$ is a $\Qcal$-e-process if and only if $\Eval_t \leq \Eval_t^\Qtt$, $\Qtt$-a.s., for all $t$ and $\Qtt \in \Qcal$, where $(\Eval_t^\Qtt)$ is some e-process for $\Qtt$.
    \item\label{prop:pointwise-to-composite:3} $(\psi_t)$ is a $(\Qcal,\alpha)$-ST if and only if $\psi_t \leq \psi_t^\Qtt$, $\Qtt$-a.s., for all $t$ and $\Qtt \in \Qcal$, where $(\psi_t^\Qtt)$ is some $(\Qtt,\alpha)$-ST.
\end{enumerate}
\end{proposition}
\begin{proof}
We only prove \ref{prop:pointwise-to-composite:1}, the other two assertions are argued analogously. Suppose we are given that $(\Pval_t)$ is a $\Qcal$-p-process. Then, by definition, $\Pval_t$ is a $\Qtt$-p-process for every $\Qtt \in \Qcal$; so choosing $(\Pval^\Qtt_t) := (\Pval_t)$ itself, we have proved the `only if' direction. 
For the other direction, suppose for every $\Qtt\in\Qcal$ we are given a p-process $(\Pval^\Qtt_t)$ for $\Qtt$, and that $(\Pval_t)$ satisfies $\Pval_t \ge \Pval_t^\Qtt$, $\Qtt$-almost surely, for all $t$. Let $\Qtt^* \in \Qcal$ be some true (arbitrary) data-generating distribution. We must argue that $(\Pval_t)$ is a $\Qtt^*$-p-process. Indeed, for any $\alpha \in [0,1]$, we have
\[
\Qtt^*(\exists t \in \NN: \Pval_t \leq \alpha) \leq \Qtt^*\left(\exists t \in \NN: \Pval^{\Qtt^*}_t \leq \alpha\right) \leq \alpha,
\]
where the first inequality follows because $\Pval_t \geq \Pval_t^{\Qtt^*}$, and the second inequality follows because $(\Pval_t^{\Qtt^*})$ is a $\Qtt^*$-p-process by assumption. This concludes the proof.
\end{proof}

The proposition provides a generic reduction from the composite setting to the pointwise setting for performing \emph{valid} inference, but we can deduce a similar result for admissible inference, presented later. While Proposition~\ref{prop:pointwise-to-composite} forms a useful building block, it makes no mention of martingales. 
% Nevertheless, we now have the appropriate context in place to summarize some known results.
The notions of essential supremum and essential infimum used below are reviewed in Appendix~\ref{sec:esssup}. Note that we will need the additional restriction that $\Qcal$ is locally dominated in order for these essential extrema to be well defined.

\begin{corollary}[Pointwise supermartingales are sufficient for composite validity]\label{cor:martingale-sufficiency-lemma}
Let  $\Qcal$ be locally dominated. Then the following statements demonstrate how supermartingales suffice for valid sequential inference. 
\begin{enumerate}[label={\rm(\arabic{*})}, ref={\rm(\arabic{*})}] 
\item\label{cor:martingale-sufficiency-lemma:1} If $\Pval_t = \esssup_{\Qtt \in \Qcal} 1\wedge \inf_{s \leq t} 1/{N^\Qtt_s}$, where $(N^\Qtt_t)$ is upper bounded by a $\Qtt$-NSM, then $(\Pval_t)$ is a $\Qcal$-p-process.
    \item\label{cor:martingale-sufficiency-lemma:2} If $\Eval_t = \mathrm{ess}\inf_{\Qtt \in \Qcal} N_t^\Qtt$, where $(N_t^\Qtt)$ is upper bounded by a $\Qtt$-NSM, then $(\Eval_t)$ is $\Qcal$-e-process.
    \item\label{cor:martingale-sufficiency-lemma:3} If $\smash{\psi_t = \essinf_{\Qtt \in \Qcal} \sup_{s \leq t} \1_{N^\Qtt_s \geq 1/\alpha}}$, where $(N_t^\Qtt)$ is upper bounded by a $\Qtt$-NSM, then $(\psi_t)$ is a $(\Qcal,\alpha)$-ST.
\end{enumerate}
Above, all $\Qtt$-NSMs have initial expected value at most one, and $(N^\Qtt_t)$ is assumed nonnegative.
\end{corollary}

% \smallskip
This corollary follows from Proposition~\ref{prop:pointwise-to-composite} and so its proof is omitted; see also Remark~\ref{rem:running-extremum}.

\subsection{Necessary conditions for admissible composite inference}

Not all constructions using martingales are admissible: we provide an analog to Proposition~\ref{prop:pointwise-to-composite} for admissibility. 

\begin{proposition}[Admissible composite tests must aggregate admissible pointwise tests]\label{prop:admissible-pointwise-composite-aggregate}
Let $\Qcal$ be locally dominated.
The following statements show how composite admissible instruments must aggregate (some) pointwise admissible instruments.
\begin{enumerate}[label={\rm(\arabic{*})}, ref={\rm(\arabic{*})}] 
    \item\label{prop:admissible-pointwise-composite-aggregate:1} 
    If $(\Pval_t)$ is $\Qcal$-admissible, then $\Pval_t = \esssup_{\Qtt \in \Qcal} \Pval_t^\Qtt$ for all $t$, where $(\Pval_t^\Qtt)$ is $\Qtt$-admissible.
    \item\label{prop:admissible-pointwise-composite-aggregate:2}  If $(\Eval_t)$ is $\Qcal$-admissible, then $\Eval_t = \essinf_{\Qtt \in \Qcal} \Eval_t^\Qtt$ for all $t$, where $(\Eval_t^\Qtt)$ is $\Qtt$-admissible.
    \item\label{prop:admissible-pointwise-composite-aggregate:3}  If $(\psi_t)$ is $\Qcal$-admissible, then $\smash{\psi_t = {\essinf}_{\Qtt \in \Qcal}~\psi^\Qtt_t}$ for all $t$, where $(\psi_t^\Qtt)$ is $\Qtt$-admissible.
    % (Or: $(\psi_t)$ is $\Qcal$-admissible only if it is $\Qtt$-admissible? False!)
%     \fbox{I suggest instead:}
%     \item\label{prop:admissible-pointwise-composite-aggregate:4}  If $(\CI_t)$ is  $\Pcal$-admissible, then $\CI_t = \bigcup_{\Ptt \in \Pbb} \CI_t^\Ptt$ for all $t$, where $(\CI^\Ptt_t)$ is $\Ptt$-admissible.
\end{enumerate}
\end{proposition}
\begin{proof}
Let $(\Pval_t)$ denote a $\Qcal$-admissible p-process. For each $\Qtt \in \Qcal$, let $(\Pval^\Qtt_t)$ be a $\Qtt$-admissible p-process that dominates $(\Pval_t)$. Such $(\Pval^\Qtt_t)$ exists thanks to Proposition~\ref{P:200905}.
Let us now define
$\Pval_t' := \esssup_{\Qtt\in\Qcal} \widetilde\Pval^\Qtt_t$, which is a $\Qcal$-p-process thanks to  Proposition~\ref{prop:pointwise-to-composite}\ref{prop:pointwise-to-composite:1}. Clearly, we have 
$\Pval_t' \geq \Pval_t$ for all $t \in \NN_0$. Since $(\Pval_t)$ is $\Qcal$-admissible, we indeed have $\Pval_t' = \Pval_t$ for all $t \in \NN_0$, yielding the assertion for p-processes. The assertions for e-processes and sequential tests are shown in the same manner.
\end{proof}

The following corollary describes the restrictions that every admissible construction necessarily satisfies.

\begin{corollary}[Pointwise martingales are necessary for composite admissibility]\label{cor:martingale-building-block-lemma}
Let $\Qcal$ be locally dominated. Then the following statements demonstrate how martingales underpin all admissible constructions.
\begin{enumerate}[label={\rm(\arabic{*})}, ref={\rm(\arabic{*})}] 
    \item\label{cor:martingale-building-block-lemma:1} If $(\Pval_t)$ is $\Qcal$-admissible, then $(\Pval_t)$ is nonincreasing  and $\Pval_t = \esssup_{\Qtt \in \Qcal} \Pval_t^\Qtt$ for all $t$, where $(\Pval_t^\Qtt)$ is a closed $\Qtt$-MM.
    \item\label{cor:martingale-building-block-lemma:2} If $(\Eval_t)$ is $\Qcal$-admissible, then $\Eval_t = \essinf_{\Qtt \in \Qcal} M_t^\Qtt$ for all $t$, where $(M_t^\Qtt)$ is a $\Qtt$-NM with $\EE_\Qtt [M_0^\Qtt]=1$.
    \item\label{cor:martingale-building-block-lemma:3} If $(\psi_t)$ is $\Qcal$-admissible, then $\smash{\psi_t = \essinf_{\Qtt \in \Qcal}\sup_{s \leq t}\1_{M^\Qtt_s \geq 1/\alpha}}$, where $(M_t^\Qtt)$ is a $\Qtt$-NM with $\EE_\Qtt [M_0^\Qtt]=1$, $M^\Qtt_\infty \in \{0,1/\alpha\}$ and $(M^\Qtt_t)$ has no overshoot at level $1/\alpha$.
\end{enumerate}
\end{corollary}

The above statements are a direct consequence of Proposition~\ref{prop:admissible-pointwise-composite-aggregate} and Theorem~\ref{thm:pointwise}.

\subsection{Sufficient conditions for admissible (composite) inference}

% Another central contribution of the paper involves identifying sufficient conditions for admissibility, where again martingales turn out to be central. 

The next proposition argues that it suffices to consider only a subset of $\Qcal$ when constructing $\Qcal$-admissible p-processes, e-processes, or sequential tests. Note that we do not require $\Qcal$  to be locally dominated below.
\begin{proposition}[Pointwise-to-composite admissibility]\label{P:200823}
Assume there exists a `reference family' $(\Qtt_i)_{i \in I} \subset \Qcal$ such that
%$\inf_{s \leq t} \Pval_t$ is $\Qcal_i$-uniform, a $\Qtt_i$ max-martingale, 
 for each $t$ and $A \in \Fcal_t$, 
\begin{quote}
if there exists $\Qtt \in \Qcal$ with $\Qtt(A) > 0$, then there exists $i \in I$ with $\Qtt_i(A) > 0$. 
\end{quote}
Then we have the following.
\begin{enumerate}[label={\rm(\arabic{*})}, ref={\rm(\arabic{*})}] 
    \item\label{P:200823:1} If $(\Pval_t)$ is a $\Qcal$-p-process and $(\Pval_t)$ is $\Qtt_i$-admissible for each $i \in I$, then $(\Pval_t)$ is $\Qcal$-admissible.
    \item\label{P:200823:2} If $(\Eval_t)$ is a $\Qcal$-e-process and $(\Eval_t)$ is $\Qtt_i$-admissible for each $i \in I$, then $(\Eval_t)$ is $\Qcal$-admissible.
    \item\label{P:200823:3} If $(\psi_t)$ is a $(\Qcal, \alpha)$-ST and $(\psi_t)$ is $\Qtt_i$-admissible for each $i \in I$, then $(\psi_t)$ is $\Qcal$-admissible.
\end{enumerate}
\end{proposition}
\begin{proof}
Let us only argue here the case of e-processes. The other cases follow in exactly the same manner.
    Assume that there exists an e-process $(\Eval_t')$ for $\Qcal$
    such that $\Qtt(\Eval'_t \geq \Eval_t) = 1$ for all $t \in \NN_0$ and all $\Qtt\in\Qcal$, and that $\Qtt^*(\Eval'_t > \Eval_t) > 0$ for some $t \in \NN_0$, and some 
    $\Qtt^*\in\Qcal$. By assumption, there exists some $i \in I$ such that $\Qtt_i(\Eval'_t > \Eval_t)>0$. Since $(\Eval_t)$ is assumed to be $\Qtt_i$-admissible, we get a contradiction. 
\end{proof}

Of course, two special cases are found at the extremes: when the reference family is a singleton, it means there is a common reference measure $\Rtt$, and when the reference family is $\Qcal$ itself, the proposition is vacuous.
The proposition is particularly useful in the first case;
% if, for instance, each $\Qtt \in \Qcal$ is locally absolutely continuous with respect to some $\Qtt^* \in \Qcal$; 
then to get an admissible e-process, for example, it suffices to construct a $\Qcal$-NSM $(M_t)$ that is also a $\Qtt^*$-NM, thanks to Proposition~\ref{P:200823}\ref{P:200823:2} and Theorem~\ref{thm:pointwise}\ref{thm:pointwise:2}. The following example demonstrates one such setting.
\begin{example}
    Recalling notation from Section~\ref{sec:gaussian}, let $\mu^m \in \Gcal^m$ denote the measure under which $(X_t)$ is i.i.d.~Gaussian with unit variance and mean $m$, and consider $\Qcal := \{\mu^m : m \leq 0\}$. 
    % let $\Qcal=\bigcup_{m \leq 0} \Gcal^m$, meaning that the observations $X_t$ have (conditionally) nonpositive mean.
    Then $G_t := \exp(\sum_{s \leq t} X_s -  t/2)$ is not a $\Qcal$-NM---it is a $\mu^0$-NM when $X_t$ is standard Gaussian, but is a $\mu^m$-NSM for $m < 0$. Nevertheless, $(G_t)$ is a $\Qcal$-admissible e-process. The reason is that $(G_t)$, being a $\mu^0$-NM, is immediately $\mu^0$-admissible, and the singleton reference family $\{\mu^0\}$  satisfies the local absolute continuity condition required to invoke Proposition~\ref{P:200823}\ref{P:200823:2}.
\end{example}

\begin{corollary}[Composite martingales are sufficient for composite admissibility]\label{cor:martingale-suff-composite} 
Consider a general composite family $\Qcal$. 
\begin{enumerate}[label={\rm(\arabic{*})}, ref={\rm(\arabic{*})}] 
    \item\label{cor:martingale-suff-composite:1} 
     $(\Pval_t)$ is $\Qcal$-admissible if it is a closed $\Qcal$-MM 
         and $\inf_{t \in \NN} \Pval_t$ is  $\Qtt$-uniformly distributed for every $\Qtt \in \Qcal$.
    \item\label{cor:martingale-suff-composite:2}  
    $(\Eval_t)$ is $\Qcal$-admissible if it is a $\Qcal$-NM with $\EE_{\Qtt}[\Eval_0] = 1$ for all $\Qtt \in \Qcal$.
    \item\label{cor:martingale-suff-composite:3}  $(\psi_t)$ is $\Qcal$-admissible if $\smash{\psi_t = \sup_{s \leq t} \1_{M_s \geq 1/{\alpha}}}$, where $(M_t)$ is a $\Qcal$-NM with $M_0=1$, $M_\infty\in\{0, 1/\alpha\}$, $\Qtt$-almost-surely, for every $\Qtt \in \Qcal$, and no overshoot at $1/{\alpha}$.
    % \item {\color{blue} Conjecture: Suppose $\Qcal$ is such that only the empty set is polar. If $(E_t)$ is $\Qcal$-valid, and it is an increasing convex function of the data, and it is a $Q$-NM for some $Q\in\Qcal$, then it is admissible.}
\end{enumerate}
\end{corollary}
This corollary is again a direct consequence of Proposition~\ref{P:200823} and  Theorem~\ref{thm:pointwise}. Recall that as in Corollary~\ref{C:200906}, a sufficient condition for $(\Pval_t)$  to be a closed $\Qtt$-MM and $\inf_{t \in \NN_0} \Pval_t$ be $\Qtt$-uniformly distributed, for some fixed $\Qtt \in \Qcal$, is that the p-process has the representation $\Pval_t = F^\Qtt(\inf_{s \leq t} 1/ M_s^\Qtt)$, where $(M_t^\Qtt)$ is a $\Qtt$-NM and $\inf_{s \in \NN_0}  1/M_s$ has an atomless distribution function $F^\Qtt$ under $\Qtt$.
    
 As an immediate consequence of Corollary~\ref{cor:martingale-suff-composite}\ref{cor:martingale-suff-composite:2}, recall the Gaussian example in Section~\ref{sec:gaussian}, and consider testing if the underlying mean is zero. Since $(G_t^0)$ is a $\Gcal^0$-NM, it is also a $\Gcal^0$-e-process, and hence a $\Gcal^0$-admissible e-process when testing against, for example, $\Pcal = \bigcup_{m \in \RR} \Gcal^m$.
%NOT PROVEN: However, $1/\sup_{s \leq t} G^0_s$ is not a $\Gcal^0$-admissible p-value; see Subsection~\ref{sec:sufficient-p} for a discussion on the role of randomization in enabling admissibility. 

\subsection{Necessary and sufficient conditions for confidence sequences}\label{subsec:CS}

Proposition~\ref{prop:construct-CS} shows that we can construct a confidence sequence by inverting a family of sequential tests. We now show that their admissibility is also tightly linked to those of the underlying tests. 

\begin{theorem}\label{thm:ST/CS-admissibility}
Recall that $\Pcal^\gamma := \{\Ptt \in \Pcal: \phi(\Ptt) = \gamma\}$.
If $(\psi^\gamma_t)$ is an admissible  $(\Pcal^\gamma,\alpha)$-sequential test for each $\gamma \in \Zcal$ then $\{\gamma \in \Zcal: \psi_t^\gamma=0\}$ is an admissible $(\phi, \Pcal,\alpha)$-confidence sequence.
Similarly, if  $(\CI_t)$ is an admissible $(\phi, \Pcal,\alpha)$-confidence sequence, then $\psi^\gamma_t := \1_{\gamma \notin \CI_t}$ yields an admissible $(\Pcal^\gamma,\alpha)$-sequential test for each $\gamma \in \Zcal$, so that $\CI_t = \bigcup_{\gamma \in \Zcal}  \{\gamma \in \Zcal: \psi^{\gamma}_t = 0\}$.  As a result, we can infer the following.
\begin{enumerate}[label={\rm(\arabic{*})}, ref={\rm(\arabic{*})}] 
     \item\label{thm:ST/CS-admissibility-1} (Validity) If the process $(N_t^\Ptt)$ is upper bounded by a $\Ptt$-NSM with initial expected value one, then $\smash{\CI_t := \bigcup_{\Ptt \in \Pbb}  \{\phi(\Ptt): \sup_{s\le t} N_s^\Ptt < 1/{\alpha}\}}$,  is a $(\phi,\Pbb,\alpha)$-CS.
     \item\label{thm:ST/CS-admissibility-2} (Admissibility)
If  $(M^\gamma_t)$ is a $\Pcal^\gamma$-NM with $M^\gamma_0=1$, $M^\gamma_\infty \in \{0,1/\alpha\}$, $\Ptt$-almost surely, for every $\Ptt \in \Pcal^\gamma$, and has no overshoot at $1/\alpha$, then $\CI_t := \{\gamma \in \Zcal: \sup_{s\leq t} M^\gamma_s < 1/\alpha=0\}$ is $\Pcal$-admissible.
\end{enumerate}

\noindent Assume now that $\Pcal^\gamma$ is locally dominated for each $\gamma$, and that  $(\CI_t)$ is $\Pcal$-admissible.
Then, for all $t \in \NN_0$, we must be able to represent $\smash{\CI_t = \bigcup_{\Ptt \in \Pbb} \{\phi(\Ptt): \psi^\Ptt_t = 0\}}$ where $(\psi^\Ptt_t)$ is $\Ptt$-admissible. Moreover, we can write
 \[\CI_t = \bigcup_{\Ptt \in \Pbb} \left\{\phi(\Ptt): \sup_{s\le t} M_s^\Ptt < \frac{1}{{\alpha}}\right\},\] 
 where $(M_t^\Ptt)$ is a $\Ptt$-NM that has no overshoot at level $1/\alpha$, with $M^\Ptt_0=1$ and $M^\Ptt_\infty \in \{0,1/\alpha\}$.
\end{theorem}
\begin{proof}
For the first statement of the theorem, note that $\CI_t = \{\gamma \in \Zcal: \psi_t^\gamma=0\}$ yields a $(\phi, \Pcal,\alpha)$-confidence sequence by Proposition~\ref{prop:construct-CS}. Suppose for contradiction that $(\CI'_t)$ is another $(\phi, \Pcal,\alpha)$-confidence sequence that witnesses the inadmissibility of $(\CI_t)$. Proposition~\ref{prop:construct-seq-tests}\ref{prop:construct-seq-tests:3} then yields a corresponding family $\{(\eta_t^\gamma)\}_{\gamma \in \Zcal}$ of sequential tests.
The inadmissibility of $(\CI_t)$ then yields some $\gamma \in \Zcal$ such that $(\eta_t^\gamma)$ strictly dominates $(\psi_t^\gamma)$, a contradiction to the assumption that $(\psi^\gamma_t)$ is admissible.
The second statement follows in exactly the same way, again by an application of Propositions~\ref{prop:construct-CS} and \ref{prop:construct-seq-tests}\ref{prop:construct-seq-tests:3}.
Statements~\ref{thm:ST/CS-admissibility-1} and \ref{thm:ST/CS-admissibility-2}  are direct corollaries of combining the first part of the theorem with Corollary~\ref{cor:martingale-sufficiency-lemma}\ref{cor:martingale-sufficiency-lemma:3} and Corollary~\ref{cor:martingale-suff-composite}\ref{cor:martingale-suff-composite:3}.
Assume now that $\Ptt^\gamma$ is locally dominated, for each $\gamma$, and that  $(\CI_t)$ is $\Pcal$-admissible. The statement then follows from Proposition~\ref{prop:admissible-pointwise-composite-aggregate}\ref{prop:admissible-pointwise-composite-aggregate:3} and Corollary~\ref{cor:martingale-building-block-lemma}\ref{cor:martingale-building-block-lemma:3}.
\end{proof}

This and the previous section argued in detail that restricting our attention to constructions based on NMs (not NSMs!) does not hurt us: these are \emph{universal} constructions. Indeed, if one is presented with a p-process, e-process, sequential test, or confidence sequence constructed in some arbitrary fashion, we show that one can always uncover a `hidden' underlying NM, such that applying Ville's inequality or the optional stopping theorem yields an instrument that is at least as good as the original one.

%%% Local Variables:
%%% mode: latex
%%% TeX-master: "main"
%%% End:

\section{A deeper investigation of admissibility}\label{sec:sufficient}

\subsection{The gap between necessary and sufficient conditions (composite)}

Recall that we were able to crisply summarize the  necessary conditions for admissibility using martingales in Corollary~\ref{cor:martingale-building-block-lemma} and sufficient conditions in Corollary~\ref{cor:martingale-suff-composite}. The following discussion probes at the gap between necessary and sufficient conditions for composite admissibility, in order to demonstrate that the gap is real. We begin with two instructive examples that demonstrate that the necessary conditions of Corollary~\ref{cor:martingale-building-block-lemma} are not actually sufficient for admissibility. 

\begin{example}[Necessary conditions for Corollary~\ref{cor:martingale-building-block-lemma}\ref{cor:martingale-building-block-lemma:1} are not sufficient]
    Assume that $\Qcal$ is the set of probability measures under which $X_1$ is Bernoulli  and 
    $X_2 = X_3 = \ldots = 0$.  
    Consider now
\[
\Pval_1  := \1_{X_1=0}  U + \1_{X_1 =1} \sqrt{U}
\]
and $\Pval_2 = \Pval_3 = \ldots = \Pval_1$. Note that $\Pval_1 \geq U$ by construction and hence is valid. 
Then $(\Pval_t)$ is a p-process and satisfies Corollary~\ref{cor:martingale-building-block-lemma}\ref{cor:martingale-building-block-lemma:1} and Proposition~\ref{prop:admissible-pointwise-composite-aggregate}\ref{prop:admissible-pointwise-composite-aggregate:1}. Here 
\[
    \Pval_1^\Qtt = \Pval_2^\Qtt = \Pval_3^\Qtt = \ldots = F_\Qtt(\Pval_1) = 
        (1-q) \Pval_1 + q \Pval_1^2, 
\]
where $q := \Qtt(X_1 = 1) \in [0,1]$, for each $\Qtt \in \Qcal$ and $F_\Qtt$ is the $\Qtt$-distribution function of $\Pval_1$.  Now observe that by definition $\Pval_1^\Qtt$ is $\Qtt$-uniform and considering small $q$'s it is easy to see that $\Pval_1 = \esssup_{\Qtt \in \Qcal} \Pval_1^\Qtt$. 
However, $(\Pval_t)$ is indeed inadmissible as it is dominated by $(\Pval_t')$ where $\Pval_t'=U$ for all $t \in \NN_0$.
\end{example}

\begin{example}[Necessary conditions for Corollary~\ref{cor:martingale-building-block-lemma}\ref{cor:martingale-building-block-lemma:2} are not sufficient] 
For example,  assume that $\Qcal$ consists of two measures $\Qtt_1$ and $\Qtt_2$. Under both $\Qtt_1$ and $\Qtt_2$, assume the distribution of $X_1$ is standard Gaussian. Furthermore, under $\Qtt_1$ the sequence $(X_2, X_3, \ldots)$ is uniform, while under $\Qtt_2$ the sequence is standard Gaussian. 
Define now two $\Qcal$-martingales $(M^{\Qtt_1}_t)$ and $(M^{\Qtt_2}_t)$ by 
\begin{align} \label{eq:220223.1}
	M^{\Qtt_1}_0 = 1 =  M^{\Qtt_2}_0;\qquad M^{\Qtt_1}_t = \exp\left(X_1 - \frac{1}{2}\right) \qquad  \text{and}  \qquad M^{\Qtt_2}_t = \exp\left(-X_1 - \frac12\right), \qquad t \geq 1.
\end{align}
 Then the e-process $(E_t)$ given by the minimum of $(M^{\Qtt_1}_t)$ and $(M^{\Qtt_2}_t)$ is clearly not admissible as it is strictly bounded from above by the constant process equal to one.    
\end{example}

\subsection{Composite Ville-like anti-concentration bounds}
\label{sec:anti-concentration}

Next, we discuss anti-concentration results, which are somewhat surprisingly insufficient for admissibility.
Since Ville's inequality has been often used in this paper to demonstrate validity, one would hope that if Ville's inequality is `essentially' tight (it holds with `almost' equality) then the corresponding inferential instruments may be close to admissible. We examine this angle next.
Below, we derive an anti-concentration (lower) bound to complement the upper bound of Ville's inequality. 

\begin{lemma}[Anti-concentration for pointwise NMs]\label{lem:anticoncentration-pointwise}
Let $(M_t)$ be a $\Qtt$--NM with $\EE_\Qtt[M_0] = 1$, so that  the `multiplicative increment' $Y_t := M_t/M_{t-1}$ (with $0/0 := 1$) has unit mean. 
Assume that the aggregate empirical variance of $(Y_t)$ is $\Qtt$-almost surely infinite, i.e.,
\begin{align} \label{eq:emp-var}
    \Qtt\left(\sum_{t \in \NN} (Y_t - 1)^2 = \infty \right) = 1. 
\end{align}
Then $M_\infty=0$, $\Qtt$-almost surely.  Fix now some $\varepsilon > 0$.
Assume that for each $t \in \NN_0$ the multiplicative increment  $Y_t$ with $Y_0 := M_0$ satisfies a tail condition, namely for each $\mathcal{F}_{t-1}$-measurable random variable $\beta > 1$ we have
\begin{align} \label{eq:200802} 
\EE_{\Qtt}\left[ Y_t \left| \mathcal{F}_{t-1}, Y_t \geq \beta \right.\right] \leq \beta (1 + \varepsilon).
\end{align}
Here $\Fcal_{-1}$ is the trivial sigma-field by convention.
Then, for any $\alpha \in (0,1]$ we have
\[
\alpha ~\geq~ \Qtt\left(\sup_{t\in \NN_0} M_t \geq \frac{1}{\alpha}\right) ~\geq~ \frac{\alpha}{1+\varepsilon}.
\]
\end{lemma}

The proof is in Appendix~\ref{sec:proofs}. 
We observe that \eqref{eq:200802} is satisfied, for example, when $\Qtt(Y_t\leq 1 + \varepsilon) = 1$ for all $t \in \mathbb{N}_0$. 
We note that \eqref{eq:emp-var} is easily satisfied when $\Ptt$ is a product measure and thus $(Y_t)$ is a sequence of i.i.d.~random variables, as long as $\Qtt(Y_t \neq 1) > 0$. 
We can now extend the above pointwise result to the composite setting, and once more this result is of independent interest. 

\begin{corollary}[Anti-concentration for bounded composite NMs] \label{C:200724}
Consider a family $\Qcal$ of probability measures and a $\Qcal$-NM $(M_t)$ with $M_0 = 1$. Define $Y_t := M_t/M_{t-1}$ (with $0/0 := 1$) and assume that for each $\varepsilon > 0$ there exists some $\Qtt \in \Qcal$ such that conditions \eqref{eq:emp-var} and \eqref{eq:200802} hold for each $t \in \NN$ and $\mathcal{F}_{t-1}$-measurable random variable $\beta > 1$. Then
\[
    \sup_{\Qtt \in \Qcal} \Qtt\left(\sup_{t\in \NN_0} M_t \geq \frac{1}{\alpha}\right) = \alpha, ~ \text{ for any $\alpha \in (0,1]$. }
\]
\end{corollary}
In words, the above result establishes rather simple and interpretable sufficient conditions under which $\Pval_t := \inf_{s\leq t} 1/M_s$ uses up all its type-I error budget, meaning that at least in a worst-case sense, Ville's inequality did not lead to a conservative test.

Unfortunately, Example~\ref{ex:sym_counter} shows, in the context of conditionally symmetric distributions (see Section~\ref{sec:examples}), that even under the assumptions of the previous corollary, such $(\Pval_t)$ does not need to be admissible. 
This further demonstrates the subtleties of establishing sufficient conditions for admissibility. 
Nevertheless, Corollary~\ref{C:200724} is of independent interest; but the fairly intuitive condition it yields does usually not suffice for admissibility.

\subsection{The role of initial randomization in determining admissibility}\label{sec:sufficient-p}
Throughout this subsection, let us consider $\Qcal = \{\Qtt\}$. 

The admissibility of p-processes is subtle  
and randomization appears to play a key role in enabling admissible constructions. The key difficulty is in dealing with atomic limiting distributions, and we delve more into this topic here with several examples.

In Corollary~\ref{C:200906}, $\inf_{s\in\NN_0} 1/ M_s$ is assumed to have an atomless distribution function $F$. Initial randomization turns out to be necessary for this to hold. To formalize this claim, consider a martingale $(M_t)$. We now argue the following fact:
\[
   \text{If $M_0 =1$, $\Qtt$-almost surely, then $\sup_{t \in \NN_0} M_t$ has an atom at one under $\Qtt$.}
\]
The proof is simple, so we present it immediately. Define $Y_t := M_t/M_{t-1}$ with $0/0:=1$ Without loss of generality, we may assume that $\Qtt(Y_1 \neq 1) > 0$. Since $\EE_\Qtt[Y_1]=1$, there exists some $\eta>0$ such that $\Qtt(Y_1 \leq 1- \eta) > \eta$. On the event $\{Y_1 \leq 1-\eta\}$, the conditional version of Ville's inequality~\eqref{eq:conditional-Ville} yields that $\Qtt(\sup_{t \in \NN} M_t \geq 1|\Fcal_1) \leq M_1 \leq 1-\eta$. Hence on this event we have $\Qtt(\sup_{t \in \NN} M_t < 1|\Fcal_1) \geq \eta$, yielding the unconditional bound $\Qtt(\sup_{t \in \NN} M_t < 1) \geq \eta^2$. This then gives
$\Qtt(\sup_{t \in \NN_0} M_t = 1) \geq \eta^2$. Hence $\sup_{t \in \NN_0} M_t$ has an atom at one, and so does the induced p-process, completing the proof of the aforementioned fact. 

In contrast, if we consider the martingale $(M'_t)$ with randomized initial value  $M_t' :=  M_t + \varepsilon U$, where  $\varepsilon > 0$, and recall that $U$ is the (independent) $\Fcal_0$-measurable $[0,1]$-uniformly distributed random variable, then  $\sup_{t \in \NN_0} M_t' = \sup_{t \in \NN_0} M_t + \varepsilon U$ has a density since it is the convolution of $\sup_{t \in \NN_0} M_t$ with a random variable that has a density. 
% Sometimes this density can be explicitly derived; for calculations involving the memoryless exponential law, see Example~\ref{eg:randomization-p-admissibility}. 

Let us consider for the moment a p-process  constructed as $\Pval_t := F(\inf_{s \leq t} 1/ M_s)$, where $(M_t)$ is a $\Qtt$-martingale with $M_0 = 1$ and $F$ is the distribution function of $\inf_{s \in \NN_0} 1/ M_s$. (Note that such a p-process always dominates $(\inf_{s \leq t} 1/ M_s)$.) Then $(\Pval_t)$ is \emph{always} inadmissible. To see this, 
define $\Pval_\infty := \inf_{t \in \NN_0} \Pval_t$ and $\delta := \Qtt(\Pval_\infty = 1) > 0$, where the inequality follows from the fact argued above. Moreover, define the conditional distribution function $G$ by $[0,1] \ni u \mapsto \Qtt(U \leq u | \Pval_\infty = 1)$. 
Let us then define $\Pval_t' := \Pval_t \wedge (1-\delta + \delta G(U))$. Then clearly $(\Pval_t')$ strictly dominates $(\Pval_t)$. Moreover, $(\Pval_t')$ is a p-process since $\Pval_\infty' := \inf_{t \in \NN_0} \Pval_t' = \Pval_\infty \wedge (1-\delta + \delta G(U))$ stochastically dominates a uniform. Indeed, for $\alpha \in (0,  1-\delta)$ we have $\Qtt(\Pval_\infty' \leq \alpha) = \Qtt(\Pval_\infty \leq \alpha) \leq \alpha$ and for $\alpha \in [1-\delta, 1]$ we get
\begin{align*}
    \Qtt(\Pval_\infty' \leq \alpha) &= 
        \Qtt(\Pval_\infty \leq 1- \delta) + \Qtt(\Pval_\infty = 1, 1-\delta + \delta G(U) \leq \alpha)  \\
        &= 1- \delta + \delta \Qtt(\delta G(U) \leq \alpha - (1-\delta) ~|~ \Pval_\infty = 1)\\
        &= 1-\delta + \delta \frac{\alpha - (1-\delta)}{\delta} = \alpha, 
\end{align*}
where we used that $G(U)$ is uniformly distributed under the \smash{conditional measure 
$\Qtt(\cdot|\Pval_\infty = 1)$}.

We note above that atoms at one are `obviously' undesirable for p-processes. Quite surprisingly, \emph{there do exist admissible anytime p-processes with atomic limiting distributions} (where the atoms are not at one); see Example~\ref{eg:atomic-pvalue}. In that example, we have $\Qtt(\inf_{t \in \NN_0} \Pval_\infty < 1/2) = 0$, and $(\Pval_t)$ is independent of the randomization device $U$; nevertheless it is \emph{impossible} to `randomize' the atom.

To end the discussion about atoms in the context of p-processes, we remark that atomic limiting distributions occur more often in discrete time than in continuous time. For example, if $(B_t)_{t \in [0,\infty)}$ is a standard Brownian motion, then $(\exp(B_t - t/2))_{t \in [0,\infty)}$ is a martingale, and $\inf_{t\geq 0} 1/\exp(B_t - t/2)$ is exactly $[0,1]$-uniformly distributed. However, the corresponding standard Gaussian NM from \eqref{eq:Gaussian-NM} has that $\inf_{t\in\NN_0} 1/ G_t^0$ is atomic when $(X_t)$ under $\Qtt$ follow the law of i.i.d.~standard Gaussians.

 In sharp contrast, initial randomization causes sequential tests based on e-processes to become inadmissible.
 Indeed, if the jumps  of $(\Eval_t)$ are continuous with positive probability then the corresponding sequential test is not admissible for \emph{any} $\alpha \in (0, 1)$ due to overshoot. Only e-processes that have atomic jumps can possibly lead to admissible tests; however, any such e-value cannot lead to an admissible test for every $\alpha$ (it will overshoot for some and not for others). 
 Example~\ref{eg:symmetric-sufficient-test} in the next section derives an admissible sequential test for (composite) symmetric distributions.

{
\color{black}
\section{SubGaussian admissibility} \label{sec:subG}

% \begin{lemma}[Conjecture]
% If $(\Eval_t)$ is admissible for $\Qcal$, then it is valid and admissible for the closure $\bar \Qcal$ of $\Qcal$, for some definition of closure. In that case $\EE_\Qtt [\Eval_t] = 1$ for some $\Qtt \in \bar \Qcal$.
% \end{lemma}

% \begin{lemma}[Conjecture]
% If $\Pcal=\{\mu_1 \times \mu_2 \times \dots : \mu_i \in \mathcal{M}\}$ and $E$ is an admissible e-variable for $\mathcal{M}$, then $\Eval_t := E(X_1)E(X_2)\dots$ is admissible for $\Pcal$. 
% \end{lemma}

% \begin{lemma}[Conjecture]
% If $\Pcal=\mathcal{M}_1 \times \mathcal{M}_2 \times \dots$ and $E_i$ is an admissible e-variable for $\mathcal{M}_i$, then $\Eval_t := E_1(X_1)E_2(X_2)\dots$ is admissible for $\Pcal$. In words, for sets $\Pcal$ of product measures, we never have to look past $\Pcal$-NSMs.
% \end{lemma}

% \begin{lemma}[Conjecture]
% If $\Eval_t$ is admissible for $\Pcal$, then it is admissible for $\mathrm{fconv}(\Pcal)$, provided the latter is well-defined.
% \end{lemma}

For a nontrivial example of an admissible e-process, consider the family $\Qcal \equiv \Gcal^\downarrow_0$ of centered conditionally 1-subGaussian laws (the proofs for subGaussian constant $\sigma \neq 1$ follow analogously).
Recalling Section~\ref{sec:subgaussian-defn}, $\Qcal$ consists of those laws $\Qtt$ such that the conditional moment generating function of $X_t$ satisfies $\EE_\Qtt[e^{\lambda X_t}|\Fcal_{t-1}] \leq e^{\lambda^2/2}$ for all $t \in \NN$ and all $\lambda \in \mathbb R$.  This is an important null hypothesis, studied by \citet{darling_confidence_1967,  Robbins_statistical_1970,howard_uniform_2019}.    Consider now the process $(\Eval_t)$ given by 
\begin{align} \label{eq:220222}
	\Eval_t = e^{\sum_{i=1}^t X_i - t/2}, \qquad t \in \NN_0.
\end{align}
Then $(\Eval_t)$ is a $\Qcal$-NSM as discussed in Section~\ref{sec:subgaussian-defn}, and thus a $\Qcal$-e-process. (The arguments in this section can be easily adapted to hold for the process $N$ in~\eqref{eq:subG-NSM2}, so we just consider $\lambda=1$ and subGaussian constant $\sigma=1$ for simplicity of exposition.)

% In addition, $(\Eval_t)$ is a $\Qtt^*$-martingale, where $\Qtt^* \in \Qcal$ is the law under which the data is i.i.d.\ standard Gaussian. In particular, $(\Eval_t)$ is admissible for the simple null $\{\Qtt^*\}$ thanks to Theorem~\ref{thm:pointwise}\ref{thm:pointwise:2}. However,  The following result shows that, nonetheless, $(\Eval_t)$ is $\Qcal$-admissible.

\begin{theorem}\label{P_subGaussian_admissible}
The e-process \eqref{eq:220222} is admissible for $\Qcal$.
\end{theorem}

The main difficulty in proving the above theorem is handling both continuous and discrete distributions. Indeed, consider first the following (simpler) proposition about $\Qcal^\sim \subset \Qcal$, the set of all distributions within the 1-subGaussian null hypothesis that have Lebesgue densities up to any finite time, and $\Qcal^{\times} \subset \Qcal^\sim$, the subclass of i.i.d.\ 1-subGaussian distributions with Lebesgue densities.

\begin{proposition}\label{prop:subG-cont}
The e-process \eqref{eq:220222} is admissible for $\Qcal^\sim$ and $\Qcal^\times$.
\end{proposition}
The proof of the above proposition is simple and works as follows. Note that~\eqref{eq:220222} is a $\Qtt_0$-NM, where (as in the following proof) $\Qtt_0$ is the product measure of i.i.d.~standard Gaussians. Thus it is immediately $\Qtt_0$-admissible, and the singleton reference family $\{\Qtt_0\}$  satisfies the local absolute continuity condition required to invoke Proposition~\ref{P:200823}\ref{P:200823:2} for both $\Qcal^\times$ and $\Qcal^\sim$.

However, Proposition~\ref{prop:subG-cont} (or an amendment of its short proof) does not imply Theorem~\ref{P_subGaussian_admissible} because $\Qcal$ contains many laws that are singular with respect to $\Qtt_0$ on each $\Fcal_t$, and so we cannot use Proposition~\ref{P:200823}\ref{P:200823:2} to deduce $\Qcal$-admissibility. In fact, $\Qcal$ contains uncountably many mutually singular laws.

% A related but conceptually separate difficulty that arises is that, in general, one-step admissibility does not imply multi-step admissibility, even when considering i.i.d. observations. To clarify, one may quite reasonably conjecture the following:
% \begin{quote}
% If $\Pcal=\{\mu^\infty: \mu \in \mathcal{M}\}$ consists of i.i.d. distributions built from a base class $\mathcal M$, and $E$ is an admissible e-variable for $\mathcal{M}$, then $E(X_1)E(X_2)\dots$ is an admissible e-process for $\Pcal$. 
% \end{quote}
% Indeed, the proof that follows will show that $e^{X-1/2}$ is an admissible e-variable for the class $\mathcal M$ being all 1-subGaussian distributions (for one observation). 
% So if the conjecture were true, Proposition~\ref{P_subGaussian_admissible} would follow smoothly from the one-period case.

Despite the fact that \eqref{eq:220222} is admissible for $\Qcal$ and $\Qcal^\times$, the example below demonstrates that it is not admissible for a nested set in between them. This further demonstrates the subtlety in proving subGaussian admissibility, created by the uncountably many mutually singular laws. 
% and thus fits the general theme of the above conjecture, it turns out that the above conjecture is actually false in general; a counterexample is provided below.

\begin{example}
% Let $\Pcal=\{\mu^\infty: \mu \in \mathcal{M}\}$ consist of i.i.d. distributions built from the base class $\mathcal M$, where 
Let $\mathcal{M}$ contain all 1-subGaussian distributions with a Lebesgue density, along with one discrete 1-subGaussian distribution: a symmetric Rademacher distribution (that places equal mass at $\pm 1$), and let $\Pcal=\{\mu^\infty: \mu \in \mathcal{M}\}$ so that $\Qcal \supset \Pcal \supset \Qcal^\times$. Clearly, $(\Eval_t)$ from~\eqref{eq:220222} is still a $\Pcal$-NSM, and Theorem~\ref{P_subGaussian_admissible} and Proposition~\ref{prop:subG-cont} claimed it was admissible for $\Qcal^\times$ and $\Qcal$. However, $(\Eval_t)$  is inadmissible for $\Pcal$! 
% However, the inclusion of this extra distribution makes it inadmissible. 
Indeed, define
$\tilde \Eval_t := e^{\sum_{i=1}^t X_i - t \log \cosh 1 }$, and
$\Eval'_t =  \Eval_t+ \1_{\{|X_1|=1\}}(\tilde \Eval_t-\Eval_t)$, or more explicitly
\[
\Eval'_t = \begin{cases}
\tilde \Eval_t \text{ if } |X_1|=1\\
\Eval_t \text{ otherwise}
\end{cases}.
\]
% For $t > 1$,  we set $\Eval'_t = \Eval_t$ if $|X_1| \neq 1$, and otherwise switches to equal $e^{X_1-1/2}e^{\sum_{i=2}^t X_i - (t-1) \log \cosh 1 }$. 
Note that $e^{X - \ln \cosh 1}$ has mean equal to 1 under the Rademacher distribution, so $(\tilde \Eval_t)$ is a nonnegative martingale under the i.i.d.\ Rademacher distribution (but not under any other distribution in $\Pcal$), and since $\ln \cosh 1 < 1/2$, we have $\tilde \Eval_t > \Eval_t$ for all $t$. Since the event $\{|X_1|=1\}$ has Lebesgue measure zero,  $(\Eval'_t)$ is a $\Pcal$-NSM. Further, $\Eval'_t \geq \Eval_t$ by definition, and $\Eval'_t > \Eval_t$ for all $t$ with positive probability under the Rademacher distribution. This makes $\Eval_t$ from \eqref{eq:220222} inadmissible for~$\Pcal$.

It is worth noting that this example does not contradict Proposition~\ref{prop_adm_superset_proof}. Indeed, this proposition is not applicable here because $\Pcal$ has fewer polar sets than $\Qcal^\times$. In particular, in the above example we have $\Eval'_t = \Eval_t$ up to $\Qcal^\times$-polar sets, but not up to $\Pcal$-polar sets. Regarding the relation to the larger set $\Qcal$, $(\Eval'_t)$ cannot be used to witness inadmissibility of $(\Eval_t)$ with respect $\Qcal$, because $(\Eval'_t)$ is not even a $\Qcal$-e-process. \end{example}

Before giving the proof of Theorem~\ref{P_subGaussian_admissible}, we present the following construction of `almost Gaussian' 1-subGaussian random variables with an atom. We write it as a lemma, since it may be of independent interest outside this paper.

\begin{lemma}[Perturbing a Gaussian into a subGaussian with one atom]
Given a standard Gaussian $X$, any point $\bar x \in \RR$ and tolerance $\varepsilon>0$, one can always construct a random variable $X^*$ such that:
\begin{itemize}
\item There is an $\varepsilon$-sized interval $(a,b)$ containing $\bar x$, such that $\EE[X \mid X \in (a,b)] = \bar x$.
\item Outside of $(a,b)$ we have $X^*=X$, and inside of $(a,b)$, $X^*$ has an atom at $\bar x$ but no other mass.
\item $X^*$ has zero mean and is 1-subGaussian.
\end{itemize}
\end{lemma}
\begin{proof}
Let $X$ be a standard normal random variable. For any $\bar x \in \RR$ and $\varepsilon > 0$, we may choose $a < \bar x < b$ with $b-a = \varepsilon$ such that
\[
\EE[X \mid X \in (a,b)] = \bar x.
\]
To see this we examine the function $h(a) = \EE[X \mid X \in (a,a+\varepsilon)]$. This function is continuous and satisfies $h(\bar x-\varepsilon) < \bar x < h(\bar x)$. Thus by the intermediate value theorem we can choose $a \in (\bar x - \varepsilon, \bar x)$ such that $h(a) = \bar x$, and then set $b = a+\varepsilon$. Now, the random variable
\[
X^* = \EE[X \mid X\1_{\{X \notin (a,b)\}}] = \begin{cases}
X & \text{if } X \notin (a,b) \\
\bar x & \text{otherwise}
\end{cases}
\]
is 1-subGaussian. Indeed, Jensen's inequality, the tower rule, and the fact that $X$ is standard normal give
\[
\EE[e^{\lambda X^*}] = \EE[e^{\EE[\lambda X \mid X\1_{\{X \notin (a,b)\}}]}] \le \EE[e^{\lambda X}] = e^{\lambda^2/2}, \qquad \lambda \in \RR.
\]
This concludes the proof.
\end{proof}

The above construction plays a crucial role in the following proof of the  main theorem.

\begin{proof}[Proof of Theorem~\ref{P_subGaussian_admissible}]
Let $(\Eval_t')$ be an e-process with 
\begin{equation}\label{eq:running-assumption-v2}
\text{$\Qtt(\Eval_t' \ge \Eval_t) = 1$ for all $\Qtt \in \Qcal$ and $t \in \NN_0$.}
\end{equation}
In order to show that $(\Eval_t)$ is admissible, we must show that the inequality in \eqref{eq:running-assumption-v2} is, in fact, an equality.

Fix $t \in \NN_0$ and assume that $\Eval'_t$ does not involve any randomization. Thus $\Eval'_t = g(X_1,\ldots,X_t)$ for some measurable function $g$. We also have $\Eval_t = f(X_1)\cdots f(X_t)$ where $f(x) = e^{x-1/2}$. Assume for contradiction that there exist $\bar x_1,\ldots,\bar x_t \in \RR$ such that
\begin{equation}\label{eq_adm_proof_v2_0}
\delta = g(\bar x_1,\ldots,\bar x_t) - f(\bar x_1) \cdots f(\bar x_t) > 0.
\end{equation}
We will use this to construct $\Qtt^* \in \Qcal$ such that $\EE_{\Qtt^*}[\Eval'_t] > 1$, which is a contradiction. To this end, let $\Qtt_0$ be the law under which the data $(X_t)$ is i.i.d.\ standard normal. Fix $\varepsilon > 0$ and define
\[
\varepsilon_i = \varepsilon^{2^{t-i}}, \quad i = 0,\ldots,t.
\]
For each $i=1,\ldots,t$ we apply the construction before the proof with $X=X_i$, $\bar x = \bar x_i$, $\varepsilon = \varepsilon_i$ to obtain a 1-subGaussian random variable $X^*_i$ and an $\varepsilon_i$-sized interval $(a_i,b_i)$ containing $\bar x_i$ (the atom). We also define the event $A_i = \{X_i \in (a_i,b_i)\}$ and observe that we have $A_i = \{X^*_i = \bar x_i\}$.

We now define a law $\Qtt^*$ under which the observations $X_1,X_2,\ldots$ behave as follows. The first observation is distributed like $X_1^*$. If the atom $\bar x_1$ is \emph{not} realized, the remaining observations are i.i.d.\ standard normal. If the atom \emph{is} realized, the second observation is distributed like $X_2^*$. If the atom $\bar x_2$ is not realized, the remaining observations are i.i.d.\ standard normal. Otherwise, the third observation is distributed like $X^*_3$, and so on. We continue like this until time $t$. After that the data is i.i.d.\ standard normal. At each stage the conditional distribution of the subsequent observation is 1-subGaussian, and hence $\Qtt^*$ belongs to $\Qcal$.

Our goal is now to prove by (backward) induction that for $k=0,\ldots,t$ we have
\begin{equation}\label{eq_adm_proof_v2_1}
\EE_{\Qtt^*}\Big[ \Eval'_t \mid \bigcap_{i \le k} A_i \Big] \ge \delta \, \Omega\left( \frac{\varepsilon_k}{\varepsilon} \right) + \prod_{i\le k} f(\bar x_i) - O(\varepsilon_k).
%\EE_{\Qtt^*}\Big[ \Eval'_t \mid \bigcap_{i \le k} A_i \Big] \ge \Big( g(\bar x_1,\ldots,\bar x_t) - f(\bar x_1) \cdots f(\bar x_t) \Big) \Omega\left( \frac{\varepsilon_k}{\varepsilon} \right) + \prod_{i\le k} f(\bar x_i) - O(\varepsilon_k).
\end{equation}
Above, we use the Knuth and Landau asymptotic notation $\Omega(h(\varepsilon)) \ge c h(\varepsilon)$ and $O(h(\varepsilon)) \le c h(\varepsilon)$ for some positive constant $c > 0$ that does not depend on $\varepsilon$. For $k=0$ the empty product $\prod_{i \le 0} f(\bar x_i)$ is understood to equal one. The empty intersection $\bigcap_{i \le 0} A_i$ equals the full sample space, which reduces the conditional expectation to an unconditional expectation.

For $k=t$ it is clear that \eqref{eq_adm_proof_v2_1} holds thanks to \eqref{eq_adm_proof_v2_0} because $\varepsilon_t / \varepsilon = 1$ and
\[
\EE_{\Qtt^*}\Big[ \Eval'_t \mid \bigcap_{i \le t} A_i \Big] = \EE_{\Qtt^*}\Big[ g(X_1,\ldots X_t) \mid \bigcap_{i \le t} A_i \Big] = g(\bar x_1,\ldots,\bar x_t).
\]
This establishes the base case. Suppose now that \eqref{eq_adm_proof_v2_1} has been proved for some $k \in \{1,\ldots,t\}$. We will show that it holds for $k-1$ as well. Using the definition of conditional expectation and the fact that $\Eval'_t \ge \Eval_t$, $\Qtt^*$-almost surely, we compute
\begin{equation}\label{eq_adm_proof_v2_2}
\begin{aligned}
\EE_{\Qtt^*}\Big[ \Eval'_t \mid \bigcap_{i \le k-1} A_i \Big]
&= \EE_{\Qtt^*}\Big[ \Eval'_t \1_{A_k} \mid \bigcap_{i \le k-1} A_i \Big] + \EE_{\Qtt^*}\Big[ \Eval'_t \1_{A_k^c} \mid \bigcap_{i \le k-1} A_i \Big] \\
&\ge \EE_{\Qtt^*}\Big[ \Eval'_t \mid \bigcap_{i \le k} A_i \Big] \Qtt^*\Big(A_k \mid \bigcap_{i \le k-1} A_i \Big) + \EE_{\Qtt^*}\Big[ \Eval_t \1_{A_k^c} \mid \bigcap_{i \le k-1} A_i \Big].
\end{aligned}
\end{equation}
We develop the terms on the right-hand side. First,
\[
\Qtt^*\Big(A_k \mid \bigcap_{i \le k-1} A_i \Big) = \Qtt_0(A_k).
\]
Moreover, using that $\Eval_t = \prod_{i \le t} f(X_i)$, the definition of $\Qtt^*$, and that $\EE_{\Qtt_0}[f(X_i)]=1$ for all $i$, we get
\begin{align*}
\EE_{\Qtt^*}\Big[ \Eval_t \1_{A_k^c} \mid \bigcap_{i \le k-1} A_i \Big] &= \prod_{i \le k-1} f(\bar x_i) \ \EE_{\Qtt^*}\Big[ f(X_k) \1_{A_k^c} \prod_{i > k} f(X_i) \mid \bigcap_{i \le k-1} A_i \Big] \\
&= \prod_{i \le k-1} f(\bar x_i) \ \EE_{\Qtt_0}\Big[ f(X_k) \1_{A_k^c} \prod_{i > k} f(X_i) \Big] \\
&= \prod_{i \le k-1} f(\bar x_i) \ \EE_{\Qtt_0}[ f(X_k) \1_{A_k^c}] \prod_{i > k} \EE_{\Qtt_0}[f(X_i)] \\
&= \prod_{i \le k-1} f(\bar x_i) \ \left(1 - \EE_{\Qtt_0}[ f(X_k) \1_{A_k}] \right) \\
&= \prod_{i \le k-1} f(\bar x_i) \ \left(1 - \EE_{\Qtt_0}[ f(X_k) \mid A_k] \Qtt_0(A_k) \right).
\end{align*}
Substituting these expressions into \eqref{eq_adm_proof_v2_2}, using the induction assumption \eqref{eq_adm_proof_v2_1}, and rearranging the resulting terms, we get
\begin{align}
\EE_{\Qtt^*}\Big[ \Eval'_t \mid \bigcap_{i \le k-1} A_i \Big]
&\ge \left( \delta \, \Omega\left( \frac{\varepsilon_k}{\varepsilon} \right) + \prod_{i\le k} f(\bar x_i) - O(\varepsilon_k) \right) \Qtt_0(A_k) \nonumber \\
&\quad + \prod_{i \le k-1} f(\bar x_i) \ \left(1 - \EE_{\Qtt_0}[ f(X_k) \mid A_k] \Qtt_0(A_k) \right) \\
&= \delta \, \Omega\left( \frac{\varepsilon_k}{\varepsilon} \right) \Qtt_0(A_k) - O(\varepsilon_k) \Qtt_0(A_k) \nonumber \\
&\quad + \prod_{i \le k-1} f(\bar x_i) - \prod_{i \le k-1} f(\bar x_i) \ \EE_{\Qtt_0}[ f(X_k) - f(\bar x_k) \mid A_k] \Qtt_0(A_k).
\end{align}
Note that these computations are valid even when $k=1$, once the meaning of empty intersections and products have been properly accounted for. Since $\Qtt_0(A_k) = \Qtt_0(X_k \in (a_k,b_k))$ with $X_k$ standard normal under $\Qtt_0$, we have that $\Qtt_0(A_k)$ is both $O(\varepsilon_k)$ and $\Omega(\varepsilon_k)$. Moreover, since the derivative of $f$ is bounded on $(a_k,b_k)$, we have that $\EE_{\Qtt_0}[ f(X_k) - f(\bar x_k) \mid A_k]$ is $O(\varepsilon_k)$. As a result, the right-hand side above can be written
\[
\delta \, \Omega\left( \frac{\varepsilon_k^2}{\varepsilon} \right) - O(\varepsilon_k^2) + \prod_{i \le k-1} f(\bar x_i) - O(\varepsilon_k^2).
\]
Thus, noting that
\[
\varepsilon_k^2 = \left( \varepsilon^{2^{t-k}} \right)^2 = \varepsilon^{2^{t-k+1}} = \varepsilon_{k-1},
\]
we conclude that
\[
\EE_{\Qtt^*}\Big[ \Eval'_t \mid \bigcap_{i \le k-1} A_i \Big] \ge \delta \, \Omega\left( \frac{\varepsilon_{k-1}}{\varepsilon} \right) + \prod_{i \le k-1} f(\bar x_i) - O(\varepsilon_{k-1}).
\]
This shows that \eqref{eq_adm_proof_v2_1} holds with $k-1$ in place of $k$, and it follows by induction that \eqref{eq_adm_proof_v2_1} holds for all $k \in \{0,\ldots,t\}$.

We now apply \eqref{eq_adm_proof_v2_1} with $k=0$ to get
\begin{equation} \label{eq_adm_proof_no_rand_conclusion}
\EE_{\Qtt^*}[ \Eval'_t ] \ge \delta\ \Omega( \varepsilon^{2^t-1} ) + 1 - O(\varepsilon^{2^t}).
\end{equation}
Thus, thanks to \eqref{eq_adm_proof_v2_0}, by choosing $\varepsilon>0$ sufficiently small we can ensure that $\EE_{\Qtt^*}[ \Eval'_t ] > 1$. This contradicts validity of $(\Eval'_t)$ and shows that \eqref{eq_adm_proof_v2_0} is impossible. We conclude that $(\Eval_t)$ is admissible.

\smallskip

We have proved the proposition in the case where no randomization is present. If there is randomization, so that $\Eval'_t = g(X_1,\ldots,X_t,U)$ involves the uniformly distributed independent randomization device $U$, the contradiction assumption \eqref{eq_adm_proof_v2_0} is replaced by the following: there exist $\bar x_1,\ldots,\bar x_t \in \RR$, $\delta' > 0$, and a subset $I \subset [0,1]$ of positive Lebesgue measure such that
\[
g(\bar x_1,\ldots,\bar x_t,u) \ge \delta' + f(\bar x_1)\cdots f(\bar x_t) \text{ for all } u \in I.
\]
Constructing $\Qtt^*$ as above (with the additional property that $U$ is uniformly distributed and independent of the data) we obtain by the same arguments that
\[
\EE_{\Qtt^*}[\Eval'_t] = \EE_{\Qtt^*}[\Eval'_t \1_I(U)] + \EE_{\Qtt^*}[\Eval'_t \1_{I^c}(U)]
\ge \Big( \delta' \, \Omega( \varepsilon^{2^t-1} ) + 1 - O(\varepsilon^{2^t}) \Big) |I| + \EE_{\Qtt^*}[\Eval_t] (1-|I|),
\]
where $|I| = \Qtt^*(U \in I) > 0$ denotes the Lebesgue measure of $I$. We reach the desired contradiction if the right-hand side is strictly greater than one. This will be the case for $\varepsilon$ small enough, provided that
\[
\EE_{\Qtt^*}[\Eval_t] = 1 - O(\varepsilon^{2^t}).
\]
But this again follows from the same argument without randomization presented earlier, applied with $\Eval'_t$ replaced by $\Eval_t$; indeed, this yields \eqref{eq_adm_proof_no_rand_conclusion} with the left-hand side replaced by $\EE_{\Qtt^*}[\Eval_t]$ and $\delta=0$ on the right-hand side. This completes the proof in the presence of randomization.
\end{proof}

}

Invoking Proposition~\ref{prop_adm_superset_proof}, we note that admissibility for $\Gcal_{0}^\downarrow$ implies admissibility for $\Gcal_{-}^\downarrow$ and $\Gcal_{\sim}^\downarrow$ as well. Despite the substantial progress made above, it remains unclear to us whether \eqref{eq:220222} is admissible for the class of all nonatomic 1-subGaussian distributions, that is continuous distributions which may or may not have densities with respect to the Lebesgue measure. This class is a superset of 1-subGaussian distributions with Lebesgue densities, but a subset of all 1-subGaussian distributions --- despite~\eqref{eq:220222} being admissible for the subset and superset, we cannot immediately infer admissibility for the sandwiched set. 
To clarify, there do exist nonatomic 1-subGaussian distributions that do not have Lebesgue densities: for example, consider the symmetric distribution on $[-1,1]$ obtained by symmetrizing the Cantor distribution. 
We do conjecture that admissibility continues to hold, but a formal proof evades us at the present moment.

\section{Admissible inference for symmetric distributions}\label{sec:examples}

Considering the results of this paper presented thus far, we demonstrate how they may inform practice. At a high level, this section constructs (using different NMs) admissible version of all four instruments studied in this paper for the class of conditionally symmetric distributions.

We return to the example from Section~\ref{sec:symmetric-NSM} where we had presented a rather elegant, and intuitive, exponential NSM for distributions that yield conditionally symmetric observations. However, the results following the example showed that the tests or confidence sequences stemming from it are inadmissible, since all admissible constructions must use NMs. We will construct such NMs, which appear to be new to the best of our knowledge (but we would not be surprised if they have been discovered before). 
Recalling the notation from Section~\ref{sec:symmetric-NSM}, let $\Scal := \Scal^0$ be the set of laws such that $X_t$ conditional on $\mathcal F_{t-1}$ is symmetric around zero. 
We will demonstrate here that inference based on $(S^{0}_t)$ defined in \eqref{eq:symmetric-NSM} is inadmissible by explicitly constructing procedures that dominate it. 

We note that $\Scal$ is \emph{not} locally dominated. Indeed, just consider $\Ptt \in \Scal$ of the form $\Ptt= \Utt \times \mu^\infty$ (where $\Utt$ denotes the uniform measure), and note that there are uncountably many mutually singular choices for $\mu$; take for instance $\mu = (\delta_x+\delta_{-x})/2$ for $x\in\RR$. Here $\delta_x$ denotes the Dirac measure at $x$.
Nevertheless, despite the lack of a reference measure, it is still possible to construct a family of $\Scal$-NMs, and thus admissible  e-values for $\Scal$. Indeed, we have the following proposition.

\begin{proposition} \label{prop:symmetry}
An adapted process $(M_t)$ with $M_0$ bounded and nonnegative is an $\Scal$-NM  if and only if  $Y_t := M_t/M_{t-1}$ (with $0/0 := 1$) is of the form $Y_t = f_t(X_t)$, where $(f_t)$ is a nonnegative predictable function such that $f_t-1$ is odd, or equivalently, $f_t(x)+f_t(-x)=2$ for all $x \in \RR$. Moreover, if $M_0 = 1$ then $(M_t)$ is an $\Scal$-admissible e-value by Corollary~\ref{cor:martingale-suff-composite}\ref{cor:martingale-suff-composite:2}. 
\end{proposition}

The proof is in Section~\ref{sec:proofs}. 
A similar characterization as above also holds for any $\Scal$-NSM;
% \begin{remark}
in the notation of Proposition~\ref{prop:symmetry}, $(M_t)$ is an $\Scal$-NSM if and only if $Y_t=f_t(X_t)$ where $(f_t)$ is a nonnegative predictable function such that 
    \begin{equation}\label{eq:S-NSM}
    f_t(x) + f_t(-x) \leq 2, \qquad t \in \NN.
    \end{equation}
Moreover, an $\Scal$-NSM $(M_t)$ can be converted to an $\Scal$-NM $(\widetilde M_t)$ with $\widetilde M_t = \prod_{s \leq t} \widetilde f_s(X_s)$ by the following mirroring operation:
\[
    \widetilde f_t(x) = 
    \begin{cases}
    f_t(x), \qquad &f_t(x) \geq f_t(-x);\\
    2 - f_t(-x), \qquad &f_t(x) < f_t(-x).
    \end{cases}
\]
Indeed, we get that $\widetilde f_t \geq f_t$ and that equality holds in \eqref{eq:S-NSM} with $f_t$ replaced by $\widetilde f_t$.
% \end{remark}

Proposition~\ref{prop:symmetry} demonstrates how to construct admissible e-values for symmetry, and we give two instantiations that we have found (subjectively) elegant.
% \begin{example}
	Let $h$ be an odd function and consider
	\[	
		f(x) = 1 + \tfrac{2}{\pi}\arctan h(x)  \qquad \text{or} \qquad f(x) = 1 + \sin h(x) \qquad \text{or} \qquad f(x) = 1 + \tanh h(x).
	\]
	Then, $(\prod_{s\leq t} f(X_s))$ is an $\Scal$--NM and thus an e-process for $\Scal$ which is admissible by Proposition~\ref{prop:symmetry}.
% \end{example}

Finally, we return to the exponential $\Scal$-NSM from \cite{de_la_pena_general_1999} from Section~\ref{sec:symmetric-NSM}, showing that it is leads to an inadmissible e-process for $\Scal$, and improving it to an admissible one by converting the NSM to an NM.

\begin{example}
 Let $\Pcal \supset \Scal$  and
recall from Section~\ref{sec:symmetric-NSM} that the process $(S_t):=(S_t^0)$ given by
\[
    S_t = \prod_{s \leq t}  g(X_s), \qquad \text{where} \qquad g(x) := \exp\left(x - \frac{x^2}{2}\right),
\]
is an $\Scal$-NSM. Further, $(S_t)$ is not a martingale unless $X_t=0$ for all $t \in \NN$, and the corresponding $\Scal$-e-process is inadmissible due to the following argument.  Define
\[
    f(x) = 
    \begin{cases}
    g(x), \qquad &x \geq 0;\\
    2 - g(-x), \qquad &x < 0.
    \end{cases}
\]
    Then $f \geq g$ with equality only if $x=0$. Further, $f(-x) - 1 = 1 - f(x)$ and finally $f$ is nonnegative since $g \leq e^{1/2} \approx 1.65$; hence  $(\prod_{s\leq t} f(X_s))$ is an $\Scal$-NM by Proposition~\ref{prop:symmetry}. This also yields that the corresponding e-value is admissible. Even though the original $\Scal$-NSM is inadmissible, we recognize the aesthetic and analytical advantage in having a simple exponential formula.
\end{example}

Let us now illustrate that the p-processes corresponding to the $\Scal$-NMs fully utilize the available Type-I error budget in the sense of the next proposition.

\begin{proposition}\label{prop:admissible-pvalue}
Consider a nonnegative function $f$,  continuous and strictly monotone at zero and such that $f-1$ is odd.  Then $M_t := \prod_{s \leq t} f(X_s)$ is an $\Scal$-NM by Proposition~\ref{prop:symmetry} and we have 
\begin{align} \label{eq:200803}
    \sup_{\Qtt \in \Scal}~ \Qtt\left(\sup_{t\in \NN_0} M_t \geq \frac{1}{\alpha}\right) = \alpha, ~ \text{ for any $\alpha \in (0,1]$. }
\end{align}
Thus, defining $\Pval_t := \inf_{s\leq t} 1/M_s$, we have that $(\Pval_t)$ is a p-process for $\Scal$ that satisfies
\begin{equation}\label{eq:p-value-sym}
 \sup_{\Qtt \in \Scal,\tau \geq 0}~ \Qtt\left( \Pval_\tau \leq \alpha \right) = \alpha, ~ \text{ for any $\alpha \in (0,1]$. }
\end{equation}
\end{proposition}
The proof is in Section~\ref{sec:proofs}. An analogous result to the above proposition is known for the class of subGaussian distributions~\cite[Proposition 4]{howard_uniform_2019}, but had only been conjectured for other nonparametric classes like $\Scal$.
Unfortunately, not every p-process for $\Scal$ constructed as \eqref{eq:p-value-sym} above is admissible; see Example~\ref{ex:sym_counter} in Appendix~\ref{sec:more-examples}. 

Finally, let us construct an $\Scal$-admissible p-process in the next example.
\begin{example}
Define the following subset of symmetric distributions:
\begin{align}\label{eq:200918}
\widetilde \Scal := \left\{\Ptt \in \Scal: \Ptt\left( X_t \neq 0 \textnormal{ i.o.}\right) = 1\right\}.
\end{align}
Next, we define the process $(\Pval_t)$ as $\Pval_0:=1$ and
    \[
        \Pval_t := 1 - \sum_{s \leq t} \frac{1}{2^{N_s}}  \1_{X_s > 0} = \Pval_{t-1} - \frac{1}{2^{N_t}}  \1_{X_t > 0}, \qquad \text{where} \qquad N_s := \sum_{i \leq s} \1_{X_i \neq 0}. 
    \]
Then it can be checked that $(\Pval_t)$ is a closed $\Scal$-MM.
    Moreover, $\inf_{t \in \NN_0} \Pval_t$ is $\Qtt$-uniform for each $\Qtt \in \widetilde \Scal$. 
    Assume for the moment that $\Qtt \in \Scal \setminus \widetilde \Scal$  can be locally dominated by some $\Qtt' \in \widetilde \Scal$.  Proposition~\ref{P:200823}\ref{P:200823:1} and Corollary~\ref{cor:martingale-suff-composite}\ref{cor:martingale-suff-composite:1} then yield that $(\Pval_t)$ is $\Scal$-admissible. 
    
    Let us now fix some $\Qtt \in \Scal \setminus \widetilde \Scal$ and argue that indeed it can be locally dominated by some $\Qtt' \in \widetilde \Scal$. To do so, define the measure $\mu := 1/2 (\delta_1 + \delta_{-1})$ with $\delta_x$ denoting again the Dirac measure at $x \in \RR$.  Moreover, let $\Htt$ denote the law of a  Poisson random variable with expectation one. On the appropriate canonical space, consider the measure $\Qtt \times \mu^\infty \times \Htt$, and write $U, (X_t), (Y_t), H$ for the canonical random variables.
    To summarize,  $U$ is uniform, $(X_t)$ is our original conditionally symmetric sequence,  $(Y_t)$ is an independent sequence of Rademacher random variables, and $H$ is Poisson.
    Define a new sequence $(X_t')$  by $X_t' := X_t \1_{H > t} + Y_t \1_{H \leq t}$ for all $t \in \NN$ and let $\Qtt'$ denote the measure induced by $U$ and $(X_t')$. Then $\Qtt' \in \widetilde \Scal$ and it can be checked that $\Qtt'$ locally dominates $\Qtt$. This completes the proof of our initial claim.
    
    As a final observation, observe that if we replace $\Scal$ by the superset of probability measures for which the conditional laws of $X_t$ have median zero, all statements still hold.
\end{example}

This section has now developed admissible e-values and p-processes for testing for symmetry, and we move next to admissible sequential tests (and thus confidence sequences).

\begin{example}\label{eg:symmetric-sufficient-test}
    Consider the null $H_0: \Ptt \in \Scal$ from definition~\eqref{eq:symmetric-distributions}, and let $\alpha=0.05$ so that $1/\alpha=20$. Consider the process $(M_t)$ defined as follows. Let $M_0=1$, and zero is an absorbing state, meaning that if $M_t=0$, then the process stays at zero from then on. If $M_t$ is nonzero, then define $M_{t+1} := M_t + \rm{sign}(X_t)\1_{X_t \neq 0}$. It is easy to check that $(M_t)$ is an $\Scal$-NM. Define $\tau$ as the first time $M_t$ reaches $20$. Then,
    $M_\tau=20$ on $\{\tau < \infty\}$, and also $M_\infty=0$, $\Qtt$-almost surely, for each $\Qtt \in \widetilde \Scal$, as in \eqref{eq:200918}.
    Invoking
    Proposition~\ref{P:200823}\ref{P:200823:3} and Corollary~\ref{cor:martingale-suff-composite}\ref{cor:martingale-suff-composite:3}
    as in the previous example
    yields that $(\1_{M_t\geq 20})$ is $(\Scal,0.05)$-admissible. 
    More generally, $(\1_{M_t \geq 1/\alpha})$ is $(\Scal,\alpha)$-admissible whenever $1/\alpha \in \NN$.
\end{example}

Of course, there is nothing special about 0.05 and 20; for any other $\alpha$, the process $(M_t)$ can be altered accordingly to yield an admissible test for that $\alpha$.
The above process $(M_t)$ also delivers admissible tests for several subsets of $\Scal$, for example for if we restricted to only Gaussian distributions with any variance. 
This is interesting because admissibility is generally not subset-proof or superset-proof, but above we have a single process $(M_t)$ that yields admissible e-values and sequential tests for a variety of subsets of $\Scal$.

Continuing from Example~\ref{eg:symmetric-sufficient-test} and using Theorem~\ref{thm:ST/CS-admissibility}, we can construct an admissible level $\alpha$ test for any $(\Scal^m)_{m \in \RR}$, so the above construction yields an admissible confidence sequence for the center of symmetry.

Thus, we have accomplished our goal of constructing admissible versions of all four instruments for sequential inference, for a composite nonparametric class of distributions.

\section{Summary}

The central contribution of this work is to identify the central role of nonnegative martingales in anytime-valid sequential inference. As a by-product, we have added several modern mathematical techniques to the toolkit of the methodologist who wishes to design statistically efficient methods for inference at arbitrary stopping times. We end with a few comments.

It is apparent to us that some of our analysis may have been simpler in continuous time. Indeed, some of the difficulty in constructing admissible sequential tests using e-values arises from the `overshoot', while the difficulty in designing admissible p-processes arises because we do not observe the process `in between' the fixed times and thus the running infimum is not \emph{exactly} uniformly distributed in the limit. Several of these problems go away with continuous time/path martingales. However, continuous path martingales only represent large-scale approximations of most actual experimental setups, which typically involve discrete events. The accuracy of these approximations would have to be assessed, especially outside very high-frequency settings like finance, and it may not be clear how to do so. We believe that the additional effort to understand admissibility in the discrete time setup was fruitful.

Following the literature, our sequential inference tools were only required to have marginal guarantees, and not conditional ones. To pick one example, we required that for each $\Qtt \in \Qcal$, an e-value must satisfy $\EE_\Qtt [\Eval_\tau] \leq 1$ at arbitrary stopping times $\tau$, but it need not satisfy $\EE_\Qtt [\Eval_t | \Fcal_s] \leq \Eval_s$. This gap between conditional and marginal guarantees is paramount: it allows for the construction of e-values that are not simply $\Qcal$-NMs, because in several settings of interest, one can show that the only nontrivial $\Qcal$-NM is the constant process that equals one at all times, but nontrivial e-values with power to detect deviations from $\Qcal$ can still be constructed. We explore these connections further by using a structural notion called `fork-convexity' in a separate work.

Finally, the paper provides a rather general treatment of the inferential tools and problem settings. However, perhaps additional insights could be gained when $\Pcal$ or $\Qcal$ have special structure, or when we pay attention to particular classes of stopping times, or restrict ourselves  to a bounded horizon; these may all be promising directions to explore. Finally, while we take a step forward in relating the various concepts used for sequential inference, and present a thorough analysis of their validity and admissibility, the question of optimality is unaddressed by our work. Of course, this usually needs to be studied by specifying appropriate alternatives, and introducing metrics by which to judge optimality (such as the GROW criterion of~\citet{grunwald_safe_2019}), and so we leave such considerations for future work.

\subsection*{Acknowledgments}
The authors are thankful to the organizers of the International Seminar on Selective Inference, which stimulated conversations that led to this paper. AR acknowledges NSF DMS grant 1916320.

{
\hypersetup{linkcolor=red}
\bibliography{main}
\bibliographystyle{plainnat}
}

 \newpage

\appendix

% \section{Appendix}
% \label{sec:appendix}

% Appendix goes here.

\section{Additional technical concepts and definitions}  \label{app:A}

\subsection{Reference measures and local absolute continuity}
\label{sec:ref_measure}

Consider a probability space with a filtration $(\Fcal_t)_{t \in \NN_0}$. Let $\Rtt$ be a particular probability measure on $\Fcal_\infty$; we think of $\Rtt$ as a \emph{reference measure}.
% All probability measures of interest will be locally dominated by $\Rtt$, and we now explain what this means and why we need it. 
We now explain the concept of local domination and how it allows us to unambiguously define conditional expectations.
% under locally dominated probability measures.

%\paragraph{Local absolute continuity.}
\begin{itemize}
    \item $\Ptt$ is called \emph{locally absolutely continuous with respect to $\Rtt$} (or \emph{locally dominated by $\Rtt$}), if $\Ptt_t\ll \Rtt_t$ for all $t\in\NN$. We write this $\Ptt\ll_\text{loc}\Rtt$. More explicitly, this means that
    \[
    \text{$\Rtt(A)=0$}\quad\Rightarrow\quad \Ptt(A)=0, \quad \text{ for any $A\in\Fcal_t$ and $t\in\NN$}.
    \]
    Local absolute continuity does \emph{not} imply that $\Ptt\ll \Rtt$. However, it does imply that $\Ptt_\tau\ll \Rtt_\tau$ for any finite (but possibly unbounded) stopping time $\tau$. Indeed, if $A\in\Fcal_\tau$ and $\Rtt(A)=0$, then $A\cap\{\tau\le t\}\in\Fcal_t$ for all $t$, and hence $\Ptt(A)=\lim_{t\to\infty} \Ptt(A\cap\{\tau\le t\})=0$.
    
    \item A set $\Pbb$ of probability measures on $\Fcal_\infty$ is called locally dominated by $\Rtt$ if every element of $\Pbb$ is locally dominated by $\Rtt$.
    
    \item Any $\Ptt\ll_\text{loc}\Rtt$ has an associated \emph{density process}, namely the $\Rtt$-martingale $(Z_t)$ given by $Z_t:=\dd\Ptt_t/\dd\Rtt_t$. Being a nonnegative martingale, once $Z_t$ reaches zero it stays there. Thus with the convention $0/0:=1$, ratios $Z_\tau/Z_t$ are well-defined for any $t\in\NN$ and any finite stopping time $\tau\ge t$. Note that each $Z_t$ is defined up to $\Rtt$-nullsets, and therefore also up to $\Ptt$-nullsets.
    
    \item If $\Ptt\ll_\text{loc}\Rtt$ has density process $(Z_t)$, the following `Bayes formula' holds: for any $t\in\NN$, any finite stopping times $\tau$, and any nonnegative $\Fcal_\tau$-measurable random variable $Y$, one has
    \[
    \EE_\Ptt[Y\mid\Fcal_t]= \EE_\Rtt\left[\left.\frac{Z_\tau}{Z_t}Y\right|\Fcal_t\right], \text{ $\Ptt$-almost surely. }
    \]
    The right-hand side is uniquely defined $\Rtt$-almost surely (not just $\Ptt$-almost surely), and therefore provides a `canonical' version of $\EE_\Ptt [Y\mid\Fcal_t]$. We always use this version. This allows us to view such conditional expectations under $\Ptt$ as being well-defined up to $\Rtt$-nullsets.
\end{itemize}

One might ask why we work with \emph{local} domination, rather a `global' condition like $\Ptt \ll \Rtt$ for all $\Ptt$ of interest. The answer is that such a condition would be far too restrictive, as we now illustrate. 
Let $(X_t)_{t\in\NN}$ be a sequence of random variables. For each $\eta\in\mathbb R$, let $\Ptt^\eta$ be the distribution such that the $X_t$ become i.i.d.\ Gaussian with mean $\eta$ and unit variance. By the strong law of large numbers, $\Ptt^\eta$ assigns probability one to the event $A^\eta :=\{\lim_{t\to\infty}t^{-1}\sum_{s=1}^t X_s=\eta\}$. Moreover, the events $A^\eta$ are mutually disjoint: $A^\eta \cap A^\nu = \emptyset$ whenever $\eta\ne\nu$. Therefore, by definition, the measures $\{\Ptt^\eta\}_{\eta\in\mathbb R}$ are all mutually singular. Since there is an uncountable number of them, there cannot exist a measure $\Rtt$ such that $\Ptt^\eta\ll\Rtt$ for all $\eta \in \mathbb R$. 
On the other hand, if $\Ptt^\eta_t$ denotes the law of the partial sequence $X_1,\ldots,X_t$, then the measures $\{\Ptt^\eta\}_{\eta\in\mathbb R}$, are all mutually absolutely continuous. In particular, we could (for instance) use $\Rtt=\Ptt^0$ as reference measure and obtain $\Ptt^\eta_t \ll_{\text{loc}} \Rtt_t$ for all $\eta\in\mathbb R$.

\subsection{Essential supremum and infimum}
\label{sec:esssup}

We briefly review the notions of essential supremum and infimum. For more information, as well as proofs of the results below, we refer to Section~A.5 in \cite{MR2169807}.

On some probability space, consider a collection $\{Y_\alpha\}_{\alpha\in\Acal}$ of random variables, where $\Acal$ is an arbitrary index set. If $\Acal$ is uncountable, the pointwise supremum $\sup_{\alpha\in\Acal}Y_\alpha$ might not be measurable (not a random variable). Alternatively, it might happen that $Y_\alpha=0$ almost surely for every $\alpha\in\Acal$, but $\sup_{\alpha\in\Acal}Y_\alpha=1$. For this reason, the pointwise supremum is often not useful. Instead, one can use the \emph{essential supremum}.

\begin{proposition}
    There exists a $[-\infty,\infty]$-valued random variable $Y$, called the \emph{essential supremum} and denoted by $\esssup_{\alpha\in \Acal}Y_\alpha$, such that
    \begin{enumerate}
        \item $Y\ge Y_\alpha$, almost surely, for every $\alpha\in \Acal$,
        \item if $Y'$ is a random variable that satisfies $Y'\ge Y_\alpha$, almost surely, for every $\alpha\in \Acal$, then $Y'\ge Y$, almost surely.
    \end{enumerate}
    The essential supremum is almost surely unique.
\end{proposition}

In words, the essential supremum is the smallest almost sure upper bound on $\{Y_\alpha\}_{\alpha\in\Acal}$. The proposition guarantees that it always exists. In some cases, more can be said: the essential supremum can be obtained as the limit of an increasing sequence.

\begin{proposition}\label{P_esssup_closed_max}
    Suppose $\{Y_\alpha\}$ is closed under maxima, meaning that for any $\alpha,\beta\in \Acal$ there is some $\gamma\in \Acal$ such that $Y_\gamma= Y_\alpha \vee Y_\beta$. Then there is a sequence $\{\alpha_n\}_{n \in \NN}$ such that $\{Y_{\alpha_n}\}_{n \in \NN}$ is an increasing sequence and $\esssup_{\alpha\in \Acal}Y_\alpha = \lim_{n \to \infty} Y_{\alpha_n}$ almost surely.
\end{proposition}

One can also define the \emph{essential infimum} by setting
\[
\essinf_{\alpha\in\Acal}Y_\alpha := -\esssup_{\alpha\in\Acal}(-Y_\alpha).
\]
This is the largest almost sure lower bound on $\{Y_\alpha\}_{\alpha \in \Acal}$. It satisfies properties analogous to those in the propositions above.

\subsection{On the choice of filtration}\label{sec:filtration-choice}

In  the paper, we assume that the filtration $(\Fcal_t)$ in use is by default the canonical filtration $\Fcal_t :=\sigma(U,X_1,\ldots,X_t)$. However, there are examples of hypothesis tests for $H_0: \Qtt \in \Qcal$ where the only $\Qcal$-NMs with respect to $(\Fcal_t)$ are almost surely constants. For the purpose of designing more powerful tests, it may make sense to coarsen the filtration.

As a first example, consider the problem of testing if a sequence is exchangeable:
\[
H_0: X_1,X_2,\dots \text{ form an exchangeable sequence.}
\]
Vovk~\cite{vovk_testing_randomness_2019} demonstrates that all martingales with respect to $(\Fcal_t)$ (under the null) are constants, and are hence all derived tests are powerless to reject the null. Nevertheless, Vovk demonstrates that  one can derive interesting and nontrivial `conformal' martingales $(M_t)$ with respect to the restricted filtration $\Gcal_t := \sigma(M_1,\dots,M_t) \subset \Fcal_t$, that do indeed have power to reject the null (for appropriate deviations from the null). In short, \underline{coarsening the filtration} is a design tool that could aid in the construction of more powerful sequential tests, and p-process and e-processes.

In the following example, we show how the choice of including external randomization $U$ into $\Fcal_0$ also helps design better p-processes. (However, it is not always possible to randomize atoms as Example~\ref{eg:atomic-pvalue} illustrates.)

\begin{example}
Assume that $\Qcal = \{\Qtt\}$, where under $\Qtt$ we have that $X_1$ is Bernoulli($1/2$) and $X_2, X_3, \ldots = 0$. Consider the canonical filtration $(\Gcal_t)$, so that $\Gcal_\infty = \sigma(X_1)$. 
Then $(M_t)$ with $M_0 = 1$, $M_t = 2 X_1$ for all $t \in \NN$ is a $\Qtt$-NM and the corresponding p-process $(\inf_{s \leq t} 1/{M_s})$ is admissible. (Indeed, any p-process $(\Pval_t)$ has to satisfy $\Qtt(\Pval_1 \leq \alpha) \in \{0,1/2,1\}$ for each $\alpha \in [0,1]$). However, by expanding the filtration using external randomization, one can easily derive a strictly smaller p-process $(\Pval_t')$ such that $\Pval_1'$ is uniform. In other words, the original p-process is only admissible under the filtration generated solely by the observations, but is inadmissible under an expanded filtration that includes a randomization device (which is the filtration $(\Fcal_t)$ in this paper).
\end{example}

The following randomization device can be found in \cite[Lemma 2]{brockwell2007universal}, but we provide a simpler proof. We remark briefly that we imagine this device being used when $Y$ is a valid p-value with a known discrete distribution, so that it recovers a smaller, but still valid, p-value. 

\begin{lemma}[Randomization device] \label{lem:rand_device}
    If $Y$ is a random variable with distribution function $F$ and $U$ is an independent uniformly distributed random variable, then 
    \[
       Y' :=  U F(Y) + (1-U) F(Y-)
    \]  
    is uniformly distributed and satisfies $Y' \le F(Y)$. 
    %Further, if $Y$ is stochastically larger than uniform, then $Y' \leq Y$.
\end{lemma}
% \com{is there a variant where $Y' \leq 1-F^-(Y)$?}

\begin{proof}
    Fix $a \in [0,1]$ and define $y := \inf\{x \in \mathbb R: F(x) \geq a\}$. Note that
    \begin{align*}
    \Pr(Y' \leq a| Y) &= \Pr\left(\left. U \le \frac{a - F(y-)}{F(y) - F(y-)} \right| Y\right)\1_{Y=y} \\
    & \quad + \Pr( U F(Y) + (1-U) F(Y-) \leq a ~|~ Y)\1_{Y\ne y}.
    \end{align*}
    (If $\Pr(Y=y)=0$, the first term should be understood as zero.) On $\{Y>y\}$ we have $a < F(Y-)$, so that the second term equals zero. On $\{Y<y\}$ we have $F(Y) \le a$, so that the second term instead equals one. Since also $U$ is uniform and independent of $Y$, we get
    \[
    \Pr(Y' \leq a| Y) = \frac{a - F(y-)}{F(y) - F(y-)} \1_{Y=y} + \1_{Y<y}.
    \]
    Taking expectations and simplifying gives $\Pr(Y' \le a) = a$, showing that $Y'$ is uniformly distributed. Finally, it is clear from the definition of $Y'$ that $Y' \le F(Y)$.
%        Then 
%        \[
%            \Pr(Y' \leq a) = \EE\left[\Pr(U F(Y) + (1-U) F(Y-) \leq a| Y) \right] 
%            =   \EE\left[\1_{Y < y} + \1_{Y = y} \frac{a - F(y-)}{F(y) - F(y-)}\right] = a, 
%        \]
%    yielding the statement.
\end{proof}

\section{Omitted proofs}
\label{sec:proofs}

\begin{proof}[Proof of Lemma~\ref{lem:equiv_uniform_defns}]
% \label{prf:lem-equiv_uniform_defns}
It is clear that \ref{L1:2} $\implies$ \ref{L1:3}.
The implication \ref{L1:1} $\implies$ \ref{L1:2} follows from
\begin{align*}
  A_T
  &= \left(\bigcup_{t\in \NN} (A_t \cap \{T = t\}) \right)
    \cup (A_\infty \cap \{T = \infty\})
    \subseteq \bigcup_{t \in \NN} A_t.
\end{align*}
For \ref{L1:3} $\implies$ \ref{L1:1}, take
$\tau := \inf\{t \in \mathbb{N}: A_t \text{ occurs}\}$, so that
$A_\tau = \bigcup_{t \in \NN} A_t$. 
\end{proof}

\bigskip

\begin{proof}[Proof of Lemma~\ref{lem:non-equiv}]
 First, \ref{L12:1} implies \ref{L12:2} since $N_T \leq \sup_{t \in \NN_0} N_t$, hence $\EE[N_T] \leq \EE[\sup_{t \in \NN_0} N_t]$, for all random times $T$. Conversely, for any $\varepsilon>0$ there exists some random time $T$ such that $N_T \geq \sup_{t \in \NN_0} N_t -\varepsilon$. Thus if \ref{L12:2} holds, then $\EE[\sup_{t \in \NN_0} N_t] \le \EE[N_T] + \varepsilon \le 1+\varepsilon$. Since $\varepsilon>0$ was arbitrary, we find that \ref{L12:2} implies \ref{L12:1}. It is clear that \ref{L12:2} implies \ref{L12:3}. 
 
 The fact that \ref{L12:3} implies \ref{L12:4} is however not obvious, and is essentially a consequence of a result by~\citet[Theorem 3]{shafer_test_2011}, as also noted recently by~\citet[Lemma 5.1]{kirichenko2020minimax}. 
 First, we note that if $(N_t)$ satisfies \ref{L12:3},
then   $\Pval_t := 1 \wedge \inf_{s \leq t} 1/N_s$ is a p-process (see also Proposition~\ref{prop:construct-p-vals}\ref{prop:construct-p-vals:1}).
 In particular, $\Pval_\infty :=  \inf_{t \in \NN_0} \Pval_t$ stochastically dominates a uniform. Therefore, for any nonnegative, nonincreasing function $f(u)$ such that $\int_0^1 f(u)\dd u = 1$, we have $\EE[f(\Pval_\infty)] \le \int_0^1 f(u)\dd u = 1$ (see also the proof of Proposition~\ref{prop:construct-e-vals}). The function $f(u) := g(1/u)$, with $g$ as in the lemma, satisfies this condition. Consequently, $\EE[g(1\vee \sup_{s\in\NN_0} N_s)] = \EE[f(\Pval_\infty)] \le 1$, as required.
 \end{proof}
 
\bigskip

\begin{proof}[Proof of Proposition~\ref{prop:NM-likelihood}]
For each $t\in\NN_0$, define a probability measure $\Ptt'_t$ on the Borel sets of $\RR^t$ by $\Ptt'_t(A) := \EE_\Qtt[M_t \1_A]$. Because $(M_t)$ is a martingale under $\Qtt$, the sequence $(\Ptt'_t)_{t\in\NN}$ forms a consistent system of finite-dimensional distributions. Therefore, by Kolmogorov's extension theorem, there exists a single probability measure $\Ptt$ on the Borel sets of $\Omega = \RR^\NN$ whose projection onto $\RR^t$ is exactly $\Ptt'_t$ for each $t\in\NN$. Put differently, $\Ptt$ satisfies $\Ptt_t = \Ptt'_t$ for all $t\in\NN$, as desired.
\end{proof}

 \bigskip

\begin{proof}[Proof of Proposition~\ref{P:200905}]
We prove the statement for p-processes; the same argument holds for e-processes and sequential tests.
The proof is based on transfinite induction. Fix some p-process $(\Pval_t)$. For all countable ordinals $\beta$, we now recursively define p-processes $(\Pval_t^\beta)$ as follows. For $\beta=1$, we set $(\Pval^\beta_t) := (\Pval_t)$. For any successor ordinal $\gamma := \beta + 1$, if $(\Pval^\beta_t)$ is $\Qcal$-admissible we set $\Pval^\gamma_t := \Pval^\beta_t$, and otherwise we let $(\Pval^\gamma_t)$ be any p-process that strictly dominates $(\Pval^\beta_t)$.  For any limit ordinal $\gamma := \lim_{n \to \infty} \beta_n$, we define $(\Pval^\gamma_t) := (\lim_{n \to \infty} \Pval^{\beta_n}_t)$. Let us now use the induction assumption that $(\Pval^\beta_t)$ is a $\Qcal$-p-process for all $\beta < \gamma$, for this limit ordinal $\gamma$.  Since $(\lim_{n \to \infty} \Pval^{\beta_n}_t)$ is a decreasing limit, we have for every $\varepsilon>0$,  $\alpha\in[0,1]$, and $\Qtt \in \Qcal$ that
\[
\Qtt\left( \inf_{t\in\NN_0} \Pval^\gamma_t \le \alpha \right) \le \Qtt\left( \lim_{n \to \infty} \inf_{t\in\NN_0} \Pval^{\beta_n}_t < \alpha + \varepsilon \right)
= \lim_{n \to \infty} \Qtt\left( \inf_{t\in\NN_0} \Pval^{\beta_n}_t < \alpha + \varepsilon \right) \le \alpha + \varepsilon.
\]
It follows that $(\Pval^\gamma_t)$ is a $\Qcal$-p-process. By transfinite induction, this holds for all countable ordinals $\beta$.

Writing $\Rtt$ for the reference probability measure, $\{\EE_\Rtt[\sum_{t\in\NN_0} 2^{-t} \Pval^\beta_t]\}_{\beta}$ defines a decreasing $[0,2]$-valued transfinite sequence. This sequence must eventually become stationary, that is, it becomes constant for all $\beta$ beyond some countable ordinal $\beta_0$. Thus $\Pval^\beta_t = \Pval^{\beta_0}_t$ for all $\beta \ge \beta_0$ and all $t\in\NN_0$. By construction, $(\Pval^{\beta_0}_t)$ must then be admissible and dominate $(\Pval_t)$. This shows that any p-process for $\Qcal$ can be dominated by a $\Qcal$-admissible  p-process.

 Let us also remark that in the case of $\Qcal$ being a singleton the statement for e-processes and sequential tests could be proved in a more constructive manner as in Subsections~\ref{SS:6.2} and \ref{SS:6.3}.
\end{proof}

\bigskip

\begin{proof}[Proof of Proposition~\ref{prop:construct-seq-tests}]
We prove the three statements in order. Let $(\psi_t)$ denote the constructed binary sequence, which we will now show is a $(\Qcal,\alpha)$-sequential tets.  Let $\tau$ denote an arbitrary stopping time, potentially infinite, and fix $\Qtt \in \Qcal$.
\begin{enumerate}[label={\rm(\arabic{*})}, ref={\rm(\arabic{*})}]
    \item $\Qtt(\psi_\tau = 1) = \Qtt(\Pval_\tau \leq \alpha) \leq \alpha$ since $(\Pval_t)$ is a $\Qtt$-p-process.
    \item $\Qtt(\psi_\tau = 1) = \Qtt(\Eval_\tau \geq 1/\alpha) \leq \alpha \EE_{\Qtt}[\Eval_\tau]  \leq \alpha$, where we used Markov's inequality and the fact that $(\Eval_t)$ is a $\Qtt$-e-process. In short, e-processes satisfy Ville's inequality.
    \item $\Qtt(\psi_\tau = 1) = \Qtt(\phi(\Qcal) \cap  \CI_\tau = \emptyset) \leq \Qtt(\phi(\Qtt) \notin  \CI_\tau) \leq \alpha$, where the first inequality follows because the event $\{\phi(\Qcal) \cap \CI_\tau = \emptyset\}$ implies that $\CI_\tau$ does not contain $\phi(\Qtt)$, which is improbable under $\Qtt$.
\end{enumerate}
The fact that the $(\Qcal, \alpha)$-sequential tests in \ref{prop:construct-seq-tests:1} and \ref{prop:construct-seq-tests:2} are nested is obvious. This completes the proof.
\end{proof}

\bigskip

\begin{proof}[Proof of Proposition~\ref{prop:construct-p-vals}]
We prove the three statements in order. Let $(\Pval_t)$ denote the constructed sequence of random variables, which we will now show in each case is a p-process.
Let $\tau$ denote an arbitrary stopping time, potentially infinite, and fix $\Qtt \in \Qcal$.
\begin{enumerate}[label={\rm(\arabic{*})}, ref={\rm(\arabic{*})}]
    \item $\Qtt(1/\Eval_\tau \leq \alpha) = \Qtt(\Eval_\tau \geq 1/\alpha)  \leq \EE_{\Qtt}[\Eval_\tau]\cdot \alpha \leq \alpha$, where we used Markov's inequality and the fact that $(\Eval_t)$ is a $\Qtt$-e-process. Since a p-process remains valid after taking the running infimum, we obtain $(\Pval_t)$ is valid.
    % \item $\Qtt^*(\Pval_\tau \leq \alpha) = \Qtt^*(\CI_\tau(\alpha) \cap \Qcal =\emptyset) \leq \Qtt^*(\Qtt^* \notin \CI_\tau(\alpha)) \leq \alpha$, where the first inequality follows by logical implication and the second because $(\CI_t(\alpha))$ is a $(\Pcal,\alpha)$-CS.
    \item\label{item:200809} $\Qtt(\Pval_\tau > \alpha) = \Qtt(\psi_\tau(\alpha)= 0) \geq 1-\alpha$,     where the equality follows since the sequential tests are nested and the inequality because $(\psi_t(\alpha))$ is a $(\Qcal,\alpha)$-sequential test.
    \item $\Qtt(\Pval_\tau > \alpha) = \Qtt(\phi(\Qcal) \cap \CI_\tau(\alpha) \neq \emptyset ) \geq \Qtt(\phi(\Qtt) \in \CI_\tau(\alpha) ) > 1-\alpha$ as in \ref{item:200809}. 
\end{enumerate}
This completes the proof.
\end{proof}

\bigskip

\begin{proof}[Proof of Proposition~\ref{prop:construct-CS}]
Let $(\CI_t)$ denote the constructed sequence of sets, which we will now show  is a $(\phi, \Pcal,\alpha)$-confidence sequence. 
To this end, let $\tau$ denote an arbitrary stopping time, potentially infinite, and fix $\Ptt \in \Pcal$; note that $\Ptt \in \Pcal^{\gamma}$ for some $\gamma \in \Zcal$. 
Then, we have $\Ptt(\phi(\Ptt) \notin \CI_\tau) = \Ptt(\gamma \notin \CI_\tau) = \Ptt(\psi_\tau^{\gamma} = 1 ) \leq \alpha$ since $(\psi^{\gamma}_t)$ is a $(\Pcal^\gamma,\alpha)$-sequential test.
This completes the proof.
\end{proof}

\bigskip

\begin{proof}[Proof of Proposition~\ref{prop:construct-e-vals}]
Define $\Eval_t := f(\Pval_t)$; we must show that $(\Eval_t)$ is a $\Qcal$-e-process.
If $(\Pval_t)$ is a $\Qcal$-p-process then for any stopping time $\tau$ and $\Qtt \in \Qcal$, the distribution of $\Pval_\tau$ is stochastically larger than a uniform random variable (denoted $V$). Thus for any calibrator $f$, we have $\EE_\Qtt[\Eval_\tau] = \EE_{\Qtt}[f(\Pval_\tau)] \leq \EE[f(V)] = \int_0^1 f(v) \dd v =1$. Since this result holds for any $\tau$ and $\Qtt \in \Qcal$, the result follows.
\end{proof}

\bigskip

\begin{proof}[Proof of Proposition~\ref{P:convex_hull}]
To see \ref{P:convex_hull:1}, fix a probability measure $\Qtt \in \rm{conv}(\Qcal)$. Then there exist $\Qtt_1, \Qtt_2 \in \Qcal$ and $\lambda \in [0,1]$ such that $\Qtt = \lambda \Qtt_1 + (1-\lambda) \Qtt_2$. 
Let now $(\Eval_t)$ denote an e-process for $\Qcal$. Consider some stopping time $\tau$ and note that
\[
    \EE_{\Qtt}[\Eval_\tau] = \lambda\EE_{\Qtt_1}[\Eval_\tau]
        + (1-\lambda) \EE_{\Qtt_2}[\Eval_\tau] \leq \lambda + 1-\lambda = 1,
\]
since $(\Eval_t)$ is a $\Qcal$-e-process. This yields that $(\Eval_t)$ is also a $\rm{conv}(\Qcal)$-e-process. The same argument also applies for valid p-processes and sequential tests. 

Next, \ref{P:convex_hull:2}  follows in a similar way. Assume  that $(\Eval_t)$ is $\Qcal$-admissible and consider some $\rm{conv}(\Qcal)$-e-process $(\Eval_t')$ that satisfies $\Qtt(\Eval_t' \geq \Eval_t) = 1$ for all $t \in \NN$ and $\Qtt \in \rm{conv}(\Qcal)$ and there exists some $\Qtt^* \in \rm{conv}(\Qcal)$ and some $t \in \NN$ such that $\Qtt^*(\Eval_t' > \Eval_t) > 0$. Since we always have $\Qtt^* = \lambda \Qtt_1 + (1-\lambda) \Qtt_2$ for some  $\Qtt_1, \Qtt_2 \in \Qcal$ and $\lambda \in [0,1]$ we also have $\Qtt_1(\Eval_t' > \Eval_t) > 0$ or $\Qtt_2(\Eval_t' > \Eval_t) > 0$, leading to a contradiction. Again, the same argument also applies for admissibile p-processes and sequential tests. 
\end{proof}

 \bigskip
 
 \begin{proof}[Proof of Lemma~\ref{lem:anticoncentration-pointwise}]
 Fix some $\alpha \in (0,1]$, let $\tau$ denote the first time $t$ that $M_t \geq 1/\alpha$, and let  
    \[
        q := \Qtt(\tau < \infty) = \Qtt\left(\sup_{t\in \NN_0} M_t \geq \frac{1}{\alpha}\right).
    \]
    Next, \eqref{eq:emp-var} yields $\Qtt(M_\infty = 0) = 1$, for example, by \cite[Theorem~4.2]{Larsson:Ruf:convergence}.
    Note that the stopped process $M^\tau$ is a  uniformly integrable martingale, yielding $\EE_{\Qtt}[M^\tau_\infty] = 1$. On the event $\{\tau = \infty\}$, we have $M^\tau_\infty = 0$.  With $M_{-1} := 1$, $Y_0 := 0$, and $\mathcal F_{-1} := \{\emptyset, \Omega\}$, these observations then yield 
    % \com{adapt notation}
    \begin{align*}
        1 &= \EE_{\Qtt}\left[M^\tau_\infty\right] 
            = \sum_{t \in \NN_0} \EE_{\Qtt}\left[M_t \1_{\tau = t}\right] 
            = \sum_{t \in \NN_0} \EE_{\Qtt}\left[M_{t-1} Y_t \1_{\tau = t}\right]
            = \sum_{t \in \NN_0} \EE_{\Qtt}\left[\EE_{\Qtt}\left[ Y_t \left| \mathcal{F}_{t-1}, Y_t \geq \frac{1}{\alpha M_{t-1}} \right.\right]
            M_{t-1} \1_{\tau = t}\right]
            \\
            &\leq \frac{1}{\alpha} (1+\varepsilon) \sum_{t \in \NN_0} \EE_{\Qtt}\left[\1_{\tau = t}\right]  = q \frac{1+\varepsilon}{\alpha}.
    \end{align*}
    This then gives $q \geq \alpha/(1+\varepsilon)$, yielding
     the claim.
\end{proof}

\bigskip

\begin{proof}[Proof of Proposition~\ref{prop:symmetry}]
    First, assume $(M_t)$ is an $\Scal$-NM and fix a time $t$. Since $(Y_t)$ is adapted, $Y_t$ is a function of $U, X_1, \ldots, X_t$. Hence we may write $Y_t = f_t(X_t)$ for some nonnegative predictable function $f_t(\cdot)$. More explicitly, $Y_t=f_t(U, X_1,\ldots,X_{t-1};X_t)$. Now pick any real numbers $x_1,\ldots,x_t$. Consider the two-point measures $\mu_s := (\delta_{-x_s} + \delta_{x_s})/2$ for all $s\le t$ and let $\Ptt := \Utt \times \prod_{s\in\NN}\mu_{s \wedge t}$ be the distribution that makes the data independent with $X_s \sim \mu_{s \wedge t}$. Then $\Ptt \in \Scal$. Moreover,
    \[
    1 = \EE_{\Ptt_x}[Y_t|\mathcal F_{t-1}] = \frac{1}{2} \left( f_t(U, X_1,\ldots,X_{t-1}; x_t) + f_t(U, X_1,\ldots,X_{t-1}; - x_t) \right).
    \]
    Since the event $\{X_i = x_i,\ i=1,\ldots,t-1\}$ has positive probability, we get
    \[
    \frac{1}{2} \left( f_t(U, x_1,\ldots,x_{t-1}; x_t) + f_t(U, x_1,\ldots,x_{t-1}; - x_t) \right) = 1.
    \]
    But the numbers $x_1,\ldots,x_t$ were arbitrary, so it follows that the function $x\mapsto f_t(U, x_1,\ldots,x_{t-1}; x) - 1$ is odd for all $x_1,\ldots,x_{t-1}$.

    For the reverse direction fix some $\Ptt \in \Scal$ and some $t \in \NN$. Then 
    \begin{align*}
        \EE_\Ptt[Y_t|\mathcal F_{t-1}] &= 
         \EE_\Ptt[f_t(X_t)|\mathcal F_{t-1}] \\
            &\stackrel{(i)}{=} \frac{1}{2}
            \left(\EE_\Ptt[f_t(X_t)|\mathcal F_{t-1}] + \EE_\Ptt[f_t(-X_t)|\mathcal F_{t-1}] \right) \\
            &= 1+ \frac{1}{2}
            \left(\EE_\Ptt[f_t(X_t)-1|\mathcal F_{t-1}] + \EE_\Ptt[f_t(-X_t)-1|\mathcal F_{t-1}] \right)\\
            &\stackrel{(ii)}{=} 1,
    \end{align*}
    yielding the statement. Above, equality $(i)$ follows by symmetry of $\Ptt$, and equality $(ii)$ follows because $f_t-1$ is odd. 
\end{proof}

\bigskip

\begin{proof}[Proof of Proposition~\ref{prop:admissible-pvalue}]
    This follows from an application of Corollary~\ref{C:200724}. Fix some $\varepsilon > 0$. Since $f$ is continuous at zero and $f(0) = 1$ there exists some $\eta > 0$ such that $f(x) \leq 1+\varepsilon$ for all $x \in (-2\eta, 2 \eta)$. This implies \eqref{eq:200802}. Moreover, since $f$ is strictly monotone at zero we may assume that $f(\eta) \neq 1$. 
    
    Consider now the measure $\mu_\eta := (\delta_{-\eta} + \delta_\eta)/2$ and note that $\Qtt_\eta := \Utt \times  \mu_\eta^\infty \in \Scal$. Hence $\Qtt_\eta(f(X_t) \leq 1 + \varepsilon) = 1$ for all $t \in \mathbb{N}$. Moreover,
    \[
        \Qtt_\eta\left(\sum_{t \in \NN} (f(X_t) - 1)^2 = \infty \right) = 
            \Qtt_\eta\left(\sum_{t \in \NN} (f(\eta) - 1)^2 = \infty \right) = 1,
    \]
    since $f-1$ is odd and $f(\eta) \neq 1$. This shows that \eqref{eq:emp-var} holds.
    Hence Corollary~\ref{C:200724} can indeed be applied and the statement follows.
\end{proof}

\section{Auxiliary examples}
\label{sec:more-examples}

The following example shows that Proposition~\ref{P:convex_hull} cannot be extended to confidence sequences.
\begin{example}[Confidence sequences do not mesh with convex closures]\label{eg:conv-closure-CS}
    Let $\mu_+$ ($\mu_-$) denote the law of a Gaussian random variable with unit variance and mean $1$ ($-1$). Moreover, let $\Pcal = \{\mu^\infty_+, \mu^\infty_-\}$ be the family of i.i.d.~laws of such distributions. 
    Consider $\phi^{\rm mean}$, which satisfies $\phi^{\rm mean}(\mu^\infty_-) = -1$ and $\phi^{\rm mean}(\mu^\infty_+) = 1$.
    Then $(\CI_t)$ given by $\CI_t=\{-1,+1\}$
    % $\CI_t = ( \bar X_t -1.96/\sqrt{t},  \bar X_t +1.96/\sqrt{t})$, where $\bar X_t = \sum_{s \leq t} X_s / t$ 
    is a (trivial) $(\phi^{\rm mean}, \Pcal, \alpha)$-CS for $\alpha \in [0,1]$. 
    Now consider the measure $\Ptt = (\mu^\infty_+ + \mu^\infty_-)/2 \in \rm{conv}(\Pcal)$, which satisfies $\phi^{\rm mean}(\Ptt) = 0 \notin \CI_t$. 
    It is clear that $(\CI_t)$ is not a $(\phi^{\rm mean}, \rm{conv}(\Pcal), \alpha)$-CS for any $\alpha\in[0,1]$. 
\end{example}

\bigskip

The next example also elaborates further on the discussion in Subsection~\ref{sec:sufficient-p} by providing an admissible p-process that has an atomic limiting distribution.

\begin{example}[Atomic admissible p-processes exist even in the presence of an independent $\Fcal_0$-measurable random device]\label{eg:atomic-pvalue}
    Consider $\Qtt$ under which $(X_t)$ are i.i.d.~uniformly distributed. Then an adapted process $(\Pval_t)$ with the following properties can be constructed.
    \begin{itemize}
        \item $\Pval_t$ is supported on $\{1/2 + k/2^{t+1}\}_{k = 1, \ldots 2^t}$;
%        \item Define $\Pval'_t = 2(\Pval_t - 1/2)$. OR $\Pval'_\infty = 2(\Pval_\infty - 1/2)$. OR $\Pval'_\infty = U/2 + (\Pval_\infty - 1/2)$ [false] OR $\Pval'_\infty = U/2 1_{\Pval_\infty = 1/2} + \Pval_\infty 1_{\Pval_\infty > 1/2}$.
        \item $\Qtt(\Pval_t = 1/2 + 1/2^{t+1}) = 1/2 +  1/2^{t+1}$ and $\Qtt(\Pval_t = 1/2 + k/2^{t+1}) = 1/2^{t+1}$ for all $k = 2, \ldots 2^t$;
        \item $(\Pval_t)$ is a $\Qtt$-MM with 
        \[
            \Qtt\left(\left.\Pval_{t+1} - \frac{1}{2} \in \left\{\frac{2k-1}{2^{t+2}}, \frac{k}{2^{t+1}}\right\}\right| \Pval_t - \frac{1}{2} = \frac{k}{2^{t+1}} \right) = 1
        \]
        for all $k = 1, \ldots 2^{t}$ and $t \in \NN$;
        \item $(\Pval_t)$ is independent of $U$.
    \end{itemize}
    Note that $\Pval_\infty := \inf_{t \in \NN} \Pval_t$ satisfies $\Qtt(\Pval_\infty \leq \alpha) = \alpha \1_{\alpha \geq 1/2} 
    \leq \alpha$ for all $\alpha \in [0,1]$; in particular $(\Pval_t)$ is a p-process and its limit $\Pval_\infty$ has an atom at $1/2$.

    We claim that $(\Pval_t)$ is $\Qtt$-admissible. Indeed, assume there exists a p-process $(\Pval_t')$ that dominates $(\Pval_t)$ (we explicitly allow $(\Pval_t')$ to depend on the randomization device $U$). Then there exists some $t \in \NN$ such that $\Qtt(\Pval_t' < \Pval_t) > 0$. Let us first assume that $\Qtt(\Pval_t' < 1/2 + 1/2^{t+1}) > 0$. In combination with the fact that $(\Pval_t)$ is a $\Qtt$-MM there exists some $n > t + 2$ such that 
    \[
        \Qtt\left(
        \left\{\Pval_t' \leq \frac{1}{2} + \frac{1}{2^{t+1}}\right\}
        \cap \left\{\Pval_\infty \geq \frac{1}{2} + \frac{1}{2^{t+2}}\right\}
        \right) > 0. 
    \]
    Since $\Pval_\infty' \leq \Pval_\infty$ and since
    $\Qtt(\Pval_\infty \geq 1/2 + {1}/{2^{t+2}}) = 1/2 - {1}/{2^{t+2}}$, we hence obtain $\Qtt(\Pval_\infty' \geq 1/2 + {1}/{2^{t+2}}) < 1/2 - {1}/{2^{t+2}}$, a contradiction to the fact that  $(\Pval_t')$ is a p-process.  We obtain similar contradictions if we assume $\Qtt( \{\Pval_t' < \Pval_t\} \cap \{\Pval_t = 1/2 + k/2^{t+1}\}) > 0$ for some $k = 2, \ldots, 2^t$. This shows that $(\Pval_t)$ is indeed $\Qtt$-admissible, despite having an atom and being independent of the randomization device.
\end{example}

\bigskip

The next example illustrates how anti-concentration bounds can be satisfied by NMs that lead to inadmissible p-processes.
\begin{example}[A p-process for $\Scal$ that satisfies  Proposition~\ref{prop:admissible-pvalue} need not be admissible]\label{ex:sym_counter}
    Fix the function $f: x \mapsto ((1+x) \wedge 2)^+$. This function satisfies the criteria of Proposition~\ref{prop:admissible-pvalue}. Hence the process $M_t := \prod_{s \leq t} f(X_s)$ is an $\Scal$-martingale
    and we also have \eqref{eq:200803}.
    Consider the p-process $(\Pval_t)$ given by $\Pval_t :=  \inf_{s\leq t} 1/M_s$. This is a p-process by \eqref{eq:200803}.
    
    Define next
    $\Pval_t' := \Pval_t$ for $t = 0,1$ and 
    $\Pval_t' := \Pval_t - 1/4\, \1_{X_1 \vee X_2 \leq -1}$
    for $t \geq 2$.
    Clearly we have $\Qtt(\Pval_t' \leq \Pval_t)=1$ for all $t \in \NN_0$ and $\Qtt \in \Scal$ and there exists  $\Qtt^* \in  \Scal$ such that $\Qtt^*(\Pval_2' < \Pval_2)>0$.  We now claim that $(\Pval'_t)$ is also an $\Scal$-p-process. We will prove this claim below.  This  assertion then yields that \eqref{eq:200803} is \emph{not} sufficient for the admissibility of the corresponding p-process in general.%, even if the null is fork-convex.
    
    We now prove the claim that $(\Pval'_t)$ is an $\Scal$-p-process.
    To do so we consider a subset  $\widetilde \Scal \subset \Scal$,
    namely those measures $\Qtt \in \Scal$ that satisfy $\Qtt(X_1 \leq -1) \in \{0, 1/2\}$.  Note that $\Scal = \text{conv}(\widetilde \Scal)$, the convex hull of $\widetilde \Scal$. Thanks to Proposition~\ref{P:convex_hull}\ref{P:convex_hull:1} it suffices to argue that $(\Pval'_t)$ is an $\widetilde \Scal$-p-process. 
    To this end, note that on the event $\{X_1 > -1\}$ we have $\Pval_t = \Pval'_t$ for all $t \in \NN_0$. On the other hand, on the event $\{X_1 \leq -1\}$ we have $\Pval_t = 1$ for all $t \in \NN_0$. 
 
    Fix now some $\alpha \in (0,1)$ and $\Qtt \in  \widetilde\Scal$. Without loss of generality we can assume that $\Qtt(X_1 \leq -1) = 1/2$, otherwise there is nothing to be argued.
    We now need to show that 
    \begin{align} \label{eq:200803.2}
       \Qtt(\Pval'_\infty \leq \alpha) \leq \alpha. 
    \end{align}
    To make headway, note that  
    \begin{align*}
        \{\Pval_\infty' \leq \alpha\}
            &= \left( \{\Pval_\infty' \leq \alpha\} 
                \cap \{X_1 > -1\}\right)
            \cup 
            \left( \{\Pval_\infty' \leq \alpha\} 
                \cap \{X_1 \leq -1\}\right)\\
            &\subset 
            \{\Pval_\infty \leq \alpha\} \cup
            \left\{\Pval_\infty - \frac{1}{4} \1_{X_1 \vee X_2 \leq -1} \leq \alpha \right\}\\
            &= 
            \{\Pval_\infty \leq \alpha\} \cup
            \left\{\frac{1}{4} \1_{X_1 \vee X_2 \leq -1} \geq 1 - \alpha \right\}.
    \end{align*}
    Thus, if $\alpha < 3/4$ then $ \{\Pval_\infty' \leq \alpha\} \subset \{\Pval_\infty \leq \alpha\}$
    and we have \eqref{eq:200803.2}. Let us now assume that $\alpha \geq 3/4$. Note that
    since $\Qtt(\Pval_\infty = 1) \geq \Qtt(X_1 \leq -1) = 1/2$ and hence $\Qtt(\Pval_\infty \leq \alpha) < 1/2$. This then yields
    \[
        \Qtt(\Pval_\infty' \leq \alpha) \leq
        \Qtt(\Pval_\infty \leq \alpha) + \Qtt(X_1 \vee X_2 \leq -1)
        < 1/2 + 1/2^2 = 3/4 \leq \alpha,
    \]
    yielding that $(\Pval_t')$ is an $\widetilde S$-p-process, hence also an $S$-p-process.
\end{example}

%%%%PLEASE REMOVE!
% \input{attic}

\end{document}